\documentclass[12pt,reqno]{article}
\usepackage{fullpage}
\usepackage{comment}
\usepackage{enumerate}
\usepackage{thmtools}
\usepackage{mathtools}
\usepackage[normalem]{ulem}
\usepackage{amssymb}
\usepackage{amsthm}
\usepackage[normalem]{ulem}
\usepackage[mathscr]{euscript}
\usepackage{pgf,tikz}
\usetikzlibrary{arrows}
\usepackage{float}
\usepackage{hyperref}
\usepackage{cite}
\usepackage{color}
\usepackage{xcolor}
\usepackage{transparent}
\usepackage{subfigure}
\usepackage{upgreek} % for $\upnu$, to distinguish $\nu$ from $v$.
\usepackage{amsfonts,txfonts}
\usepackage{bm}
\usepackage{epstopdf}
\usepackage{psfrag}
\definecolor{azure(colorwheel)}{rgb}{0.0, 0.5, 1.0}
\definecolor{hanpurple}{rgb}{0.32, 0.09, 0.98}
\definecolor{iris}{rgb}{0.35, 0.31, 0.81}
\definecolor{byzantine}{rgb}{0.74, 0.2, 0.64}
\definecolor{ashgrey}{rgb}{0.7, 0.75, 0.71}
\definecolor{battleshipgrey}{rgb}{0.52, 0.52, 0.51}
\hypersetup{colorlinks=true,pdfborder=000,  citecolor  = blue, citebordercolor= {magenta}}
\hypersetup{linkcolor=black}

\makeatletter
\let\reftagform@=\tagform@
\def\tagform@#1{\maketag@@@{(\ignorespaces\textcolor{purple}{#1}\unskip\@@italiccorr)}}
\renewcommand{\eqref}[1]{\textup{\reftagform@{\ref{#1}}}}
\makeatother
\makeatletter
\DeclareUrlCommand\ULurl@@{%
  \def\UrlLeft{\uline\bgroup}%
  \def\UrlRight{\egroup}}
\def\ULurl@#1{\hyper@linkurl{\ULurl@@{#1}}{#1}}
\DeclareRobustCommand*\ULurl{\hyper@normalise\ULurl@}
\makeatother

\def\lessim{\ \lower4pt\hbox{$
		\buildrel{\displaystyle <}\over\sim$}\ }
\def\gessim{\ \lower4pt\hbox{$\buildrel{\displaystyle >}
		\over\sim$}\ }

\numberwithin{equation}{section}

\newtheorem{thm}{Theorem}[section]    
\newtheorem{lem}[thm]{Lemma}         
\newtheorem{prop}[thm]{Proposition}        
\newtheorem{coro}[thm]{Corollary}
          
\theoremstyle{definition}

\newtheorem{rmk}[thm]{Remark}

\newcommand{\N}{\mathbb{N}}

\newcommand{\R}{\mathbb{R}}
\newcommand{\C}{\mathbb{C}}
\newcommand{\bbS}{\mathbb{S}}

\newcommand{\B}{\textbf}
\newcommand{\non}{\nonumber}

\newcommand{\1}{\B{\rm \B{1}}}
\newcommand{\jc}{\,,\,}
\newcommand{\jdots}{\jc \dots \jc}
\newcommand{\jd}{\,\cdot\,}
\newcommand{\jddots}{\jd \dots \jd}
\newcommand{\leftp}{\left(}
\newcommand{\rightp}{\right)}
\newcommand{\leftb}{\left[}
\newcommand{\rightb}{\right]}
\newcommand{\leftcb}{\left\{}
\newcommand{\rightcb}{\right\}} 
\newcommand{\tc}{\,:\,}
\newcommand{\ands}{\text{ and }}
\newcommand{\eq}[1]{\begin{align*}#1\end{align*}} % quickly display w/ no label

\newcommand{\til}{\tilde} 
\newcommand{\wt}{\widetilde}

\newcommand{\un}{\underline}
\newcommand{\tri}{\triangleq}

\newcommand{\oith}{\textsuperscript{\rm \hspace{.4mm}th}}

\newcommand{\eps}{\epsilon}

\newcommand{\vp}{\varphi}
\newcommand{\sig}{\sigma}

\newcommand{\pa}{\partial}

\newcommand{\ra}{\rangle}
\newcommand{\la}{\langle}

\newcommand{\sgn}{\text{\rm sgn}}

\newcommand{\diag}{{\rm{diag}}}

\newcommand{\E}{\mathbb{E}}

\newcommand{\prob}{\mathbb{P}}

\newcommand{\rmc}{\text{ {\rm c}} }
\newcommand{\crit}{\text{\rm Crt}}
\newcommand{\Crit}{\text{\rm Crt}}
\newcommand{\critNell}{\Crit_{N,\,\ell}}
\newcommand{\CritNell}{\Crit_{N,\,\ell}}
\newcommand{\HNp}{H_{N,\,p}}
\newcommand{\goe}{\text{\rm GOE}}

\newcommand{\ind}{\text{\rm ind}}

\newcommand{\semi}{\mu_{\text{\rm sc}}}

\newcommand{\iid}{\text{\rm iid}}
\newcommand{\pair}{\text{\rm pair}}
\newcommand{\ang}{\text{\rm ang}}

\newcommand{\sfr}{{\text{\rm\textsf{r}}}}

\newcommand{\sfx}{{\text{\rm\small\textsf{x}}}}
\newcommand{\sfy}{{\text{\rm\small\textsf{y}}}}

\newcommand{\sfC}{{\text{\rm\small\textsf{C}}}}
\newcommand{\sfE}{{\text{\rm\small\textsf{E}}}}
\newcommand{\matM}{\B{\rm \B{M}}}
\newcommand{\matG}{\B{\rm \B{G}}}
\newcommand{\matX}{\B{\rm \B{X}}}
\newcommand{\matS}{\B{\rm \B{S}}}
\newcommand{\matT}{\B{\rm \B{T}}}
\newcommand{\matI}{\B{\rm \B{I}}}
\newcommand{\mate}{\B{\rm \B{e}}}
\newcommand{\matW}{\B{\rm \B{W}}}
\newcommand{\matA}{\B{\rm \B{A}}}
\newcommand{\matB}{\B{\rm \B{B}}}
\newcommand{\matC}{\B{\rm \B{C}}}
\newcommand{\matR}{\B{\rm \B{R}}}
\newcommand{\lbob}{\bm{\left[}}
\newcommand{\rbob}{\bm{\right]}}
\newcommand{\Een}{ E ^ { \, {\text{\rm en.}} } }
\newcommand{\Eind}{ E ^ { \, {\text{\rm Hess.}} } }
\newcommand{ \Psipell }{\Psi _ { p , \, \ell }}
\newcommand{ \scrI }{ \mathscr{I} }
\newcommand{ \rateell }{ \mathscr{I}^{\, (\ell)} }
\newcommand{\tsfE}{{\text{\rm\tiny\textsf{E}}}}

\newcommand{\minor}{   {\text{\rm minor}}  }
\newcommand{\zout}{   {\text{\rm zero}}  }
\newcommand{\Xminor}{ \matX _ {\minor, \, i} (r) }
\newcommand{\Xzout}{ \matX _ {\zout, \, i} (r) } 
\newcommand{\rp}{\mathfrak{r}}
\newcommand{\np}{\mathfrak{n}}

\title{The number of saddles of the spherical $p$-spin model}

\author{Antonio Auffinger \thanks{Department of Mathematics, Northwestern University, tuca@northwestern.edu, research partially supported by NSF Grant CAREER DMS-1653552, Simons Foundation/SFARI (597491-RWC), and  NSF Grant 1764421.} \\
\small{Northwestern University}
\and
Julian Gold \thanks{Department of Mathematics, Northwestern University, julian.gold@northwestern.edu, research partially supported by NSF PostDoctoral Research Fellowship DMS-1803622.}
	\\ \small{Northwestern University}
}

\begin{document}
\maketitle 


\begin{abstract} 
We show that the quenched complexity of saddles of the spherical pure $p$-spin model agrees with the annealed complexity when both are positive. Precisely, we show that the second moment of the number of critical values of a given finite index in a given interval has twice the growth rate of the first moment.

\end{abstract}
%
%%% table of contents
%\tableofcontents
%%%
%
%%%%%%%%%%%%%%%%%%%%%%%%%%%%%%%%%%%%%%%%%%%%%%%%%%%
%%%%%%%%%%%%%%%%%%%%%%%%%%%%%%%%%%%%%%%%%%%%%%%%%%%
%
\section{\large Introduction}
%
%%%%%%%%%%%%%%%%%%%%%%%%%%%%%%%%%%%%%%%%%%%%%%%%%%%
%%%%%%%%%%%%%%%%%%%%%%%%%%%%%%%%%%%%%%%%%%%%%%%%%%%
%
``How many critical values does a typical random Morse function have on a high dimensional manifold? How many of given index, or below a given level? What is the topology of level sets?" These questions were asked almost 10 years ago in \cite{
	ABC
 } which studied a class of natural random Gaussian functions on high-dimensional spheres, known as the pure spherical 
$ p $-spin model. The main result of \cite{
	ABC
 } was a rigorous derivation of the annealed complexity of the model, that is, asymptotics in 
$ N $, the dimension of the sphere, for the 
{\it mean }%							emph [ mean ]
number of critical points of given index in a given sub-level set. In particular, the authors of \cite{ABC} showed that the average number of local minima grows exponentially with $N$. The annealed complexity also allowed the authors to obtain information on this high-dimensional non-convex landscape, including a computation of the ground state energy, access to the averaged Euler characteristic, and the existence of diverging barriers between local minima.  

Five years after the annealed complexity was derived, in a remarkable article \cite{
	subag2017complexity
 }, Eliran Subag showed that the asymptotics obtained in \cite{
	ABC
 } for the number of local minima are valid without taking expectation. The current article aims to complete the picture for the complexity of saddles of the spherical pure 
$ p $-spin. We show (in a sense described below) that the quenched complexity i.e. the logarithm of the number of of critical points of finite index 
$ \ell $ in a given sub-level set agrees with the averaged complexity (the logarithm of the mean).

The spherical pure 
$ p $-spin glass model is defined as follows. Let 
$ p $ be an integer larger than 
$ 2 $ (the case 
$ p = 2 $ is rather trivial regarding complexity functions). Let 
$ \bbS _ N 
	= 
		\{ 
			\sig \in \R ^ N 
				: 
					\| \sig \| ^ 2 
						= 
							N 
		\}$ be the 
$ (N - 1) $-dimensional sphere of radius 
$ \sqrt { N } $. The pure 
$ p $-spin Hamiltonian is the following Gaussian random function on 
$ \bbS _ N $:
\begin{align*}							%%%
\HNp ( \sig ) 
	= 
		\frac{ 1 }
			{ 
				N ^
				{ 
					( p - 1 ) / 2
				}
			} 
				\sum_
				{
					i_1 
						\jdots 
							i_p
				} 
					J_
					{
						i_1 
							\jdots 
								i_p
					} \sig_
					{
						i_1
					} 
						\jddots
							\sig_
							{
							 	i_p
							} ,
\end{align*} 							%%%
where the coefficients 
$ J_{
	i_1
		\jdots 
			i_p}$ are i.i.d. standard Gaussians. This is a smooth, centered Gaussian function whose covariance is a function of the geometry of the sphere:
\begin{align*}							%%%
\E 
	\HNp ( \sig ) \HNp ( \sig ' ) 
		= 
			N 
			\leftp 
				\frac{ 1 }
					{ N }
					\la 
						\sig 
							\jc
								\sig '
					\ra 
			\rightp ^ p  ,
\end{align*}							%%%
where 
$ \la
	\jd
 \ra$ denotes the standard inner product in 
$ \R ^ N $.
We now introduce the complexity of spherical spin glasses. For any Borel set 
$ B 
	\subset 
		\R$ and integer 
$ 0
	\leq 
		\ell 
			<
				N $, consider the random number 
$ \CritNell ( B ) $ of critical values of the function 
$ H _ {
	N , \, p 
 } $ in the set 
$ NB 
	\equiv 
		\{
			Nx
				: 
					x \in B 
		\} $ with index equal to 
$ \ell $,
\begin{equation}						%%%
\label{defWk}%							defWk 
\CritNell ( B ) = 
	\sum _ {
		\sigma
			\tc 
				\nabla H _ N ( \sigma ) 
					\, = \, 
						0 
	}
		\1
			\leftcb  
				H _ N ( \sigma ) \in N B 
			\rightcb 
		\cdot \1
			\leftcb 
				\ind
					\leftp
						\nabla^2 H_N ( \sigma )
			\rightp 
				= 
					\ell
			\rightcb . 
\end{equation}							%%%
Here 
$ \nabla $, 
$ \nabla^2 $ are the gradient and the Hessian restricted to 
$ \bbS_N$, and 
$ \ind ( \nabla ^ 2 H _ { N } ( \sigma) ) $ is the index of the Hessian at 
$ \sigma $, i.e. the number of negative eigenvalues of this matrix.
To define the complexity function we first define the  \emph{energy threshold} 
\begin{equation}						%%%
\label{Ein}%							Ein
\sfE_\infty 
			\tri 
				2 \sqrt{
					\frac{ p - 1 }
						{ p }
				} .
\end{equation}							%%%
For 
$\ell \geq 1$, let  
$I _ \ell
	:
		( - \infty , - \sfE_\infty ]
			\to
				\R$ be given by
\begin{align}							%%%
\label{cookies}%						COOKIEESS
I _ \ell( u ) 
	&\tri
		\frac{ 2 \ell }
			{ \, \sfE _ \infty ^ 2 }
		\int _ { u } ^ { \, - \sfE _ \infty } 
			(
				z ^ 2 - \sfE _ \infty ^ 2 ) ^ { 1 / 2 } d z 		\non \\ %
	&=
	-\ell \cdot 
		\leftb
			\frac { u } 
				{ \sfE _ \infty ^ 2 } 
			\sqrt
			{
				u ^ 2 - \sfE _ \infty ^ 2 
			} 
			- \log 
			\leftp
				- u + \sqrt 
				{ 
					u ^ 2 - \sfE _ \infty ^ 2 
				} 
			\rightp 
			+  \log \sfE _ \infty
		\rightb. 
\end{align}							%%%
\begin{rmk}							%%%%%%%%%
\label{r:Is}%							r:ls
	In \cite{
		ABC
	}, it is proved that 
	$ I _ \ell ( u ) $ is the rate function of the LDP for the 
	$ \ell \oith $ smallest eigenvalue of a GOE matrix with the proper normalization of the variance of the entries. The case 
	$ \ell = 1 $ was first proved in \cite{
		benarous1997large
	}.
\end{rmk}								%%%%%%%%%
For any integer  
$\ell 
	\geq 
		0$, the complexity function of saddles of index 
$\ell$ is defined as
\begin{equation}						%%%
\label{e:thetakp} %						e:thetakp
\Sigma _ \ell ( u ) 
	\equiv
		\Sigma _ { p , \, \ell } ( u )
			\tri
				\begin{cases} %
					\frac {1}
						{2}
					\log ( p - 1 )  
    					-  \frac { p - 2 }
						{ 4 ( p - 1 ) }
					u ^ 2 
					- ( \ell + 1 ) 
					I _ 1 ( u ) ,
						& \text { if }%
							u 
								\leq 
									- \sfE _ \infty ,							\\ %
						\frac { 1 }
							{ 2 }
						\log ( p - 1 ) 
						-  \frac { p - 2 }
							{ p } , 
							& \text { if }%
								u 
									\ge 
										- \sfE _ \infty .
	\end{cases} %
\end{equation}							%%%
We note that 
$ \Sigma _ \ell  ( u ) $ are non-decreasing, continuous functions on 
$ \R $, with maximal value given by 
$ \frac{1}
	{2} 
\log(p-1) - \frac{p-2}
		{p} 
			>
				0$ (see Figure
	 ~\ref{Getabetterpicture}). As 
	$ u $ goes to 
	$ - \infty $, 
	$ \Sigma _ \ell ( u ) $ approaches 
	$ - \infty $. We thus introduce 
	$ \sfE _ \ell 
				> 
					0 $ as the unique solution to: 
\begin{equation}						%%%
\label{e:E_k} %							e:E_k
\Sigma _ \ell ( - \sfE _ \ell )
	= 
		0 .
\end{equation}							%%%
As suggested by the left-hand side of Figure~\ref{Getabetterpicture}, the sequence 
$ ( - \sfE _ \ell )_
	{
		\ell
			\geq 
				0
	} $ is increasing and converges to the energy threshold 
$ - \sfE _ \infty $.

\begin{figure}							%%%%%%%%%
\label{Getabetterpicture}%					Getabetterpicture
	\begin{center}
		\includegraphics
			[width=7cm]
				{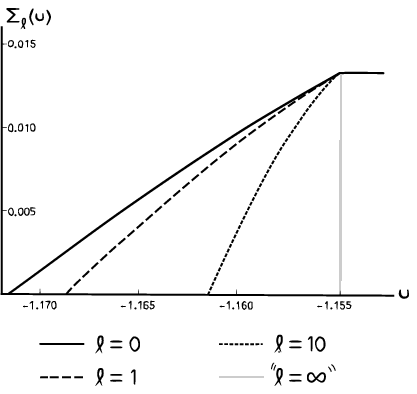}
		\includegraphics
			[width=7cm]
				{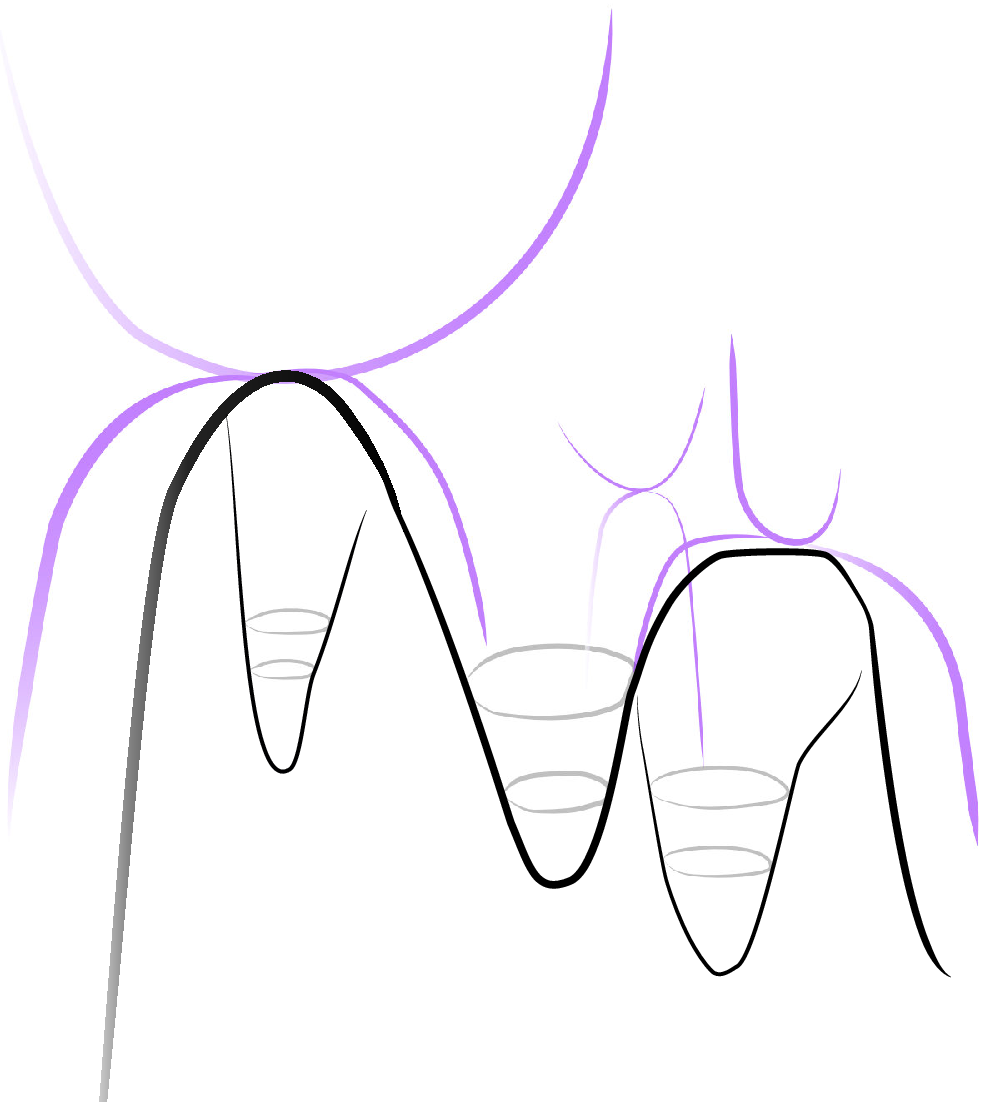}
	\end{center}
	\caption{%%
		On the left side, a graph of the complexity functions 
		$ \Sigma _ \ell $ for 
		$ p = 3 $ and 
		$ \ell = 0,1 $ and $10$, as well as for
		$ `` \ell = \infty ." $ All these functions agree for 
		$ u
			\ge
				- \sfE _ \infty $. On the right, a heuristic depiction of index-$1$ saddles first appearing at a threshold strictly above the ground state energy, as the landscape is scanned from bottom to top. 
	} %%
\end{figure}							%%%%%%%%%

Our first main result shows that for all energy values for which $\mathbb E \critNell (u) $ does not tend to $0$, the second moment agrees at the exponential scale with the square of the first.	
			
\begin{thm}								%%%%%%%%%
\label{thm:exp_match}%						thm:exp_match
For any 
$ p 
	\geq 
		3 $, 
$ \ell \in \{ 0 , 1 , \dots \} $ and 
$ u \in ( - \sfE _ \ell , - \sfE _ \infty ) $, 
\begin{align*}								%%%
\lim _ 
	{ N \to \infty } 
		\frac {1 } 
			{ N } 
			\log  \E 
				\leftp
					\critNell 
						\leftp \,
							( - \infty , u )
						\, \rightp 
				\rightp^2  
					&= 
						2 \lim _ 
							{ N \to \infty } 
								\frac { 1 }
									{ N } 
									\log 
										\E  
											\critNell 
												\left( \,
													( - \infty , u ) 
												\, \right)   				\\ %
					&= 2 \Sigma _ { p , \, \ell } ( u ) .
\end{align*}%%%
\end{thm}	

\begin{rmk}
The case $\ell =0$, i.e., counting the number of local minima was established in \cite[Theorem 1.5]{subag2017complexity}.
\end{rmk}
				%%%%%%%%%
%
Before we provide a rough idea of the proof of the theorem above, let us mention some historical aspects. The study of complexity of spin glass models has a long history outside pure mathematics, especially in the physics community. Indeed, many complex systems in physics, biology and computer science are characterized by high-dimensional landscapes full of local minima and saddles of any order. Starting in the '80s with the development of spin-glass theory \cite{
	Bray_1980,
	Kurchan_1991,
	CS95,
	CS,
	Cavagna_1997,
	KMV,
	MPV
 }, physicists have obtained several predictions for the number of critical points and local minima in mean-field models of glasses. A formula for the complexity of local minima in the pure $p$-spin was derived by Crisanti and Sommers \cite{
	CS95
 } and by Crisanti, Leuzzi and Rizzo \cite{
	CLR 
 } both at zero and positive temperature (a.k.a. the TAP complexity).

Major contributions were also given in related models. Fyodorov \cite{
	F0,
	F1,
	MR3469265
 } pioneered the use of random matrix theory in  complexity calculations.  He and his co-authors also provided examples of topology trivialization \cite{
	Fyo2016
 }, predictions for the Hessian spectrum \cite{
	MR3876574
 } and applications to directed polymers \cite{
	FyoPol
 }. The mixed $p$-spin model was studied in \cite{
	AB
 } and 
				\cite{AChenBip}. %							MissingNo.

The optimization of high-dimensional non-convex functions is the main task of several machine learning algorithms. There has been a recent burst of striking results relating the complexity and dynamics of spin glasses to those of deep neural networks. In this direction, we refer the reader to the following (non-exhaustive) list of papers at the intersection of computer science, mathematics, and physics \cite{
	ChoromanskaHMAL15,
	NIPS2014_5486,
	PhysRevX.9.011003
 }.  Last, for the pure 
$ p $-spin, a remarkable, rich prediction of the landscape of energy barriers was provided by Ros, Biroli, and Cammarota \cite{
	Ros_2019
 }. 
Different aspects of the landscape of the pure $p$-spin were also rigorously studied in the past. Fluctuations of the partition function and minimum energy were derived in \cite{
	SZ20, 
	MR3878351,
	MR3649446,
	MR3554380,
	Subag2017
 }. A theory that relates the landscape to more classical thermodynamical approaches was constructed in \cite{
	AC15, 
	AC16, 
	AC18Advances, 
	AZeng, 
	AChen15PTRF, 
	JT
 }.  
 
 In mathematics, computations of moments of the number of critical points were done in other settings. The reader is invited to check the work of Sarnak and Wigman \cite{SW}, Douglas, Shiffman and Zelditch \cite{SZ04, douglas2006} and Nazarov and Sodin \cite{NS} and the references therein. In those works, Gaussian fields are on a fixed space and, contrary to our setting, asymptotics are taken in parameters of different nature than the dimension. 
%
%%%%%%%%%%%%%%%%%%%%%%%%%%%%%%%%%%%%%%%%%%%%%%%%%%%
%
\subsection{Idea and novelty of the proof}
%
%%%%%%%%%%%%%%%%%%%%%%%%%%%%%%%%%%%%%%%%%%%%%%%%%%%
%
The starting point of the proof is the use of the Kac-Rice formula to obtain an expression for the second moment of 
$ \CritNell ( B ) $ as in \cite{
	subag2017complexity
 }. The main difficulty in this step comes from the presence of the constraint on the index of the Hessian and the absolute value of the determinant of $\nabla^{2}H_{N}$. Subag does not encounter these difficulties as counting the number of local minima effectively removes the absolute value of the determinant. 
 
 For level sets  near the global minima, the asymptotics of the total number of critical points coincide with the asymptotics of local minima. This provides the result for index-$0$ critical points. 
In order to obtain asymptotics for saddles we can't go through the same route. We do as follows. First, we note that the Hessian matrices 
$ \matM _
	{ 1 } ^ { N - 1 } $ and 
$ \matM _
	{ 2 } ^ { N - 1 } $ are correlated Gaussian matrices having 
$ ( N - 2 ) $-dimensional principal minors
$ \matG _
	{ 1 } ^ { N - 2 } $ and 
$ \matG _ 
	{ 2 } ^ { N - 2 } $, which are correlated shifted GOEs. There are essentially three steps in our proof. 
\begin{enumerate} %							enumerate
	\item[(1)] We use an isotropic semicircle local law to control the resolvent of $\nabla^{2}H_{N}$ and to transfer the Hessian index from the
$ \matM _ 
	{ i } ^ { N - 1 } $ to the
$ \matG _
	{ i } ^ { N - 2 } $. %
	\item[(2)] We realize the eigenvalues of the matrices 
$ \matG _ 
	{ 1 } ^ { N - 2 } $ and 
$ \matG _ 
	{ 2 } ^ { N - 2 } $ as two time points of a Dyson Brownian motion and derive a large deviation principle for its 
$ \ell \oith $ line. By contraction, we obtain an LDP for the pair of 
$ \ell \oith $ smallest eigenvalues. %
	\item[(3)] We optimize the resulting bound and recover the complexity function. %
\end{enumerate} %
Step $(1)$ is based on the recent success of rigidity results obtained in random matrix theory \cite{
	KY13,
	ELS09,
	benaych2018lectures
 }. In Step $(2)$, we realize the joint law of the eigenvalues of 
 $ \matG _ 
 	{ i } ^ { N - 2 } $, 
 $ i 
 	= 
		1 , 2 $ as two time points of a Dyson Brownian motion. In order to obtain an upper bound on the formula we use large deviation estimates for the pair of 
 $ \ell \oith $ largest eigenvalues. This can be done as in \cite{
 	donati2012large
 } where the case 
 $ \ell 
 	=
		1 $ was solved. Step $(3)$, although just computational, requires intricate calculus (aided by the analysis in \cite{
	ABC,
	subag2017complexity
 }).

\subsection{Refinement of Theorem \ref{thm:exp_match}}

Theorem  \ref{thm:exp_match} matches the quenched complexity and the annealed complexity at exponential scales. The theorem below is an enhancement of Theorem \ref{thm:exp_match} and establishes the almost sure behavior of the number of saddles of finite index. 
\begin{thm} \label{thm:Jesus} For any $\ell \geq 0$, and any $u \in (-\sfE_{\ell}, -\sfE_{\infty})$ we have 
\[
\frac{\critNell \left( \,( - \infty , u ) \, \right)  }{\mathbb E \critNell \left( \,( - \infty , u ) \, \right) } \to 1
\]
in probability and in $L^{2}$ as  $N$ goes to infinity.
\end{thm}

The proof of Theorem~\ref{thm:Jesus} is rather long and heavily computational but uses the same technology that we develop to prove Theorem  \ref{thm:exp_match}. We provide a summary of its proof in the appendix and further details in a forthcoming paper.

\subsection{Acknowledgments}

Both authors would like to thank Yi Gu and Eliran Subag for fruitful conversations related to the results of this paper. They also want to thank Yi Gu for useful comments on a previous version of this manuscript. The second author wishes to thank Pax Kivimae for helpful discussions.

\section{\large Notation and key inputs}
\label{sec:notation_inputs}%					sec_notation_inputs
%
%%%%%%%%%%%%%%%%%%%%%%%%%%%%%%%%%%%%%%%%%%%%%%%%%%%
%%%%%%%%%%%%%%%%%%%%%%%%%%%%%%%%%%%%%%%%%%%%%%%%%%%
%
We collect results necessary to our argument, starting with two main theorems of \cite{
	ABC
 }, introducing notation as necessary on the way. For 
 $ \vartheta > 0 $, define the following generalization of 
	\eqref{cookies}. 
\begin{align}								%%%
\label{eq:GOE_rate_general}%					eq:GOE_rate_general
I _ 1 ( u ; \vartheta ) 
	\tri
		\begin{cases} %
			\,\int_
				{ 2 \vartheta } ^ u 
			\vartheta^
			{
				-1
			} \leftp
				\leftp
					\frac { z } 
						{ 2 \vartheta } 
				\rightp^2 -1 
			\rightp ^ { 1 / 2 } 
			dz  
				& 
					u 
						\geq 
							2 \vartheta , 								\\ %
+\infty &\text{ otherwise. } 
		\end{cases} %
\end{align}								%%%
In this paper, an 
$ N 
	\times 
		N $ GOE matrix has law denoted $\goe_N$, with the convention that 
$ \matX \sim \goe_N $ has Gaussian entries 
$ \bm{[} 
	\matX 
\bm{]}_
	{ i , j } $ with variance
\begin{align*}								%%%
\E 
\lbob
	\matX
\rbob _
	{
		i , j 
	} ^ 2 
		= 
			N ^
				{
					-1
				}
			(
				1 + \delta _ 
					{
						ij
					}
			) \, , 
\end{align*}								%%%
a normalization ensuring 
	\eqref{eq:GOE_rate_general} with 
 $ \vartheta
 	 =
	 	1 $ is the rate function governing the leading eigenvalue of 
 $ \matX $. Recalling
 $ \Sigma _ \ell 
 	\equiv 
		\Sigma _ 
			{
				\ell , p 
			} $ from 
	\eqref{e:thetakp}, define the total complexity function 
\begin{equation}							%%%
\label{e:thetap}%							e:thetap
\Sigma ( u )
	\tri
		\begin{cases}%
    			\frac { 1 }
				{ 2 } 
			\log ( p - 1 ) 
    			-   \frac { p - 2 } 
				{ 4 ( p - 1 ) } 
			u ^ 2 - I _ 1 ( u ) ,
				& \text{ if } 
					u 
						\leq 
							- \sfE _ \infty ,						\\% case 1 
   			\frac { 1 } 
				{ 2 }
			\log ( p - 1 ) 
			-  \frac { p - 2 }
				{ p } 
			u ^ 2 , 
				& \text {if }  
					0 
						\leq 
							u 
								\leq 
									 - \sfE _ \infty , 				\\% case 2
    			\frac { 1 } 
				{ 2 } 
			\log ( p - 1 ) 
				& \text{ if }  
					0 
						\leq 
							u ,							 % case 3
  		\end{cases}%
\end{equation}								%%%
where $I_{1}(u)\tri I_{1}(-u;\sfE_{\infty}/2)$.%

We first record the averaged complexity results of \cite{ABC}.
\begin{thm}								%%%%%%%%%
	[{
			\cite
			[
				Theorem 2.5 and Theorem 2.6
			]
		{
			ABC
		}
}] 

\label{thm:ABC}%							thm:ABC

For all 
$ p 
	\geq 
		2 $, 
$ \ell 
	\geq 
		0 $, and $u \in (-\sfE_\ell, -\sfE_\infty )$, 
\begin{align*}								%%%
\lim _
	{ 
		N \to \infty
	} 
		\frac { 1 } 
			{ N } 
		\log 
			\E 
				\critNell
				(
					 ( - \infty , u ) 
				) 
					= 
						\Sigma _
							{ 
								\ell , p 
							} ( u ) 
								\quad 
									\ands
										\quad 
										 	\lim _ 
												{
													N \to \infty
												} 
												\frac{ 1 }
													{ N } 
												\log 
													\E 
														\crit _ N ( 
															( - \infty , u ) 
														) 
															= 
																\Sigma_{p} (u) .
\end{align*}								%%%
\end{thm}									%%%%%%%%%
For 
$ x \in \R $, let 
\begin{align}								%%%
\label{eq:Omega_def}%						eq:Omega_def
\Omega ( x )  
	& \tri 
		\int _ \R 
			\log | \lambda - x | \, 
		\semi \, 																	\\ %
	& = 
		\begin{cases} %
			\frac { x ^ 2 }
				{ 4 }  
			- \frac { 1 }
				{ 2 } 
					& \text{ if } 
						0 
							\leq 
								| x | 
									\leq 
										2 \,,										\\ % case 1
 			\frac{ x ^ 2 }
				{ 4 } 
			- \frac { 1 }
				{ 2 } 
			- \leftp
				\frac{ | x | }
					{ 4 } 
				\sqrt { 
					x ^ 2 - 4 
				} 
				- \log 
					\leftp 
						\frac { | x | + \sqrt 
							{ 
								x ^ 2 - 4 
							} 
						}
							{2} 
					\rightp  
			\rightp 
 				& \text{ if } 
					| x | 
						> 
							2.													 % case 2 
		\end{cases} %
\end{align}								%%%
denote the (negative of the) logarithmic potential of the semicircle law $ \semi $, whose density with respect to Lebesgue measure is 
\begin{align}								%%%
\frac{1}
	{2\pi}
\sqrt{
	4 - x ^ 2 
} 
\, \1 
	\leftcb
		 | x | \, \leq \, 2
	\rightcb  ,
\end{align}								%%%
agreeing with the variance convention 
$ \vartheta
	=
		1 $ for GOE matrices. 
\begin{rmk}								%%%%%%%%%
\label{rmk:IOmega}%						rmk:IOmega
The complexity functions 
$ \Sigma, \Sigma_\ell $ can be phrased in terms of 
$ \Omega $ through the following identities: 
\begin{align}								%%%
\label{eq:IOmega}%							eq:IOmega
\Omega ( x ) 
	&= 
		\frac { x ^ 2 }
			{ 4 } 
		- \frac { 1} { 2 } 
		- I _ 1 ( x ; 1 ) , \\
I_{1}(x,1) &= \frac{1}{\lambda} I_{1}(\lambda x, \lambda) \quad \text{ for all } \lambda >0.
\end{align}			
The function $\Omega$ describes the exponential-scale asymptotics of the determinant of $\goe_N$ matrices in the limit $N \to \infty$. These are relevant because of the Hessian determinant factors in the Kac-Rice formula are related to determinants of GOE matrices through Lemma~\ref{lem:rosetta_stone} below. 
\end{rmk}									%%%%%%%%%
Letting $\bbS 
	\tri 
		\{
			s \in \R^N 
				: 
					\| s \| ^ 2  
						=
							1
		\}$ denote the \emph{unit} $(N-1)$-sphere, for $s \in \bbS$, let
\begin{align*}
f( s ) 
	\equiv 
		f_{N,\,p} (s) 
			\tri 
				\frac { 1 } 
					{ \sqrt { N } }  
				\HNp 
				\leftp
					\sqrt{N} s 
				\rightp 
\end{align*}
denote the rescaled $ p $-spin Hamiltonian with domain $\bbS$. This is a centered, smooth Gaussian function on $\bbS$ 
 with
$ \E 
	f ( s ) 
	f ( t ) = \la 
	s , t 
\ra ^ p $ for
$ s, t \in \bbS$. The rescaled landscape $f$ thus has a particularly simple covariance structure, making it convenient to work with. Of course, counting the critical points of $f$ is equivalent to counting those of $H$.

The following lemma describes the joint law of $( \nabla^2f(\np), \nabla^2 f(\rp))$ conditioned on $\np, \rp \in \bbS$ being critical points at given energies, where
\begin{align}								%%%
\label{eq:np_def}%							eq:np_def
\frak { n }   
	&\tri (
	0 , \dots , 0 , 1
	) ,\\							%%%
\label{eq:rp_def}%							eq:rp_def
\rp \equiv \rp(r)
	&\tri (
		0 , \dots , 0 , \sqrt
					{
						1 - r ^ 2
					} , r 
	) . 
\end{align}								%%%
Kac-Rice expresses the number of pairs of critical points of $f$ as an integral over $\bbS \times \bbS$. At fixed overlap $r$, the rotational symmetry of the law of $f$ reduces the integral over $\bbS \times \bbS$ to an integral in $r$. The determinant factor in Kac-Rice then becomes a product of two terms: an expectation in the randomness described below, localized at $\np$ and $\rp$, and an entropy factor accounting for the dimension and volume of the sphere. Lemma~\ref{lem:rosetta_stone} thus provides an essential description of the geometry around pairs of critical points, and allows for further analysis of the Kac-Rice formula.

\begin{lem}								%%%%%%%%%
	[{
			\cite
			[
				Lemma 13
			]
			{
				subag2017complexity
			}
		}] 
\label{lem:rosetta_stone}%					lem:rosetta_stone
Let 
$ E
	\equiv 
		( E _ i ) _ 
			{ i = 1 } ^ { N - 1 } $ be an orthonormal frame on the unit sphere 
$ \bbS $, and let 
$ \frak { n } , \frak { r } \in \bbS $ be as in 
\eqref{eq:np_def} and 
\eqref{eq:rp_def}. For any 
$ r \in ( - 1 , 1 ) $, the following holds conditional on 
$ f ( \frak { n } ) 
	= 
		u _ 1 ,  
f ( \frak { r } ) 
	= 
		u _ 2 $ and 
$ \nabla f ( \frak { n } ) 
	= \nabla f ( \frak { r } ) 
		= 0 $: the pair
\begin{align}								%%%
\left( 
	\frac { 
			\nabla ^ 2 f ( \frak{n} ) 
		}
		{ 
			\sqrt{
				(N-1) p(p-1) 
			}
		} , 
	\frac {
			\nabla ^ 2 f ( \frak{r} ) 
		}{ 
			\sqrt{ 
				(N-1) p(p-1) 
			}
		} 
\right) 
\end{align}								%%%
has the same law as 
\begin{align}								%%%
\label{eq:M_joint}%							eq:M_joint
\left( 
	\un { \matM } _ 
		{ \, 1 } ^ { N - 1 } 
	( r , u _ 1 , u _ 2 ) , \, 
	\un { \matM } _ 
		{ \, 2 } ^ { N - 1 }  
	( r , u _ 1 , u _ 2 ) 
\right) ,
\end{align}								%%%
where 
\begin{align}								%%%
\label{eq:big_m_hat}%						eq:big_m_hat
\un { \matM } _ 
	{ \, i } ^ { N - 1 } 
( r , u_1 , u_2 )
		= \matM _ 
			{ \, i } ^ { N - 1 } 
		( r ) 
		- \frac { 2 } 
		{
			{ 
				\sfE _ \infty \sqrt{ N - 1 }
			}  
		}
		u_i \matI ^ { N - 1 }
		+ \frac { m _ i ( r , u _ 1 , u _ 2 ) }
		{ 
			\sqrt{
				( N - 1 ) p ( p - 1 )
			}
		} 
		\mate _ 
			{ N - 1 } ^ { N - 1 } 
		\,,
\end{align}								%%%
with 
$ \sfE _ \infty $ the energy threshold defined in 
\eqref{Ein}, where 
$ \matI^{N-1}$ denotes the 
$ ( N - 1 ) 
	\times 
		( N - 1 ) $ identity matrix, and where the constants 
$ m _ i $ are defined in 
\eqref{eq:little_m}. The matrix 
$ \mate _ 
	{ N - 1 } ^ { N - 1 } 
$ has a $0$ in every entry except the last row and column, where there is a $1$. The
$ \matM _ 
	{ i } ^ { N - 1} $ are Gaussian random matrices whose block structure
\begin{align}								%%%
\label{eq:rosetta_block}%						eq:rosetta_block
\matM _ 
	{ i } ^ { N - 1 } 
( r ) 
	\equiv  
		\left( 
			\begin{matrix} %
				\matG  _ 
					{ i } ^ { N - 2 }  
				( r )
						\vspace{2mm} & \, 
							Z _ 
								{ i } ^ { N - 2} 
							( r ) 									\\ % 
				\, Z _ 
					{ i } ^ { N - 2 }  
				( r ) ^ T  
					\, & \, 
						Q _ i ( r ) 
			\end{matrix} %
		\right) 
\end{align}								%%%
satisfies (1) -- (4) below. 
\begin{enumerate}
\item[(1)] %															( item 1 )
	The pairs 
	$ ( \,
		\matG _ 
			{ 1 } ^ { N - 2 } 
		( r ) , \, \matG _ 
			{ 2 } ^ { N - 2 }  
			( r ) 
	\, ) $, 
	$ ( \,
		Z _ 
			{ 1 } ^ { N - 2 }  
		( r ) , \, Z _ 
			{ 2 } ^ { N - 2 }  
		( r ) 
	\, ) $ and 
	$ ( \,
		Q _ 1( r ) , \, Q _ 2 ( r )
	\, ) $ are independent. 
\item[(2)] %															( item 2)
	The  
	$ \matG _ 
			{ i } ^ { N - 2 } 
			( r ) $ are  
	$ ( N - 2 ) 
		\times 
			( N - 2 ) $ random matrices such that 
	$ \sqrt{
		\frac { N - 1 } 
			{ N - 2 } 
	} 
	\, \matG _
		{ i } ^ { N - 2 } ( r ) $ 
	$ \sim \goe _ { N - 2 } \, $, and such that in law,
	\begin{align}							%%%
	\label{eq:GOE_joint}%					eq:GOE_joint
	\leftp
		\begin{matrix}%
			\matG _
				{ 1 }^ { N - 2 } (
			 r )
			 	\vspace{3mm}											\\ % 
			\matG _
				{ 2 } ^ { N - 2 } ( r )
		\end{matrix}%
	\right)
		=
			\left(
				\begin{matrix}%
					\sqrt{
						1 - | r | ^ { p - 2 } 
					} \, \bar { \matG } _
						{ 1 } ^ { N - 2 } 
					+ \, ( \sgn ( r ) ) ^ p  
					\sqrt{ 
						| r | ^ { p - 2 } 
					} \, \bar { \matG } 
						_ 0 ^ { N - 2 } 
						\vspace{3mm} 		\\ %
					\sqrt { 
						1 - | r | ^ { p - 2 } 
					} \, \bar { \matG } 
						_ 2 ^ { N - 2 }   
					+ \sqrt{ 
						| r | ^ { p - 2 } 
					} \, \bar{ \matG }_ 
						0 ^ { N - 2 }  
	\end{matrix}%
	\right)\,,
	\end{align}							%%%
	where 
	$ \bar { \matG } 
		_ 0 ^ { N - 2 } , \, \bar { \matG } _
			1 ^ { N - 2 } , $ and 
	$ \bar { \matG } 
		_ 2 ^ { N - 2 } $ are independent matrices, each distributed as 
	$ \matG _
		1 ^ { N - 2 } 
	( r ) $, so that 
	$ \sqrt{
		\frac { N - 1 } 
			{ N - 2 } 
	} 
	\, \bar{ \matG } 
		_ i ^ { N -  2 } 
			\sim
				\goe _ { N - 2 } \, $ for
	$ i 
		= 
			0, 1, 2 $. 
\item[(3)]															% ( item 3 )
	The column vectors 
	\begin{align*}							%%%
	Z _ 
		i ^ { N - 2 } 
	( r ) 
		\equiv
			\leftp
				\lbob 
					Z _ 
						i ^ { N - 2 } 
					( r ) 
				\rbob_j
			\rightp_
				{ j = 1 } ^ { N - 2 } 
	\end{align*}							%%%
	are Gaussian vectors such that, for any 
	$ j 
		\leq 
			N - 2 $, the pair 
	$( \, 
		\bm{[}
			Z _ 
				1 ^ { N - 2 } 
			( r )
		\bm{]}_j 
			, \, 
			\bm{[}
				Z _ 
					2 ^ { N - 2 }   
				( r ) 
			\bm{]}_j
	\, ) $ is independent of the analogous pair for any distinct 
	$ j ' 
		\leq 
			N - 2 $. The correlations within each pair are governed by
	$ \bm{\Sigma} _ Z ( r ) $, defined in
	\eqref{eq:Sigma_Z}:
	\begin{align}							%%%
	\label{eq:Z_normalization}%				eq:Z_normalization
	\leftp \,
		\lbob
			Z _ 
				1 ^ { N - 2 } 
			( r )
		\rbob _ j 
		,	 
		\lbob 
			Z _ 
				2 ^ { N - 2 } 
			( r )  
		\rbob_j
	\, \rightp 
		\sim 
			\mathcal{N}
			\leftp 
				0 
					, \, 
						( ( N - 1 ) p ( p - 1 ) ) ^ { - 1 } 
						\, \bm{\Sigma} _ Z ( r )  
	\rightp \,, 
	\end{align}							%%%
\item[(4)]%														( item 4 )
	The 
	$ Q _ i ( r ) $ are Gaussian with correlation stricture
	$ \bm{\Sigma} _ Q ( r ) $, defined in \eqref{eq:Sigma_Q_12}. 
	\begin{align}							%%%
	\label{eq:Q_normalization}%				eq:Q_normalization
	\leftp \,
		Q _ 1( r ) 
			, 
				Q _ 2 ( r ) 
	\, \rightp 
		\sim 
			\mathcal{N}
			\leftp
				0
					, \, 
						( ( N - 1 ) p ( p - 1 ) )^{-1} 
						\, \bm{\Sigma}_Q(r) 
			\right) \,. 
\end{align}								%%%
\end{enumerate}%
\end{lem}									%%%%%%%%%
\begin{rmk}								%%%%%%%%%
\label{rmk:rosetta}%							rmk:rosetta
The objects introduced in Lemma~\ref{lem:rosetta_stone} are central enough that, for the convenience of the reader, we now explicitly list abbreviations.
When the overlap parameter 
$ r $ and the dimension 
$ N $ are known implicitly, we drop these from the indexing, writing 
$ \matG _ i $ in place of 
$ \matG _
	i ^ { N - 2 } 
( r ) $,
$ Z _ 
	{ \, i } $ for 
$ Z _
	i ^ { N - 2 }
( r ) $, and 
$ Q _ i $ for 
$ Q _ i ( r ) $. We also write 
$ \matM _ i $ in place of
$ \matM _
	i ^ { N - 1 } 
( r ) $, so that 
\eqref{eq:rosetta_block} becomes 
\begin{align}								%%%
\label{eq:hess_block_structure}%				eq:hess_block_structure
\matM _ i 
	=  
		\leftp 
			\begin{matrix} %
				\matG_i 
					& 
						Z _ i 										\\ % 
				Z _ i ^ T 
					& Q _ i 
			\end{matrix} %
		\rightp \,.
\end{align}								%%%
The factor 
$ 2 / \sfE _ \infty $ in
\eqref{eq:big_m_hat} is like a change of units for the 
$ u _ i $, allowing us to reinterpret these energy levels of the pure
$ p $-spin energy landscape within the setting of a GOE matrix spectrum. This correspondence is also fundamental, and we write
\begin{align}								%%%
\label{eq:gamma_def}% 						eq:gamma_def
\gamma_p 
	\tri 
		\frac { 2 } 
			{ \, \sfE _ \infty }
				\equiv
					\sqrt{
						\frac { p }
							{ p - 1 }
					}
\end{align}								%%%
for brevity. For 
$ z \in \R $, let us use a bar to denote the rescaling performed on the energy parameters in 
\eqref{eq:big_m_hat}:
\begin{align}								%%%
\bar{z} = \frac{\gamma_p z}{\sqrt{N-1}}\,,
\label{eq:u_bar_def}
\end{align}								%%%
and this bar notation will be in effect in every section but 
Section~\ref{sec:exp_match}. The setting of 
Section~\ref{sec:exp_match} is one in which the limit 
$ N \to \infty $ has already been taken, so we repurpose the bar notation there. 

When the energy parameters 
$ u _ 1 $ and 
$ u _ 2 $
are also implicitly known, we make the following abbreviation:
\begin{align}								%%%
\label{eq:un_m_def}%						eq:un_m_def
m
	_ { \, i } ^ \circ  
		\tri 
			\frac{m_i(r,u_1,u_2) }{ \sqrt{(N-1)p(p-1)}} \,,
\end{align}								%%%
recalling the definition 
\eqref{eq:little_m} of the
$ m _ i ( r , u _ 1 , u _ 2 ) $. We write the matrices 
$ \un { \matM } _ 
	{ \, i } ^ { N - 1 } 
( r , u _ 1 , u _ 2 ) $ in 
\eqref{eq:M_joint} as 
$ \un { \matM } _ { \, i } $, so that in the above notation, 
\begin{align*}								%%%
\un { \matM } _ { \, i } 
	= 
		\matM _ i 
		- \bar{u}_i \matI ^ { N - 1 } 
		+ m _ { \, i } ^ \circ \mate _ { N - 1 } ^ { N - 1 } \, ,
\end{align*}								%%%
and finally, the 
$ ( N - 2 ) $-dimensional principal minor of 
$ \un { \matM } _ { \, i } $ shall be denoted 
$ \un { \matG } _ { \, i} $, so that
\begin{align*}								%%%
\un { \matG } _ { \, i } 
	\equiv 
		\matG _ i 
		- \bar{ u } _ i 
		\matI ^ { N - 2 } .
\end{align*}								%%%
\end{rmk}									%%%%%%%%%
A subset of 
$ \R $ is 
\emph{nice }%														emph : nice 
if it is a finite union of non-empty open intervals. For nice 
$ B 
	\subset 
		\R $, let 
$ \mathcal{C} _ N ( B ) $ denote the set of critical points of 
$ \HNp $ whose energies lie in 
$ N B $:
\begin{align}								%%%
\mathcal{C}_N(B) 
	= 
		\leftcb 
			\sig \in \bbS_N 
				: 
					\nabla \HNp ( \sig ) 
						= 
							0, \, 
								\HNp ( \sig ) \in N B 
		\rightcb ,
\end{align}								%%%
and for 
$ \ell 
	= 
		0 , 1 , 2 
			\dots 
				N - 2 $, let 
$ \mathcal { C } _ 
	{ N , \, \ell } 
( B ) $ be the subset of 
$ \mathcal { C } _ N ( B ) $ consisting of critical points with index 
$ \ell $. Define, for 
$ B 
	\subset 
		\R $ and 
$ R 
	\subset 
		( - 1 , 1 ) $ nice,
\begin{align}								%%%
\label{eq:crit_2_notation}%					eq:crit_2_notation
\leftcb
	\critNell (\,B, \, R\,) 
\rightcb _ { \bm { 2 } } 
	\tri 
		\# \leftcb
			( 
				\sig , \, \sig ' 
			) \in \mathcal{ C } _ 
				{ N , \, \ell } 
			( B )
			\times \mathcal { C } _
				{ N , \, \ell }
			( B ) 
				 : 
				 	q ( 
						\sig , \, \sig ' 
					) \in R
		\rightcb ,
\end{align}								%%%
where the overlap function 
$ q ( \, \cdot \, , \, \cdot \, ) $ is  
\begin{align}%%%
q ( \sig , \sig ' ) 
	\tri 
		\frac { 
				\la \, \sig , \, \sig ' \ra 
			} 
			{ 
				\| \, \sig \, \| \, \| \, \sig' \, \| 
			} \,,
\end{align}								%%%
where for a vector
$ x \in \R ^ N $, the notation 
$ \| \, x \, \| $ denotes its 
$ \ell ^ 2 $-norm, i.e. 
$\| \, x \, \| ^ 2 \equiv \la x, x \ra $. Finally, we let 
\begin{align}								%%%
\label{eq:omega_N_def}%						eq:omega_N_def
\omega _ N 
	\tri 
		\frac{ 2 \pi ^ 
			{ N / 2 }
		}
		{ 
			\Gamma ( N / 2 ) 
		}	
\end{align}								%%%
denote the surface area of the 
$ ( N - 1 ) $-dimensional 
unit %														emph:unit
 sphere $\bbS$. 
\section{\large Proof of main result and some corollaries}\label{sec:outline}
%
%%%%%%%%%%%%%%%%%%%%%%%%%%%%%%%%%%%%%%%%%%%%%%%%%%%
%%%%%%%%%%%%%%%%%%%%%%%%%%%%%%%%%%%%%%%%%%%%%%%%%%%
%
In this section we provide the proof of Theorem \ref{thm:exp_match} modulo two other results and establish some of its consequences. Let
\begin{align}								%%%
\label{eq:C_N_Kac_Rice}%					eq:C_N_Kac_Rice
\sfC_N 
	\tri 
		\leftp \,
			\frac{ ( N - 1 ) \, ( p - 1) }
				{ 2 \pi } 
		\, \rightp ^ { N - 1 } ,
\end{align}								%%%
\vspace{3mm}%												vspace : 3mm
\begin{align}%%%
\label{eq:F_G_Kac_Rice}%					eq:F_G_Kac_Rice
\mathcal{G}(r) 
	\tri 
		\leftp 
			\frac{ 
				1 - r ^ 2 
			}
				{
				1 - r ^ 
					{ 2 p - 2 } 
				} 
	\rightp ^ { 1 / 2 } 
		\quad \ands \quad
		 	\mathcal { F } ( r ) 
				\tri 
					\frac{
						1-r^{2p-2} 
					} 
						{ 
							 ( 
							 	\mathcal { G } ( r ) 
							) ^ 3
							\sqrt{
								1- (pr^p -(p-1)r^{p-2} )^2
							} 
						} \,. 
\end{align}

The first step is to use Kac-Rice. 
\begin{lem}								%%%%%%%%%
\label{lem:our_Kac_Rice}%					lem:our_Kac_Rice
Let 
$ r \in ( - 1 , 1 ) $, and from \eqref{eq:Sigma_U}, consider
$ ( \,
	U _ 1 ( r ) , U _ 2 ( r ) 
\, )  
	\sim 
		\mathcal { N } ( 0 , \bm{ \Sigma } _ U ( r ) ) $. For this 
$ r $, independently sample the matrix
$ \matM _ i $ with law described in 
Lemma~\ref{lem:rosetta_stone}, and using these, construct the associated matrix 
$\un { \matM } _ { \, i } 
	\equiv 
		\un { \matM } _
			{ \, i } ^ { N - 1 }
		( 
			r , U _ 1 ( r ) , U _ 2 ( r ) 
		) $. For any nice 
$ B 
	\subset 
		\R $ and nice 
$ R 
	\subset 
		( - 1 , 1 ) $, 
\begin{align}								%%%
\label{eq:our_Kac_Rice_output}%				eq:our_Kac_Rice_output
\E 
	\leftcb 
		\CritNell( 
			B ,  R
		) 
	\rightcb _ { \bm { 2 } } 
		&= \sfC_N 
			\int _ R dr 
				\leftp 
					\mathcal{G}(r)
				\rightp^N 
				\mathcal{F}(r) 
				\, \E 
					\leftp 
						\prod _ { i = 1 , \, 2 } 
							\left| 
								\det \leftp
									\un { \matM } _ { \, i } 
								\rightp 
							\right| 
						\1 \leftcb 
							\Een _ B \cap \Eind _ \ell 
						\rightcb 
					\rightp \,,
\end{align}								%%%
where 
$ \Een _ B $ is the event 
$ \{ \, U _ 1 ( r ) ,  U _ 2 ( r ) \in \sqrt{N} B \, \} $, and where 
$ \Eind _ \ell $ is the event that both 
$\un { \matM } _ { \, i } $ have index 
$ \ell $. The terms
$ \sfC _ N , \, \mathcal { G } ( r ) , \, \mathcal { F } ( r ) $ are as above in 
\eqref{eq:C_N_Kac_Rice} and 
\eqref{eq:F_G_Kac_Rice}.
\end{lem}									%%%%%%%%%
%	************************************************		%%%%%%%%%
\begin{proof} 
The proof is a standard application of the Kac-Rice formula.
\end{proof}
%	************************************************		%%%%%%%%%
%

The second step is to bound the right-hand side of \eqref{eq:our_Kac_Rice_output}, at the exponential scale using a \emph{bounding function} denoted $\Psipell$. Define for $\ell \in \N$, $r \in (-1,1)$ and $u_i < -2$, 
\begin{align} 								%%%
\Psipell ( r , u _ 1 , u _ 2 ) 
	&\tri 
		1 + \log ( \, p - 1 \, )  
		+ \log \mathcal { G } ( r ) 
		+ \sum _ 
			{ i
				\, = \, 
					1 , \, 2 } 
			\Omega 
			\leftp 
				\gamma_p \left| u_i \right| 
			\rightp 												\non \\ %
	&\quad 
		- \ell \cdot \mathscr { I } _ r 
			\leftp
				\gamma _ p | u _ 1 | , \, \gamma _ p | u _ 2 | 
			\rightp
			 - \frac { 1 } 
			 	{ 2 }
				\lbob
				\begin{matrix}%  
					u _ 1 & u _ 2 
				\end{matrix}%
			\rbob 
			\bm { \Sigma } _ U ( r ) ^ { - 1 }  
			\lbob 
				\begin{matrix}%
					u _ 1												 \\ %
					u _ 2 
				\end{matrix} %
			\rbob , \label{eq:bounding}%		eq:bounding
\end{align}								%%%
where 
$ \scrI _ r $ is the rate function from our LDP governing the leading eigenvalue pairs of correlated GOE matrices, Theorem~\ref{coro:LDP_ell}. Reading the expression for $\Psipell$ from left to right, the first three terms arise from exponential scale asymptotics of the entropy factors discussed just before the statement of Lemma~\ref{lem:rosetta_stone}. The $\Omega$ terms appear for the reasons described in Remark~\ref{rmk:IOmega}. 

The rate function $ \scrI_r$ is present in \eqref{eq:bounding} because of the Hessian index constraint $ \1 \{ \Eind _ \ell \} $ 
in \eqref{eq:our_Kac_Rice_output}. To use the LDP Theorem~\ref{coro:LDP_ell} for correlated GOE matrices, an intermediate step is required. As in \cite{subag2017complexity}, we will bound the Hessian determinants by determinants of related GOE matrices described in Lemma~\ref{lem:rosetta_stone}. We must also effectively replace the indicator function 
$ \1 \{ \Eind _ \ell \} $ 
in 
\eqref{eq:our_Kac_Rice_output} by an analogous constraint on corresponding GOE matrices. The next result, proven at the start of Section~\ref{sec:exp_bounds}, enables this ``index transfer." 

\begin{prop}								%%%%%%%%%
\label{prop:index_transfer}%					prop:index_transfer
Let 
$ \eps 
	> 
		0 $ be small, let 
$ r _ 0  \in (0,1) $, and for 
$ i 
	= 
		1 , 2 $ consider the families of random matrices 
\begin{align*}								%%%
\leftcb 
	\un { \matM } _ { \, i } 
	(
		r , u _ 1 , u _ 2 
	) 
		: 
			r \in [ - r _0 , \, r _ 0 ] 
				\ands 
					u _ 1 , \, u _ 2 \in 
					\left[
						-1 / \eps , \, - 2 - \eps 
					\right] 
\rightcb ,																\\ %
\leftcb 
	\un { \matG } _ { \, i }
	(
		r , u _ 1 , u _ 2 
	) 
		: 
			r \in [ - r _ 0 , \, r _ 0 ]
				\ands 
					u _ 1 , \, u _ 2 \in  
					\leftb 
						-1 / \eps , \, - 2 - \eps 
					\rightb
\rightcb ,
\end{align*}								%%%
as defined in Lemma
~\ref{lem:rosetta_stone} and Remark
~\ref{rmk:rosetta}. 
There is a coupling of these families, and an almost surely finite random variable 
$ N _ 0 ( \,\eps , \, r_0 \, ) $ so that 
$ N 
	\geq 
		N _ 0 $ implies 
\begin{align*}								%%%
\ind
\leftp
	\un{ \matM } _ { \, i } ( r , u _ 1, u _ 2 )
\rightp 
	= 
		\ind
		\leftp
			\un { \matG } _ { \, i } ( r , u _ 1 , u _ 2 ) 
		\rightp
\end{align*}								%%%
holds simultaneously for 
$ i = 1, 2 $ and for all 
$ r \in [ - r _ 0 , \, r _ 0 ] $ and 
$ u _ 1 , u _ 2 \in [ - 1 / \eps , \, - 2 - \eps ] $. 
\end{prop}									%%%%%%%%%
Theorem~\ref{thm5} below is the analogue of \cite[Theorem 5]{
	subag2017complexity
 } once Proposition~\ref{prop:index_transfer} is supplied. It is the main output of Section~\ref{sec:exp_bounds}.

\begin{thm}								%%%%%%%%%
\label{thm5}%								thm5
Suppose that 
$ B 
	\subset 
		( - \infty , - \sfE _ \infty ) $ is nice, and consider 
$ ( - r _ 0 , r _ 0 ) 
	\subset 
	( - 1 , 1 ) $ for some 
$ 0 
	< 
		r _ 0 
			< 
				1 $. The following bound holds: 
\begin{align}								%%%
\label{eq:thm5_statement}%					eq:thm5_statement
\limsup
	_ { N \to \infty } 
		\frac { 1 }
			{ N } 
			\log 
			\E 
			\leftcb 
				\critNell ( \,
					B , ( - r _ 0 , \, r _ 0 ) 
				\,) 
			\rightcb _ { \bm{ 2 } } 
				\, \leq 
					\sup _ 
						{ \substack
							{ r \, \in \, ( - r _ 0 , \, r _ 0 ) \\ 
							u _ 1 , \, u _ 2 \, \in \, B 
							}
						}  \Psipell\, (r,u_1,u_2) \,,
\end{align}								%%%
where $\Psipell$ is the bounding function defined in \eqref{eq:bounding}. 
\end{thm}									%%%%%%%%%
Analysis of the bounding function in Section~\ref{sec:exp_match} yields the next result. 
\begin{prop} 
\label{prop:L60}
Let $B \subset (- \sfE_\ell , -\sfE_\infty )$ be a nice set
\begin{align}								%%%
\label{eq:lem6_0}%							eq:lem6_0
\sup_{r \in (-1,1)}
	\sup_{ u_1, u_2 \in B} \Psipell(r,u_1,u_2) 
		=
			\sup_{u \in B} \Psipell(0,u,u).
\end{align}								%%%
\end{prop}

Theorem~\ref{thm:exp_match} follows directly from Proposition~\ref{prop:L60} and Theorem~\ref{thm5}. Its proof, given in Section~\ref{sec:exp_match}, hinges on the following relationship between the bounding and complexity functions: for
$ u 
	< 
		- \sfE _ \infty $, 
\begin{align}								%%%
\label{eq:thm37start}%						eq:thm37start 
\Psipell(0, u , u ) 
	\equiv 
		2 \Sigma _ {p , \, \ell } ( u ).
\end{align}%%%
%
%\begin{thm}								%%%%%%%%%
%\label{thm:exp_match}%						thm:exp_match
%For any 
%$ p 
%	\geq 
%		3 $, 
%$ \ell \in \{ 0 , 1 , \dots \} $ and 
%$ u \in ( - \sfE _ \ell , - \sfE _ \infty ) $, 
%\begin{align*}								%%%
%\lim _ 
%	{ N \to \infty } 
%		\frac {1 } 
%			{ N } 
%			\log  \E 
%				\leftp
%					\critNell 
%						\leftp \,
%							( - \infty , u )
%						\, \rightp 
%				\rightp^2  
%					&= 
%						2 \lim _ 
%							{ N \to \infty } 
%								\frac { 1 }
%									{ N } 
%									\log 
%										\E  
%											\critNell 
%												\left( \,
%													( - \infty , u ) 
%												\, \right)   				\\ %
%					&= 2 \Sigma _ { p , \, \ell } ( u ) \,.
%\end{align*}%%%
%\end{thm}									%%%%%%%%%
%
%
\begin{rmk}								%%%%%%%%%
\label{rmk:parallel_crit}%						rmk:parallel_crit
The constraint 
$ u > - \sfE _ \ell $ ensures 
$ \Sigma _ { p , \, \ell } ( u )
	>
		0 $, and hence that  
\begin{align*}								%%%
\lim _ 
	{ N \to \infty } 
		\frac { 1 }
			{ N } 
			\log 
			\E 
					\leftp 
						\critNell ( \,
							( - \infty , \, u ) 
						\, )
					\rightp^2  
						> 
							\lim _ 
								{ N \to \infty } 
									\frac { 1 }
										{ N }  
										\log 
										\E 
											\critNell( \,
												( - \infty , \, u )
											) \,,
\end{align*}								%%%
from which it follows 
\begin{align*}								%%%
\lim _ 
	{ N \to \infty } 
		\frac { 1 }
			{ N } 
			\log 
				\E 
						\leftp
							\critNell( \,
								( - \infty , \, u )
							\, ) 
						\rightp^2 
							\equiv 
								\lim _
									{ N \to \infty } 
										\frac { 1 }
											{ N } 
											\log 
												\E 
													\leftcb
														\critNell( \, 
															( - \infty , \, u ) , \, ( - 1 , \,1) 
														\, ) 
													\rightcb _ { \bm { 2 } } .
\end{align*}								%%%
\end{rmk}									%%%%%%%%%
The corollary below says that, at the exponential scale, most pairs of critical points of fixed index 
$ \ell $ are nearly orthogonal. 
\begin{coro}								%%%%%%%%%
\label{coro:cor8}%							coro:cor8
For any 
$ p 
	\geq 
		3 $ and any 
$ u \in ( - \sfE _ \ell ,  - \sfE _ \infty ) $ and 
$ \eps 
	>
		0 $, let 
$ R_\eps 
	\tri 
		( - 1 , - 1 ) \setminus [ - \eps, \eps ] $. Then,  
\begin{align*}								%%%
\lim _ 
	{ N \to \infty } 
		\frac { 1 } 
			{ N } 
			&\log 
				\E 
					\leftcb 
						\critNell( \, 
							(  - \infty , \, u ) , \, ( - 1 , \, 1 ) 
							\, ) 
					\rightcb _ { \bm { 2 } } 
			 			> \, 
			 				\lim _ 
							 	{ N \to \infty } 
									\frac { 1 }
										{ N } 
										\log 
											\E 
											\leftcb 
												\critNell( \, 
													( - \infty , \, u ) , \,  R _ \eps
												\, )
											\rightcb _ { \bm { 2 } } .
\end{align*}								%%%
\end{coro}									%%%%%%%%%

For $0 < \rho < 1$, we make the abbreviation 
\begin{align}								%%%
\label{eq:abbrev_ov_int}%						eq:abbrev_ov_int
\leftcb
	\critNell ( B ) 
\rightcb _ { \bm { 2 } }^\rho 
	\tri 
		\leftcb
			\critNell
				( \, 
					B , \, ( - \rho , \, \rho ) 
				\, )
		\rightcb _ { \bm { 2 } }.
\end{align}		
We close the section by recording a last consequence of Theorem \ref{thm:exp_match}, Corollary~\ref{lem:sub20}, which is itself relevant to the proof of Theorem~\ref{thm:Jesus}. Corollary~\ref{lem:sub20} is analogous to \cite[Lemma 20]{
	subag2017complexity
 } and has a similar proof, repeated here because it is short. 
\begin{coro}								%%%%%%%%%
\label{lem:sub20}%							lem:sub20
Let 
$ u_* \in ( - \sfE _ \ell , - \sfE _ \infty ) $, 
$ \rho \in ( 0 , 1 ) $ and 
$ \eps > 0 $. Then
\eq{										%%%
\lim _ 
	{ N \to \infty } 
		\frac{ 
			\E 
				\leftcb 
					\critNell ( 
						(\, u_* - \eps , \, u_* \, )
					)
				\rightcb _ { \bm { 2 } } ^ \rho  
		}
		{
			\E 
				\leftb 
					\critNell ( \, 
						( - \infty , \, u_* ) 
					\, ) 
				\rightb^2
		} 
			=
				1\,.
		}	
\end{coro}

\begin{proof}
We work within the almost sure event that 
$ \critNell ( \,
	(- \infty , \, u _ *)
\, ) 
	= \critNell ( \,
		(- \infty, \, u_* - \eps)
	\, )
	+ \critNell ( \,
		(u_* - \eps , \, u_*) 
	\, ) $. Rearranging terms after squaring both sides of this, one finds:
\begin{align*}
\Big[
	\critNell ( \,
		(- \infty , \, u_* ) 
	\, )
\Big]^2
&- \Big[
	\critNell ( \,
		(u_* - \eps , \, u_*)
	\, )
\Big]^2																	\\ %
	& \quad = 
		\Big[
			\critNell ( \,
				- \infty , \, u_* - \eps
			\, )
		\Big]^2
		+ 2 \critNell ( \,
					( - \infty , \, u_* - \eps ]
		\, ) \cdot \critNell ( \, 
			( u_* - \eps , \, u_* ) 
		\, ) 
\end{align*}								%%%
Theorem~\ref{thm:exp_match} furnishes exponential-scale asymptotics for these terms:
\begin{align*}
\lim _ 
	{ N \to \infty } 
		\frac { 1 }
			{ N }
		\log 
			\E 
				\Big[
					\critNell ( \,
						( - \infty , \, u_* ) 
					\, ) 
				\Big]^2 
			&= 2 \Sigma _ { p , \, \ell } ( u_* ) \\
\lim _ 
	{ N \to \infty } 
		\frac { 1 }
			{ N }
		\log 
			\E 
				\Big[
					\critNell ( \,
						( - \infty , \, u_* - \eps ) 
					\, ) 
				\Big]^2 
			&= 2 \Sigma _ { p , \, \ell } ( u_* - \eps ) ,	\\
\lim _ 
	{ N \to \infty } 
		\frac { 1 }
			{ N }
		\log 
			\E \,
				2
				\critNell ( \,
					( - \infty , \, u_* - \eps ) 
				\, )
				\cdot \critNell ( \,
					( u_* - \eps , \, u_* )
				\, )
			&= 
				\Sigma _ { p , \, \ell } ( u_* ) 	
				+ \Sigma _ { p , \,\ell } ( u_* - \eps ) ,
\end{align*}								%%%
the last line following from Cauchy-Schwarz.
							%%%
As the complexity function 
$ u \mapsto \Sigma _ { p , \, \ell } ( u ) $ is strictly increasing over the interval 
$ ( - \infty, - \sfE _ \infty ) $, the above displays imply
\eq{							%%%
\lim _
	{ N \to \infty } 
		\frac { 
			\E 
				\Big[
					\critNell (\,
						( - \infty , \, u_* ) 
					\, )
				\Big]^2
		}
		{
			\E 
				\Big[
					\critNell (\,
						( u_* - \eps , \, u_* )
					\, )
				\Big]^2
		}
			=
				1.
}								%%%
This puts us in the desired situation, as far as the energy parameter. As 
$ u_* 
	>
		-\sfE _ \ell $, one has $\Sigma _ { p , \, \ell } (u_* ) 
	> 
		0,$ so that by Remark~\ref{rmk:parallel_crit}, 
\begin{align*}								%%%
\lim _ 
	{ N \to \infty }
		\frac
		{
			\E
				\leftcb
					\critNell ( \,
						( u_* - \eps , \, u_* )
					\, )
				\rightcb _ { \bm { 2 } } ^ 1
		}
		{
			\E
				\Big[
					\critNell ( \, 
						( u_* - \eps , \, u_* ) 
					\, )
				\Big] ^ 2
		}
			=
				1.
\end{align*}								%%%
Recalling that 
$ R_ \rho $ denotes
$ ( - 1, 1 ) \setminus [ - \rho , \rho ] $,  Corollary~\ref{coro:cor8} implies
\begin{align*}
\lim _ 
	{ N \to \infty }
		\frac
		{
			\E
				\leftcb
					\critNell ( \,
						( u_* - \eps , \, u_* )  , \, R _ \rho
					\, )
				\rightcb _ { \bm { 2 } } 
		}
		{
			\E
				\leftcb
					\critNell ( \,
						( u_* - \eps , \, u_* )
					\, )
				\rightcb _ { \bm { 2 } } ^ 1
		}
			=
				0,
\end{align*}
completing the proof. 
 \end{proof}

\vspace{4mm}%														vspace : 4mm
%
%%%%%%%%%%%%%%%%%%%%%%%%%%%%%%%%%%%%%%%%%%%%%%%%%%%
%
%
%\subsection{Outline: auxiliary results, appendix} 
%
\vspace{4mm}
\section{\large Index transfer and exponential bounds}\label{sec:exp_bounds} 
The output of this section is a proof of Theorem~\ref{thm5}, given in the last subsection. In the prior subsection, we outline the key steps in the proof of Theorem~\ref{thm5}, one of which is the application of the ``index transfer" result Proposition~\ref{prop:index_transfer}. We prove Proposition~\ref{prop:index_transfer} in the first subsection.  

The notational conventions described in Remark~\ref{rmk:rosetta} are in effect throughout the section. For $i = 1,2$ and $\ell \in \{0, \dots, N-2\}$, define the events
\begin{align}
E_{\ell,\,i} &\tri \left\{ \ind(\un{\matM}_{\,i}) = \ell \right\} , \label{eq:E_ell_def} \\
E_{\ell,\,i}^* &\tri \left\{ \ind(\un{\matG}_{\,i}) = \ell \right\} . \label{eq:E_ell_star_def}  
\end{align}

\subsection{Index transfer}{\label{sec:4_1}} 
The next result allows us to reformulate Proposition~\ref{prop:index_transfer} into an equivalent statement about the resolvent of a GOE matrix.

\begin{lem}[{\cite[Equation 2]{lazutkin1988signature}} ] Let $\matS$ be a symmetric block matrix, denote its signature by $\sgn(\matS)$. Write $\matS$ and its inverse $\matS^{-1}$ in block form with the same block structure:
\begin{align*}
\matS = \begin{pmatrix} \matA & \matB \\ \matB^T & \matC \end{pmatrix}\,,\quad\quad S^{-1} = \begin{pmatrix} \matA' & \matB' \\ (\matB')^T & \matC' \end{pmatrix} . 
\end{align*}
In this setting, we have that $\sgn(\matS) = \sgn(\matA) + \sgn(\matC')$.
\label{lem:lazutkin}
\end{lem}

Enumerate the eigenvalues of $\un{\matG}_{\,i}$ and $\un{\matM}_{\,i}$ in ascending manner as $\{ \lambda_j(\un{\matG}_{\,i}) \}_{j=1}^ {N-2}$ and $\{\lambda_j(\un{\matM}^{\,i})\}_{j=1}^ {N-1}$. By the interlacement property,
\begin{align}
\lambda_j (\un{\matM}_{\,i}) \leq \lambda_j(\un{\matG}_{\,i}) \leq \lambda_{j+1}(\un{\matM}_{\,i})
\label{eq:interlace}
\end{align} 
holds for all $1 \leq j \leq N-2$. On the event $E_{\ell,\,i}$ defined just above in \eqref{eq:E_ell_def}, display \eqref{eq:interlace} implies $\ind(\un{\matG}_{\,i}) \in \{\ell-1, \ell\}$. To apply Lemma~\ref{lem:lazutkin}, first express $\un{\matM}_{\,i}^{-1}$ in block form with block structure as in \eqref{eq:hess_block_structure}, writing  
\begin{align}
\lbob \un{\matM}_{\,i}^{-1}\rbob_{N-1,\, N-1} = \left( \lbob \un{\matM}_{\,i} \rbob_{N-1,\,N-1} - \left\la Z_i, \un{\matG}_{\,i}^{-1} Z_i \right\ra \right)^{-1} 
\label{eq:lazutkin_variable}
\end{align}
using the Schur complement formula, and note that Lemma~\ref{lem:lazutkin} and \eqref{eq:interlace} together imply the index of $M^i$ is equal to the index of $G^i$ exactly when 
\begin{align}
X_i \equiv X_i(u_i) \tri Q_i -\bar{u}_i + m _ { \, i } ^ \circ - \left\la Z_i, \un{\matG}_{\,i}^{-1} Z_i \right\ra > 0\,. 
\label{eq:X_def} 
\end{align}

Given a Wigner matrix $\matA$ with spectrum $\sig(\matA)$ and spectral parameter $z \in \C \setminus \sig(\matA)$, let
\begin{align}
\matR(z) \equiv R(z;\matA) \tri (\matA - z\matI)^{-1} \,,
\end{align}
denote the resolvent of $\matA$, and given a real spectral parameter $u < -2$, denote the Stieltjes transform of the semicircle law by
\begin{align}
m(u) &\tri \int \frac{1}{\lambda - u} \semi(d \lambda) \equiv \frac{-u + \sqrt{u^2 - 4}}{2}\,.
\end{align}
We use the following local law to control the inner product term in \eqref{eq:X_def}. 

\begin{thm}[{\cite[Theorem 10.3]{benaych2018lectures}} ] Let $\matA$ be an $N \times N$ Wigner matrix with resolvent $\matR(z) \equiv \matR(z;\matA)$. Fix $\eps >0$ and define the interval $S(\eps) \tri (-\infty, -2 -\eps]$. There is $c(\eps) >0$ so that, for all deterministic unit vectors $e \in \R^N$, all small $\delta >0$ and large $D >0$, 
\begin{align}
\prob\left( \,\sup_{u\, \in\, S(\eps)}| \la e , \matR(u) e \ra - m(u) | \geq  c N^{\delta-1/2} \right) \leq N^{\,-D}\,
\label{eq:local_law}
\end{align}
holds when $N \geq N_0(\delta, D)$, i.e. when $N$ is sufficiently large depending on the parameters $\delta,$ and $D$.  
\label{thm:local_law}
\end{thm}
% Should we be explicit about the dependence of D large and delta small on parameter epsilon (presumably). I also still feel I could be more careful in the above subsection about when I use $u$ and when I use $\bar{u}$. Ultimately we're interested in an energy of order one.

\begin{rmk} Theorem~\ref{thm:local_law} follows from the proof of Theorem 10.3 in \cite{benaych2018lectures}, see equation (10.6) and the associated footnote. The latter theorem was stated using the notion of \emph{stochastic domination uniform in a set of parameters}, see \cite[Definition 2.5]{benaych2018lectures}. The parameter set in our case is $S \equiv S(\eps)$ above. Following \cite[Remark 2.7]{benaych2018lectures} and  \cite[Remark 2.6]{bloemendal2014isotropic}, Theorem~\ref{thm:local_law} upgrades a uniform bound on the collection of probabilities
\begin{align}
\left\{ \prob\left( | \la e , \matR(u) e \ra - m(u) | \leq cN^{\delta -1/2} \right) : u \in S  \right\} \,
\end{align}
into the simultaneous bound \eqref{eq:local_law} controlling a supremum. Moreover, though Theorem~10.3 was stated for a parameter set $S = \{ u + i\eta : |u| \geq 2 +\eps, \eta > 0 \}$, it one can send the imaginary part $\eta$ to zero to recover Theorem~\ref{thm:local_law}, see for instance \cite[Remark 2.7]{bloemendal2014isotropic}. 
\end{rmk} 

\begin{proof}[Proof of Proposition~\ref{prop:index_transfer}] It will suffice to show the result for one matrix, so we suppress $i$ in our notation. Write the interval $(-\infty, -2-\eps]$ as $S(\eps)$, as in the notation of Theorem~\ref{thm:local_law}, and suppose $u \in \R$ is such that $\bar{u} \in S(\eps)$, recalling \eqref{eq:u_bar_def}. 

Expressing $\un{\matG}$ as $\matG - \bar{u}\matI$, write $|\sigma(\matG) - \bar{u})|$ for the distance of the spectrum $\sigma(\matG)$ to $\bar{u}$, and let $E_{(1)}$ be the event $\{ | \sig(\matG) - \bar{u} | \geq \eps /2 \}$. As $\sqrt{(N-1)/(N-2)} \matG \sim \goe_{N-2}$,
the law of the smallest eigenvalue of $\matG$ and that of a $\goe_{N-2}$ matrix are exponentially equivalent. By \cite[Theorem 4.2.13]{dembo1998large} and the large deviation principle for the leading eigenvalue of a GOE matrix, \cite[Theorem 6.2]{benarous2001aging}, it follows that 
\begin{align}
\prob\left(E_{(1)}^\rmc\right)  \leq \exp(-c_1 N) \,,
\end{align}
for some $c_1(\eps) >0$. 

On the high-probability event $E_{(1)}$, the resolvent $\matR(\bar{u}) \equiv \matR(\bar{u};\matG) \equiv \un{\matG}^{-1}$ is well-defined.
Using Theorem~\ref{thm:local_law}, fix $\delta < 1/2$ small and $D > 1$ large so that for all unit vectors $e \in \R^{N-2}$ and $N$ sufficiently large,
\begin{align*}
\prob\left( \, \sup_{\bar{u}\, \in\, S(\eps)}| \la e , \matR(\bar{u}) e \ra - m(\bar{u}) | \geq  c N^{\delta-1/2} \right) \leq N^{\,-D} , 
\end{align*}
where $c(\eps) >0$. Write $Z_\ang$ for the random unit vector $Z / \| Z \|$, and define the event 
\eq{ 
E_{(2)} \tri \left\{|\la Z_\ang, \matR(\bar{u}) Z_\ang \ra - m(\bar{u}) | < c(\eps) N^{\delta-1/2} \right\},
}
noting that
\eq{
\prob\left( E_{(2)}^\rmc \right) &= \E_\ang \left[ \prob_{G} \left( | \la e , \matR(\bar{u}) e \ra - m(\bar{u}) | \geq cN^{\delta-1/2}\, \big|\, Z_\ang = e \right) \right] \leq N^{-D} .
}
Directly above, $\E_{\ang}$ denotes expectation with respect to $Z_\ang$, $\prob_G$ is the law of the matrix $G$, and we have used the independence of $Z$ and $G$ stated in item (1) of Lemma~\ref{lem:rosetta_stone}. 

The random magnitude $\|Z\|^2$ is $(N-1)^{-1} a(p,r)$ times a $\chi^2_{N-2}$ random variable. Recalling the covariance matrix $\bm{\Sigma}_Z(r)$ of $Z$ given in \eqref{eq:Sigma_Z}, the constant $a(p,r) >0$ is $\bm{\Sigma}_{Z,11}(r) / p(p-1)$. We now use standard concentration results for Lipschitz functions of Gaussians, for instance \cite[Theorem 5.6]{BLM}, for $\zeta >0$ to be chosen later,
\begin{align}
\prob \left( \left| \,\| Z \| - \E \| Z\| \,\right| \geq \zeta \right) \leq 2 \exp\left( -\zeta^2N / 2a(p,r)  \right).
\label{eq:bound_chi_pr}
\end{align}
In particular, the explicit form of the mean of $\chi$-distributed random variables and gamma function asymptotics imply that, for $N$ sufficiently large,
\begin{align*}
\prob \left( \left|\, \| Z\| - \sqrt{ a(p,r)}\, \right| > 2\zeta \right) \leq 2 \exp\left(- \zeta^2 N/ 2a(p,r) \right)\,,
\end{align*}
where the complement of the event on the left-hand side will be denoted $E_{(3)}$.

Lastly, we give a high probability bound on the first term of $X$: writing $E_{(4)} = \{ Q - \bar{u} + m^\circ \geq - \bar{u} - \zeta \}$, there is $c_4(\zeta,p,r) > 0$ so that
\begin{align}
\prob( E_{(4)}^\rmc) \leq \exp( -c_4 N ) \,,
\end{align}
which follows from the fact that $m^\circ$ is deterministic, on the order of $N^{-1/2}$, while $Q$ is a centered Gaussian with variance on the order of $N^{-1}$. 

Before concluding, we make two comments. The dependence of the constant $c_4$ on $p$ and $r$ comes from the variance of $Q$, given in \eqref{eq:Sigma_Q_11}. The dependence on $r$ can be dropped by noting the variance of $Q$ is uniformly bounded in $r$. While not transparent from \eqref{eq:Sigma_Q_11}, it follows directly from \cite[Lemma 15]{subag2017complexity}, which bounds the moments of the variable \eqref{eq:W_def} introduced in the next subsection. On the other hand, it is straightforward to show the constant $a(p,r)$ in \eqref{eq:bound_chi_pr} is bounded uniformly in $r$ from above by one. Statements we make below can thus be shown to hold uniformly in $r$ as well as $\bar{u} \in S(\eps)$: though the variables $Z$ and $Q$ change with $r$, they can each be realized as the appropriate function of $r$ times a fixed standard Gaussian or standard Gaussian vector, each rescaled according to $N$.

Using the aforementioned bound $a(p,r) \leq 1$, and working within the intersection of $E_{(1)}$ through $E_{(4)}$, the inequality 
\begin{align*}
X(u) \geq - \bar{u} - \zeta - (1 +2 \zeta)^2 \left( m(\bar{u}) + c(\eps) N^{\delta-1/2} \right) \,
\end{align*}
holds simultaneously for all $\bar{u} \in S(\eps)$, and by the above comments, for all $r$. For all $\bar{u} \in S(\eps)$, we have $-\bar{u} - m(\bar{u}) >0$. For $\zeta$ tuned appropriately in terms of $\eps$, and for $N$ large, each $X(u)$ is positive; our ability to choose $\zeta$ well relies on the fact that $S(\eps)$ is compact, as we first need a uniform lower bound on $-\bar{u} - m(\bar{u})$ by some positive constant (depending on $\eps$). We apply Borel-Cantelli to complete the proof. \end{proof}

\subsection{Inputs to the proof of Theorem~\ref{thm5}}{\label{sec:4_2}} The strategy for the proof of Theorem~\ref{thm5} goes as follows.
\begin{enumerate}
\item[(1)] We use lemmas from \cite{subag2017complexity} to bound the output of the Kac-Rice formula, described by Lemma~\ref{lem:our_Kac_Rice}, in terms of determinants of GOE matrices.
\item[(2)] The index constraint on the Hessians persists in the above bound, so we apply Proposition~\ref{prop:index_transfer} and Corollary~\ref{coro:index_transfer_thm5} to transfer the index constraint to the GOE matrices. 
\item[(3)] We use the LDP Theorem~\ref{coro:LDP_ell} on the eigenvalues of these GOE matrices. 
\item[(4)] We make a change of variables, as in \cite{subag2017complexity}, in order to apply Varadhan's lemma.
\item[(5)] We apply Varadhan's lemma and complete the proof.
\end{enumerate}

The starting point for the proof of Theorem~\ref{thm5} is Lemma~\ref{lem:our_Kac_Rice}, so we no longer consider fixed energies as in the last subsection. Recall that the covariance matrix $\bm{\Sigma}_U(r)$, given in \eqref{eq:Sigma_U}, describes the law of the energies of a pair of critical points of the rescaled $p$-spin landscape. 

For fixed $r$, the matrices $\matM_{\,i}$ are defined as in the previous subsection. Let $(U_1(r), U_2(r)) \sim \mathcal{N}(0, \bm{\Sigma}_U(r))$ be independent of $(\matM_1,\matM_2)$, and for $i=1,2$, define the random variables
\begin{align}
\bar{U}_i \equiv \bar{U}_i (r) = \frac{2}{\sfE_\infty\sqrt{N-1}} U_i(r)\,,
\label{eq:U_bar_def}. 
\end{align}
In this context, the shifted matrices $\un{\matM}_{\,i}$ are now defined conditionally on $(U_1,U_2)$ through the usual identity: 
\begin{align*}
\un{\matM}_{\,i} = \matM_i - \bar{U}_i \matI + m^\circ_{\,i} \mate_{N-1,\,N-1} 
\end{align*}
For $B \subset \R$ fixed, define the events
\begin{align}
A_i &\tri \left\{ U_i \in \sqrt{N}B \right\} \,, \label{eq:A_def} 
\end{align}
and write $A, E_\ell$ and $E_\ell^*$ respectively for the intersections of the events in \eqref{eq:A_def}, \eqref{eq:E_ell_def} and \eqref{eq:E_ell_star_def} over $i = 1,2$. Note that the event $\Eind_\ell $ from \eqref{eq:our_Kac_Rice_output} and $E_\ell$ are the same event. 

Following \cite{subag2017complexity}, bound the output of Lemma~\ref{lem:our_Kac_Rice} using H\"older's inequality, splitting terms using the truncation functions defined presently. For $\kappa > \eps >0$, write $h_\eps(x) \tri \max(x,\eps)$, and write
\begin{align}
h_\eps^\kappa(x) \tri 
\begin{cases}
 \eps & \quad x < \eps \\
 x & \quad x \in [\eps, \kappa] \\
 1 & \quad x > \kappa
 \end{cases}
 \quad \text{ and } \quad 
 h_\kappa^\infty(x) \tri
 \begin{cases}
 1 & \quad x \leq \kappa \\
 x & \quad x > \kappa 
 \end{cases} \,,
 \label{eq:h_def}
\end{align}
so that $h_\eps^\kappa(x) h_\kappa^\infty(x) \equiv h_\eps(x)$. For $i=1,2$, define the random variables
\begin{align}
W_i \equiv W_i(r)  &\tri \left( 2 \sum_{j=1}^{N-2}\lbob\un{\matM}_{\,i} \rbob_{j,\,N-1}^2 + \lbob \un{\matM}_{\,i} \rbob_ {N-1,\,N-1}^2 \right)^{1/2}  \non \\
&= \left( 2 \left\| Z_i \right\|^2 + \left(Q_i -\bar{U}_i + m^\circ_{\,i} \right)^2 \right)^{1/2} .
\label{eq:W_def}
\end{align}

Let $2 \leq m \in \N$, and let $q \equiv q(m) = m / (m-1)$ be the H\"older conjugate of $m$, and define:
\begin{align}
\mathcal{E}^{1}(r) \equiv \mathcal{E}_{\eps,\, \kappa}^{(1)}(r) &\tri \E \left( \prod_{i=1,2} \prod_{j=1}^{N-2} \left[h_\eps^\kappa \left( \lambda_j \left(\un{\matG}_{\,i} \right) \right) \right]^q \cdot \1 \left\{A \cap E_\ell \right\} \right) \label{eq:calE1_def}\\
\mathcal{E}^2(r) \equiv \mathcal{E}_{\eps,\, \kappa}^{(2)}(r) &\tri  \E \left( \prod_{i=1,2} \prod_{j=1}^{N-2} \left[h_\kappa^\infty\left( \lambda_j\left(\un{\matG}_{\,i}\right) \right) \right]^{2m}  \right)  \\
\mathcal{E}^3(r) \equiv \mathcal{E}_{\eps,\, \kappa}^{(3)}(r) &\tri \E \left[ \frac{ W_1(W_1 +\eps) }{ \eps } \right]^{4m} \E \left[ \frac{W_2 (W_2 + \eps] }{\eps} \right)^{4m} \,. 
\end{align}

\begin{coro} For $q,m$ and $\mathcal{E}^1$ above, we have:
\begin{align}
\limsup_{N \to \infty} & \frac{1}{N} \log \E [ \crit_{N,\ell}(B,I_R) ]_2 \non \\
&\leq 1 + \log(p-1) +  \limsup_{N\to \infty} \frac{1}{qN} \log \left( \int_{-r_0}^{\,r_0} \left(\mathcal{G}(r) \right)^{qN} \mathcal{E}^1(r) dr \right) \,.
\label{eq:Thm5_2}
\end{align}
\label{coro:thm5_step1}
\end{coro}

\begin{proof} We use \cite[Lemma 14]{subag2017complexity} and H\"older on the expectation inside the integral in \eqref{eq:our_Kac_Rice_output}:
\begin{align}
\E \left( \prod_{i=1,2} \left| \det \left(\un{\matM}_{\,i} \right) \right| \1\{A\cap E_\ell\} \right) \leq \left( \mathcal{E}^1 \right)^{1/q}  \left( \mathcal{E}^2 \right)^{1/2m}  \left( \mathcal{E}^3 \right)^{1/4m} \,,
\label{eq:thm5_step1_1}
\end{align} 
and we insert into the integral in \eqref{eq:our_Kac_Rice_output} to obtain
\begin{align}
\limsup_{N \to \infty} & \frac{1}{N} \log \E [ \crit_N(B,I_R) ]_2 \non\\
& \quad \leq \limsup_{N \to \infty} \frac{1}{N} \log C_N + \limsup_{N \to \infty} \frac{1}{qN} \log \left( \int_{-r_0}^{\,r_0} \left(\mathcal{G}(r) \right)^{qN} \mathcal{E}^1(r) dr \right) \label{eq:Thm5_0}\\
& \quad\quad\quad\quad+ \limsup_{N \to \infty} \frac{1}{mN} \log \left( \int_{-r_0}^{\,r_0} \left( \mathcal{F}(r) \right)^{m} \left( \mathcal{E}^2(r) \right)^{1/2} \left( \mathcal{E}^{3}(r) \right)^{1/4} dr \right) \,. 
\label{eq:Thm5_1}
\end{align}
It is straightforward to show the first summand in \eqref{eq:Thm5_0} is $1 + \log (p-1)$. Using Lemma 15 and Lemma 16 (ii) in \cite{subag2017complexity}, the term in \eqref{eq:Thm5_1} is zero for $\kappa$ large, and the proof is complete.\end{proof}

We next state a consequence of Proposition~\ref{prop:index_transfer}. 

\begin{coro} Let $0 < r_0 < 1$, and let $\sfr$ be a uniform random variable over the interval $(-r_0,r_0)$. Conditionally on $\sfr$, let $(U_1(\sfr),U_2(\sfr)) \sim \mathcal{N}( 0, \bm{\Sigma}_U(\sfr) )$, with $\bm{\Sigma}_U(\sfr)$ from \eqref{eq:Sigma_U}. Conditionally on $\sfr, U_1(\sfr)$ and $U_2(\sfr)$, define the Hessian matrices $\un{\matM}_{\,i}$ as above. Writing $\bar{U}_i(\sfr)$ for $\gamma_pU_i(\sfr)/ \sqrt{N-1}$, let $E_{\downarrow} \equiv E_{\downarrow,\,N}$ denote the event $\{ \bar{U}_1(\sfr), \bar{U}_2(\sfr) < -2\}$, and let $D_\ell$ denote the event $E_\ell^* \,\Delta \,E_\ell$. Then, as $N \to \infty$,
\begin{align*}
\1 \left\{ D_\ell \cap E_\downarrow \right\} \to 0\, \quad \text{ a.s. }
\end{align*}
with respect to the randomness of $\sfr$, the pair of energies $(U_1(\sfr),U_2(\sfr))$ and the additional randomness used to define the $\un{\matM}_{\,i}$. 
\label{coro:index_transfer_thm5}
\end{coro}

\begin{comment}
Consider the families of random matrices 
\begin{align}
\left\{ M^i(r,u_1,u_2) : r \in (-1,1) \text{\rm\, and } u_1,\,u_2 \in (-\infty,-2) \right\} \,,\\
\left\{ G^i(r,u_1,u_2) : r \in (-1,1) \text{\rm\, and } u_1,\,u_2 \in (-\infty,-2) \right\} \,,
\end{align}
where above, $M^i(r,u_1,u_2) \equiv M^{i}_{N-1}(r,u_1,u_2)$ and $G^i \equiv \hat{G}_{N-2}^i(r) - u_iI$ are defined according to \eqref{eq:M_joint} and \eqref{eq:GOE_joint}. There is a random, almost surely finite $N_0$ so that $N \geq N_0$ implies $\ind(M^i) = \ind(G^i)$ holds simultaneously for $i = 1,2$ and for all $r \in (-1,1)$ and $u_1,u_2 \in (-\infty, -2)$. 
\end{comment}

\begin{proof} Introduce the sequence $(\eps_k)_{k\geq 2}$ defined by $\eps_k \tri 2^{-k}$, noting that the events 
\begin{align*}
E_{\downarrow,\,k} \tri \left\{ - \eps_k^{-1} \leq \bar{U}_1(\sfr), \bar{U}_2(\sfr) \leq -2 -\eps_k \right\} 
\end{align*}
are nested and exhaust $E_\downarrow \equiv \{  \bar{U}_1(\sfr), \bar{U}_2(\sfr) < -2  \}$. Applying Proposition~\ref{prop:index_transfer} to this $\eps_k$, note that as $N \to \infty$, $\1 \{ D_\ell \cap E_{\downarrow,\,k} \} \to 0$ almost surely, completing the proof. 
\end{proof}

The last ingredient needed before starting the proof of Theorem~\ref{thm5} is a bound on the term $\mathcal{E}^1$ appearing in \eqref{eq:Thm5_2}. The bound is given in terms an approximate of the function $\Omega$ from \eqref{eq:Omega_def}. Define
\begin{align}
\Omega_\eps^\kappa(x) \tri \int \log_\eps^\kappa ( |\lambda - x | ) \semi(d\lambda)  \label{eq:omega_trunc_def}\,,
\end{align}
with $\log_\eps^\kappa(x) \tri \log ( h_\eps^\kappa(x) )$. 

\begin{lem} Let $\delta >0$. For any $q > 0$ and nice $B$, there is $c(\delta) > 0$ so that 
\begin{align}
\mathcal{E}^{1}(r) &\leq \exp \left(-cN^2 \right) \non \\
&\quad + \E \left( \exp \left( qN \sum_{i=1,2} \Omega_{\eps}^{\kappa} \left( \bar{U}_i \right) + 2q\delta N \right) \cdot \1\{A \cap E_\ell \} \right)\,. \non
\end{align}
\label{lem:16i}
\end{lem}

\begin{proof} Lemma~\ref{lem:16i} is essentially Lemma 16 (i) in \cite{subag2017complexity}. The same proof, which uses the LDP Theorem~\ref{thm:spectral_LDP}, goes through. The truncation enables the application of this LDP, as $x \mapsto \log_\eps^\kappa(x)$ is Lipschitz.  \end{proof}

\subsection{Proof of Theorem~\ref{thm5}}{\label{sec:4_1}} As discussed, the proof of Theorem~\ref{thm5} also uses the large deviation principle Theorem~\ref{coro:LDP_ell}. For the associated rate functions $\mathscr{I}_r, J$, and $J_r$ in \eqref{eq:leading_rate_1}, \eqref{eq:leading_rate_2}, and \eqref{eq:leading_rate}, we adopt the convention that $\mathscr{I}_r( -u_1, -u_2) \tri \mathscr{I}_r(u_1, u_2)$ for $u_1, u_2 > 2$, and likewise for $J$ and $J_r$. 

Below we suppose, to avoid redundancy with \cite{subag2017complexity}, that $\ell \geq 1$. We also suppose for simplicity that $B$ is an interval. Lastly, we recall the notation $\gamma_p$ introduced in \eqref{eq:gamma_def} for $2/\sfE_\infty$, and we note that under the hypotheses of Theorem~\ref{thm5}, for $\gamma_p B \tri \{ \gamma_p x : x \in B\}$, 
\begin{align}
\gamma_p B \subset (-\infty, -2) \,. \label{eq:thm5_good_containment}
\end{align}

\emph{Step 1: bounding the output of Kac-Rice.} We take Corollary~\ref{coro:thm5_step1} one step further. Consider the integral in \eqref{eq:Thm5_2}, and write 
\begin{align}
\int_{-r_0}^{\,r_0} \left( \mathcal{G}(r) \right)^{qN} \mathcal{E}^{\,1}(r) dr &= 2r_0 \E  \left[ \left( \mathcal{G}(\sfr) \right)^{qN} \mathcal{E}^1(\sfr)  \right],
\label{eq:Thm5_2.5}
\end{align}
having reinterpreted the integral on the left as expectation in a random variable $\sfr$ uniformly distributed over $(-r_0,r_0)$, and independent of the other variables present. \\

\emph{Step 2: transferring the index.} Use Lemma~\ref{lem:16i} on \eqref{eq:Thm5_2.5}:
\begin{align}
\int_{-r_0}^{\,r_0}  \left[ \mathcal{G}(r) \right)^{qN} \mathcal{E}^{1}(r) dr \leq 2r_0 e^{2\delta qN} \E \left( \exp \left( qN \psi\left(\sfr, \bar{U}_1, \bar{U}_2 \right)  \right) \cdot \1 \{ A \cap E_\ell \}  \right] +  \exp(-cN^2) \,,
\label{eq:plus_def}
\end{align}
where the function $\psi$ is defined by
\begin{align}
\psi(r, \bar{u}_1, \bar{u}_2 )  \equiv \psi_\eps^\kappa(r, \bar{u}_1, \bar{u}_2 ) \tri \log \mathcal{G}(r) + \sum_{i=1,2} \Omega_\eps^\kappa (\bar{u}_i) \,.
\end{align}

\begin{comment} In anticipation of applying Proposition~\ref{prop:index_transfer} to the event $E_\ell$ inside \eqref{eq:Thm5_2.5}, introduce the event
\begin{align}
F^{\,+} \equiv F^{\,+}(\eta) \tri \{ \bar{U}_i \leq -(2 + 2\eta) \text{ for } i= 1,2 \}\,,
\end{align}
where $\eta > 0$ is small, to be chosen later. Let $F^{\,-} \equiv F^{\,-}(\eta)$ denote the complement of this event, and decompose $\mathcal{E}^{1}$ as the sum $\mathcal{E}^{\,+} + \mathcal{E}^{\, - }$, with
\begin{align}
\mathcal{E}^{\,\pm}(r) \equiv \mathcal{E}_{\eps, \,\kappa,\eta}^{\,\pm}(r) &\tri \E \left( \prod_{i=1,2} \prod_{j=1}^{N-2} h_\eps^\kappa(\lambda_j(G^i))^q \1\{A \cap E_\ell \cap F^{\,\pm}\} \right) \,. \label{eq:cal_E_pm} 
\end{align}
\end{comment}
As in the statement of Corollary~\ref{coro:index_transfer_thm5}, write $D_\ell$ for the symmetric difference $E_\ell^*\, \Delta \, E_\ell$, and write $E_\downarrow = \{ \bar{U}_1, \bar{U}_2 < -2 \}$. By \eqref{eq:thm5_good_containment}, we have $A \subset E_\downarrow$. We use this containment, starting from \eqref{eq:plus_def}, to control the limsup on the right of \eqref{eq:Thm5_2}:
\begin{align}
\limsup_{N \to \infty} &\frac{1}{qN} \log\, \int_{-r_0}^{\,r_0}  \left(\mathcal{G}(r) \right)^{qN} \mathcal{E}^{1}(r) dr - 2\delta \non  \\
&\leq \limsup_{N\to \infty} \frac{1}{qN} \log  \E \left[ \exp \left( qN \psi \left(\sfr, \bar{U}_1, \bar{U}_2 \right)  \right) \cdot (\1 \{ A \cap E_\ell^* \}  + \1 \{ D_\ell \cap E_\downarrow \} )\right] \label{eq:Thm5_3pre} \\
&\leq\limsup_{N\to \infty} \frac{1}{qN} \log  \E \left[ \exp \left( qN \psi \left(\sfr, \bar{U}_1, \bar{U}_2 \right)  \right) \cdot \1 \{ A \cap E_\ell^*  \}   \right]  \label{eq:Thm5_3}.
\end{align}
Going from \eqref{eq:Thm5_3pre} and \eqref{eq:Thm5_3} above, we have used that Corollary~\ref{coro:index_transfer_thm5} implies $\1 \{ D_\ell \cap E_\downarrow \}$ tends almost surely to zero as $N \to \infty$, in addition to the fact that $\mathcal{G}(r)$ is bounded above by one, uniformly in $r$, and that $\Omega_\eps^\kappa \leq \log \kappa$. \\

\emph{Step 3: using the eigenvalue LDP.} Clearly,
\begin{align}
E_\ell^* \subset \left\{ \lambda_\ell( \matG_i) \leq \bar{U}_i \text{ for } i = 1,2 \right\} \tri A_\goe .
\end{align}

Unfold the expectation in \eqref{eq:Thm5_3} and use the above containment to bound \eqref{eq:Thm5_3} from above by
\begin{align}
\E_{\,\sfr} \E_{\,\pair} \left[  \left( \exp (qN \psi \left(\sfr, \bar{U}_1, \bar{U}_2 \right)  \right) \cdot \1\{A\} \cdot \E_{\,\goe} \left( \1\{A_\goe\} \big|\, \sfr, \bar{U}_1, \bar{U}_2 \right)  \, \big|\, \sfr  \right] ,
\label{eq:Thm5_4}
\end{align}
where $\E_{\,\sfr}$, $\E_{\,\pair}$ and $\E_{\,\goe}$ denote expectation taken in $\sfr$, the pair $(\bar{U}_1, \bar{U}_2)$ and $(\matG_1, \matG_2)$ respectively. For $\delta >0$ above, take $N$ large enough to apply Theorem~\ref{coro:LDP_ell} to the pair $(\lambda_\ell(\matG_1), (\lambda_\ell(\matG_2))$:
\begin{align}
\limsup_{N \to \infty} &\frac{1}{qN} \log\, \int_{I_R}  \left(\mathcal{G}(r) \right)^{qN} \mathcal{E}^{1}(r) dr  - 2\delta \non\\
\begin{split}
 \leq\limsup_{N \to \infty}\frac{1}{qN} \log \E_{\,\sfr} \left[ \E_{\,\pair}  \left( \exp \left( qN \psi \left(\sfr ,\bar{U}_1, \bar{U}_2 \right) \vphantom{\int}  \right.\right. \right. \quad\quad\quad\quad\quad\quad\quad\quad\quad\quad\quad&\\ %phantom 
 \left.\left.\left. - N\inf_{z_i \,\leq\, \bar{U}_i} \mathscr{I}_\sfr^{(\ell)}( z_1, z_2) + \delta N  \right)\cdot \1\{A\} \Big| \sfr \right)\right].
 \end{split} \label{eq:Thm5_5}
\end{align}
Define
\begin{align*}
\xi(r,\bar{u}_1, \bar{u}_2) \equiv \xi_\eps^\kappa(r,\bar{u}_1, \bar{u}_2) \tri \psi(r,\bar{u}_1,\bar{u}_2) - \frac{1}{q} \inf_{z_i \,\leq\, \bar{u}_i} \mathscr{I}_r^{(\ell)}(z_1, z_2 )  \,,
\end{align*}
so that \eqref{eq:Thm5_2} and \eqref{eq:Thm5_5} together yield the following bound. 
\begin{align}
\limsup_{N\to \infty} &\frac{1}{N} \log \E [ \crit_{N,\ell}(B,I_R) ]_2\\
&\leq 1 + \log(p-1) \,+ 3\delta + \limsup_{N\to \infty}\frac{1}{qN} \log \E \left[ \exp \left( qN \xi \left(\sfr ,\bar{U}_1, \bar{U}_2 \right) \right) \cdot \1\{ A \}   \right].
\label{eq:Thm5_6}
\end{align}
We handle the above expression using Varadhan's integral lemma after making a change of coordinates.\\

\emph{Step 4: change of coordinates.}  For $(\wt{U}_1, \wt{U}_2)$ a pair of independent standard normals,
\begin{align}
(U_1(r), U_2(r) ) =_d (\wt{U}_1, \wt{U}_2) \bm{\Sigma}_U(r)^{1/2} \,.
\label{eq:Thm5_u_tilde}
\end{align}
Let  $Y_r$ be the linear transformation acting from the right on the triple $(r,\tilde{u}_1, \tilde{u}_2)$, viewed as a row vector, which acts as the identity on the first coordinate and via
\begin{align}
(\tilde{u}_1, \tilde{u}_2) \mapsto  (\tilde{u}_1, \tilde{u}_2) \bm{\Sigma}_U(r)^{1/2} \,
\end{align}
on the latter two coordinates. Define
\begin{align}
T(B) = Y_r^{-1}((-r_0,r_0)\times B \times B ))\,,
\end{align}
and for $i = 1,2$, make the abbreviation $\wt{V}_i \equiv  \wt{U}_i / \sqrt{N}$
so that  $\{(\sfr, \wt{V}_1,\wt{V}_2) \in T(B)\}$ corresponds to the event $A$ under this change of coordinates. Let $\tilde{\xi}$ denote the function $\xi$ under the same change of coordinates: 
\begin{align}
\tilde{\xi}(r,\til{u}_1,\til{u}_2) \tri \xi \circ Y_r \left( r,\til{u}_1,\til{u}_2\right) \,. 
\label{eq:Thm5_zeta_tilde}
\end{align}
Returning to \eqref{eq:Thm5_6}, make this change of variables:
\begin{align}
\limsup_{N\to \infty} &\frac{1}{N} \log \E [ \crit_{N,\ell}(B,I_R) ]_2\\
&\leq 1 + \log(p-1) \,+ 3\delta \\
&\quad \limsup_{N\to \infty}\frac{1}{qN} \log \E\left[  \exp \left( qN \til{\xi} \left(\sfr ,\frac{\gamma_p\wt{U}_1}{\sqrt{N-1}}, \frac{\gamma_p\wt{U}_2}{\sqrt{N-1}} \right)  \right) \cdot \1\left\{ \left(\sfr, \wt{V}_1,\wt{V}_2 \right) \in T(B)\right\}\right],
\end{align}
where we emphasize that the $\wt{V}_i$ are just multiples of the $\wt{U}_i$. The triple of random variables $(\sfr , \wt{U}_1 / \sqrt{N-1}, \wt{U}_2/ \sqrt{N-1} )$ is exponentially equivalent to the triple $(\sfr, \wt{U}_1/ \sqrt{N}, \wt{U}_2/\sqrt{N})$, and the latter satisfies a large deviation principle with good rate function 
\begin{align}
J_0(r, \til{u}_1, \til{u}_2) = \til{u}_1^2/ 2 + \til{u}_2^2 / 2 \,,
\end{align}
and this LDP we apply Varadhan's lemma to:
\begin{align}
\limsup_{N\to \infty} &\frac{1}{N} \log \E [ \crit_{N,\ell}(B,I_R) ]_2 \label{eq:Thm5_8}  \\
&\leq 1 + \log (p-1) + 3\delta \\
&\quad + \frac{1}{q} \sup_{ (r,\, \,\til{u}_1,\, \,\til{u}_2)\, \in\, T(B) } \left( q \til{\xi} (r, \til{u}_1, \til{u}_2 ) - \frac{\til{u}_1^2}{2} - \frac{ \til{u}_2^2}{2} \right) \,. 
\end{align}

\emph{Step 5: finishing the proof.} To complete the proof, we first take $m \to \infty$ so that $q \to 1$, and we undo the change of variables \eqref{eq:Thm5_u_tilde} in \eqref{eq:Thm5_zeta_tilde}.  This yields the following upper bound on \eqref{eq:Thm5_8}:
\begin{align}
\begin{split}
1 + \log(p-1) + 3\delta + \sup_{r\, \in\, (-r_0,\,r_0)} \sup_{u_1,\,u_2\, \in\, B} \left(\xi(r,\gamma_p u_1, \gamma_p u_2)  - \frac{1}{2} \lbob
				\begin{matrix}%  
					u _ 1 & u _ 2 
				\end{matrix}%
			\rbob 
			\bm { \Sigma } _ U ( r ) ^ { - 1 }  
			\lbob 
				\begin{matrix}%
					u _ 1												 \\ %
					u _ 2 
				\end{matrix} %
			\rbob \, \right)\,,
\end{split} 
\end{align}
which we write below more explicitly:
\begin{align}
\begin{split}
1 + \log(p-1) + 3\delta \hspace{70mm}& \\
+\sup_{r \in (-r_0, \, r_0)} \sup_{u_1,\,u_2\, \in\, B} \left( \log\mathcal{G}(r) + \sum_{i=1,\,2} \Omega_\eps^\kappa( \gamma_p u_i)  \right. \hspace{30mm}&\\
\left. \quad- \inf_{z_i \, \leq\, \gamma_p u_i} \mathscr{I}_r^{(\ell)}(z_1, z_2) -  \frac{1}{2} \lbob
				\begin{matrix}%  
					u _ 1 & u _ 2 
				\end{matrix}%
			\rbob 
			\bm { \Sigma } _ U ( r ) ^ { - 1 }  
			\lbob 
				\begin{matrix}%
					u _ 1												 \\ %
					u _ 2 
				\end{matrix} %
			\rbob \, \right)\,.
\end{split} 
\label{eq:Thm5_9}
\end{align}
Recalling the expression for $\Psi_{p,\ell}$ from \eqref{eq:bounding}, we take $\kappa \to \infty$ and $\delta, \eps \to 0$ in \eqref{eq:Thm5_9}:
\begin{align}
\label{eq:Thm5_10}
\limsup_{N\to \infty} &\frac{1}{N} \log \E [ \crit_{N,\ell}(B, (-r_0,r_0)) ]_2 \leq \sup_{
	\substack{ r\, \in\, (-r_0,\,r_0),\\ u_1,\,u_2\, \in\, B
	}
	} \Psi_{p,\ell} (r,u_1,u_2) \,.
\end{align}
Above, we used Lemma~\ref{lem:rfm_diag} to remove the infimum in going from \eqref{eq:Thm5_9} to \eqref{eq:Thm5_10}, and the proof is complete. \qed

\section{\large Moment matching and analysis of the variational problem}\label{sec:exp_match}

The main task of this section is to prove that, at the exponential scale, the first and second moments of $\crit_{N,\ell}((-\infty,u))$ match. To this end, we analyze the supremum on the right-hand side of 
\eqref{eq:thm5_statement} with three lemmas similar to \cite[Lemma 6, Lemma 7]{
	subag2017complexity
 }. 
 
Our bounding function is maximized when the two energy parameters 
$ u _ 1 $ and 
$ u _ 2 $ are balanced.
\begin{lem}								%%%%%%%%%
\label{coro:lem6}%							coro:lem6
For nice 
$ B 
	\subset 
		( 
			- \infty , - \sfE _ \infty 
		) $, and for any fixed 
$ r \in ( - 1 , 1 ) $, we have
\begin{align}								%%%
\label{eq:lem6_1}%							eq:lem6_1
\sup _ 
	{ u _ 1 , \, u_2 \, \in \, B } 
		\Psipell 
		(
			r , u _ 1 , u _ 2 
		) 
			\, = \, 
				\sup _ { u \, \in \, B } \,
					\Psipell
					(
						r , u , u 
					) .
\end{align}								%%%
Moreover, for such 
$ B $,
\begin{align}								%%%
\label{eq:lem6_2}%							eq:lem6_2
\limsup _ 
	{ N \to \infty } 
	\frac { 1 } 
		{ N } 
		\log 
		\E 
			\leftcb 
			\critNell
				( \,
					B , \, ( - 1 , 1 ) 
				\, ) 
			\rightcb _ { \bm { 2 } } 
				\, \leq 
					\sup _ 
						{ 
						\substack{
							r \, \in \, ( - 1 , \, 1 ) \\
							u \, \in \, B
							}
						} 
					 \Psipell 
					 (
					 	r , u , u 
					)  . 
\end{align}								%%%
\end{lem}									%%%%%%%%%
Along the diagonal of energy parameters, the bounding function is greatest (as a function of overlap parameter $r$) either when $ r 
	= 
		0 $, corresponding to orthogonal points on the sphere, or when $r =\pm1$, corresponding to parallel points on the sphere.

\begin{lem}								%%%%%%%%%
\label{lem:prelem7}%						lem:prelem7
For fixed $u \in (-\infty, -\sfE_\infty)$, the function 
$ r 
	\mapsto 
				\Psipell ( r , u , u ) $, can be extended to a continuous function 
$ r
	\mapsto 
				\Psipell ^ { \, \bullet } ( r , u , u ) $ on $[-1,1]$, and this extension is maximized in the set $\{-1,0,1\}$.
\end{lem}									%%%%%%%%%
The third lemma implies, for 
$ u $ in the range of energies relevant to index $\ell $ critical points, the bounding function is maximized at 
$ r 
	= 
	0 $. 
\begin{lem}								%%%%%%%%%
\label{lem:lem7_1}%							lem:lem7_1
We adopt the notation for the extension in Lemma~\ref{lem:prelem7}. 
\begin{itemize} %
\item[(i)] For any 
	$ u \in ( - \sfE _ \ell , - \sfE _ \infty ) $, one has 
	$ \Psipell ^ { \, \bullet } ( 1 , u , u ) 
		< 
			\Psipell ( 0 , u , u ) $. 
\item[(ii)] Moreover, for 
	$ r \in \{ 0 , 1 \} $, we have	
	$\sup _ 
		{ 
			u \,\in\, ( - \infty , \, - \tsfE _ \ell ]
		} 
		\Psipell(r, u, u) \leq 0. 
	$						%%%
\end{itemize} %
\end{lem}									%%%%%%%%%

We prove these lemmas in the next subsection and then use the lemmas to prove Theorem~\ref{thm:exp_match}. As in  Section~\ref{sec:exp_bounds}, we write 
$ \gamma _ p $ for 
$ 2 / \sfE_\infty $. As discussed just after \eqref{eq:u_bar_def}, we repurpose the bar notation in this section: for 
$ z \in \R $, write 
$ \bar{z} $ for 
$ \gamma_pz $. For
$ \bm{\Sigma} _ U ( r ) $ defined in
\eqref{eq:Sigma_U}, let
\begin{align}
\mathcal{H}(r,u_1,u_2) 
	\tri 
		\frac{1}{2} 
			\lbob \begin{matrix} u_1 & u_2 \end{matrix} \rbob  \bm{\Sigma}_U(r)^{-1} \lbob \begin{matrix} u_1 \\ u_2 \end{matrix} \rbob ,
\end{align}
and for brevity, define:
\begin{align}
\Psipell^{\,u}(r) &\tri \Psipell(r,u,u) \non \\
\mathcal{H}^u(r) &\tri \mathcal{H}(r,u,u) \equiv u^2 \frac{ 1-r^p +(p-1) r^{p-2} (1-r^2) }{1-r^{2p-2} + (p-1) r^{p-2} (1-r^2) } .
\label{eq:H_def}
\end{align}

\subsection{Maximizing the bounding function}

\begin{proof}[Proof of Lemma~\ref{coro:lem6}] 
The proof follows that of 
\cite[Lemma 6]{subag2017complexity}. For $\bar{u} = \gamma_pu$, when $u \in B \subset (-\infty, -\sfE_\infty)$, we have $\bar{u} \in (-\infty,-2)$. As $\log$ is concave, the function $x \mapsto \Omega(x)$ restricted to $(-\infty,-2)$ is concave.  Moreover, as $\bm{\Sigma}_U(r)^{-1}$ is positive definite for each $r \in (-1,1)$, the function 
\begin{align}
\Phi(r,u_1,u_2) \tri \sum_{i=1, \,2} \Omega( \bar{u}_i) - \mathcal{H}(r,u_1,u_2) 
\end{align}
is also concave. Given $u \in (-\infty, -\sfE_\infty)$, define the domain $D(u) \tri (\sfE_\infty + u, -\sfE_\infty -u)$. For $z \in D(u)$, note that $\Phi^*(z) \tri \Phi(r,u+z, u-z)$ is well-defined, and is concave as a function of $z$. As $\Phi$ is symmetric in $u_1$ and $u_2$, it follows that (writing $\pa_z$ for a derivative in $z$) $\pa_z \Phi^*(0) =0$, and hence
\begin{align}
\sup_{z \, \in \, D(u)} \Phi^* (z) = \Phi_{r,u}^*(0) \equiv \Phi_{p,\ell}(r,u,u) .
\end{align}
By Lemma~\ref{lem:rfm_diag},
\begin{align}
\sup_{z \, \in \, D(u)} - \mathscr{I}_r( u+z,u-z) = -\mathscr{I}_r(u,u) \,,
\end{align}
and as $r$ is fixed, $\Psi_{p,\ell}$ differs from $\Phi - \mathscr{I}_r$ by a constant, so \eqref{eq:lem6_1} follows from the containment
\begin{align*}
B \times B \subset \{(u,z) : u \in B,\, z \in D(u) \},
\end{align*}
while \eqref{eq:lem6_2} follows from \eqref{eq:lem6_1} and Theorem~\ref{thm5}. \end{proof}
\begin{comment}
Turning to \eqref{eq:lem6_2}, the above argument does not go through directly because we lose the concavity of $x \mapsto \Omega(x)$ for $x$ not in $(-\infty,-2)$. We circumvent this by subdividing $B$ into small pieces and using the uniform continuity of $\Omega$ on the interval $[-\sfE_0, \sfE_0]$, which of course contains $(-\sfE_\ell, 0)$. As $B$ has non-empty intersection with $(-\sfE_\ell, \sfE_\ell)$, we have that $\crit_{N,\ell}(B)$ diverges to infinity with $N$ by Theorem~\ref{thm:ABC}, which is the only requirement needed of $B$ for the same proof to hold.
\end{comment} 

\begin{rmk}Suppose $u$ is fixed, and view $\Psipell^{\,u}(r)$ as a function of $r$ only. We remark that the piecewise nature of $\mathscr{I}_r$ is inherited by the bounding function. When restricted to $[0,1)$, $r \mapsto \Psipell^{\,u}(r)$ is defined piecewise on the intervals $\left[0, r_*(u)\right] \cup \left[r_*(u),1\right)$, where 
\begin{align}
r_*(u) \tri v(\gamma_pu)^{-2/{(p-2)}}\,.
\label{eq:r_star}
\end{align}
Suppressing parameters in the notation, let $\psi_\perp$ denote \eqref{eq:bounding} with $\scrI_r$ replaced by $J$ from \eqref{eq:leading_rate_1}. Similarly, let $\psi_\parallel$ denote \eqref{eq:bounding} with $\scrI_r$ replaced by $J_r$ from \eqref{eq:leading_rate_2}. We summarize the above by writing
\begin{align}
\label{eq:Psi_cases_2}
\Psipell^{\,u}(r) = 
\begin{cases}
\psi_\perp (r) &\quad 0 \leq r \leq r_*(u) \,   \\
\psi_\parallel (r) &\quad r_*(u) \leq r < 1.
\end{cases}
\end{align}
\end{rmk}

\begin{proof}[Proof of Lemma~\ref{lem:prelem7}] Consider the term $\mathcal{H}^u(r)$ in $\Psipell^{\,u}$. Inspecting the fraction in \eqref{eq:H_def}, and accounting for the parity of $p$, observe that $\mathcal{H}^u(r)$ is an even function of $r$ if and only if $p$ is even. In the case that $p$ is odd, and for any $r \in (-1,0)$, we see that $\mathcal{H}^u(r)$ is strictly larger than 1, while for $r \in (0,1)$, it is strictly less than one. It thus suffices to consider $\Psipell^{\,u}$ restricted to  the interval $[0,1)$.

We now wish to extend \emph{both} $\psi_\perp$ and $\psi_\parallel$ to continuous functions on $[0,1]$. In each case, the only obstruction comes from the terms
\begin{align}
\log \mathcal{G}(r) - \mathcal{H}^u(r) .
\label{eq:prelem7_1}
\end{align}
These terms are easily extended to the endpoint $r = 1$ via L'H\^opital's rule. This is done more explicitly in the proof of Lemma~\ref{lem:lem7_1}, so we consider the first claim of the Lemma~\ref{lem:prelem7} settled, denoting the extension of $\Psipell^{\,u}$ by $\Psi^\bullet$. Let us also use $\psi_\perp^\bullet$ and $\psi_\parallel^\bullet$ to denote the continuous extensions of $\psi_\perp$ and $\psi_\parallel$ to $[0,1]$.  

Because the rate function $J$ in \eqref{eq:leading_rate_1} does not depend on $r$, the $r$-derivative of $\psi_\perp^\bullet$ agrees with that of the bounding function considered in \cite[Lemma 7]{subag2017complexity}. From this lemma, we know the latter function attains its maximum over $[0,1]$ within $\{0,1\}$. It follows that the maximum of $\psi_\perp^\bullet$ over $[0,1]$ is attained at either $0$ or $1$. 

By Lemma~\ref{lem:rate_equal} we have $\psi_\perp(r) = \psi_\parallel(r)$ at $r = r_*(u)$. Moreover, Lemma~\ref{lem:rate_fn_dom} implies that $\psi_\perp^\bullet (r) \leq \psi_\parallel^\bullet(r)$ on the interval $[r_*(u),1]$. The proof of Lemma~\ref{lem:rate_fn_dom} implies $r \mapsto -J_r(u,u)$ is strictly increasing on this interval, so the proof is complete. To see that the argument $r_*(u)$ does not need to be considered, note that  $\Psi^\bullet (r_*(u)) = \psi_\perp^\bullet (r_*(u))$, and therefore a maximum at $r_*(u)$ would contradict \cite[Lemma 7]{subag2017complexity}.    \end{proof}

\begin{proof}[Proof of Lemma~\ref{lem:lem7_1}] We lose no generality supposing $\ell \geq 1$. Fix $u \in (-\sfE_\ell, -\sfE_\infty)$ and write $\Psipell^{\,u}$ as $\Psi$ throughout the proof. We let $\Psi^\bullet$ denote the continuous extension of $\Psi$ furnished by Lemma~\ref{lem:prelem7}. Because of this lemma, our task is to compare $\Psi^\bullet(0)$ and $\Psi^\bullet(1)$. Plugging in $r =0$, we find
\begin{align*}
\Psi^\bullet(0) = 1 + \log(p-1) + 2 \Omega(\bar{u} ) - \ell 2 I_1(\bar{u};1) - u^2\,. 
\end{align*}
To compute $\Psi^\bullet(1)$, we use L'H\^opital's rule twice:
\begin{align*}
\lim_{r \to 1} \log \mathcal{G}(r) &=\frac{1}{2} \log\left( \frac{1}{p-1} \right) \\
\lim_{r\to 1} \mathcal{H}^u(r)  &=  \left(\frac{3p -2}{4(p-1)}\right) u^2\,.
\end{align*}
As $\lim_{r\to 1} J_r( \bar{u}, \bar{u}) = I_1(\bar{u};1)$, 
\begin{align*}
\Psi^\bullet (1) = 1 + \log(p-1) + 2 \Omega(\bar{u}) - \ell I_1(\bar{u} ;1 ) + \frac{1}{2} \log \left( \frac{1}{p-1} \right) -  \left( \frac{3p-2}{4(p-1) } \right) u^2.
\end{align*}
Observe that the difference between these two values can be expressed using the complexity function $\Sigma_\ell$:
\begin{align}
\Psi^\bullet(1) - \Psi^\bullet(0) &= \ell I_1(\bar{u} ; 1) - \frac{1}{2} \log \left( p-1\right) + \frac{p-2}{4(p-1)} u^2 
\non \\
& = - \Sigma_\ell(u) - I_1(\bar{u};1) ,
\label{eq:lem7_1}
\end{align}
and \eqref{eq:lem7_1} is strictly negative for $u \in (-\sfE_\ell, -\sfE_\infty)$ by definition of $\sfE_\ell$.  This gives us the first statement in the lemma. 

To complete the proof, note that for all $u < -\sfE_\infty$, $\Psi^\bullet(0) \equiv 2\Sigma_{\ell}(u)$. Combining this observation with \eqref{eq:lem7_1}, one has for such $u$
\begin{align}
\Psi^\bullet (1) = \Sigma_{\ell}(u) - I_1(\bar{u};1) \equiv \Sigma_{\ell-1}(u) \,,
\label{eq:lem7_2}
\end{align}
using the hypothesis $\ell \geq 1$. 

Display \eqref{eq:lem7_1} implies $\Psi^\bullet(1) \leq \Psi^\bullet(0)$ on the interval $[-\sfE_{\ell-1} , -\sfE_\infty)$, but we know that $\Psi^\bullet(0)$ is negative on the interval $[-\sfE_{\ell-1}, -\sfE_\ell)$ because the complexity function $\Sigma_{p,\ell}$ is negative there. The proof is complete, as \eqref{eq:lem7_2} implies $\Psi^\bullet (1)$, which is greater than $\Psi^\bullet(0)$ on the interval $(-\infty, \sfE_{\ell-1})$, is also negative on this interval. \end{proof}

\subsection{Proofs of Theorem~\ref{thm:exp_match} and Corollary~\ref{coro:cor8}}

\begin{proof}[Proof of Theorem~\ref{thm:exp_match}]%*******************************************%%%%%%%%%
By 
\eqref{eq:thm37start}, and because the complexity function 
$ \Sigma _ { p , \, \ell } ( u ) $ is strictly positive for the energies $u$ considered, it will suffice to show
\begin{align*}								%%%
\limsup _ 
	{ N \to \infty } 
		\frac { 1 } 
			{ N } 
			\E 
			\leftcb
				 \crit_{N,\ell}( \,
				 	(-\infty,u), (-1,1) 
				\, ) 
			\rightcb _ { \bm { 2 } } 
				\leq 
					\Psipell
					(
						0 , u , u
					) \,.
\end{align*}								%%%
Theorem
~\ref{thm5} bounds the left-hand side, in the display directly above, in terms of a supremum over the bounding function $\Psipell$. Lemma~\ref{coro:lem6}, Lemma~\ref{lem:prelem7} and Lemma
~\ref{lem:lem7_1} (i) then imply
\begin{align*}								%%%
\limsup _ 
	{ N \to \infty } 
		\frac { 1 } 
			{ N } 
			\E 
				\leftcb 
					\critNell
						( \, 
							( - \infty , u ) , ( - 1 , 1 ) 
						\, ) 
				\rightcb _ { \bm { 2 } } 
				\, \leq \,
					\leftp \,
						\sup _ 
							{ 
								\substack{ 
									r \, \in \, \{ \, 0 , \, 1 \, \} \\%
									 v \, \in \, ( - \infty, \, - \tsfE _ \ell )
								}
							} 
								\Psipell ( 
									r , v , v 
								) 
					\rightp 
						\vee 
							\leftp 
								\sup _ 
									{ \, v \,  \in \, ( - \tsfE _ \ell , \, u \, ) } 
										\Psipell (
											0 , v , v 
										) 
							\rightp ,
\end{align*}								%%%
and the proof is complete applying Lemma
~\ref{lem:lem7_1} (ii) to the left-most supremum, recalling that over the interval 
$ ( - \sfE _ \ell , - \sfE _ \infty ) $,
the complexity functions 
$ u \mapsto \Sigma _ { p \, , \ell } $ are strictly increasing. As in the proof of Lemma~\ref{lem:prelem7}, symmetry and parity considerations show it suffices to consider only $r \in \{ 0 , \, 1 \}$ in the above supremum, versus
$ r \in \{ -1 , 0 , 1 \} $. 
\end{proof}%**********************************************%%%%%%%%%

\begin{proof}[Proof of Corollary~\ref{coro:cor8}] As in \cite{subag2017complexity}, let us define the function
\begin{align}
g(r) = (1- u^{-2} \mathcal{H}^u(r) ) \equiv  \frac{r^p - r^{2p-2} }{1 - r^{2p-2} + (p-1) r^{p-2} (1-r^2) } \,,
\label{eq:little_g_def}
\end{align}
which will be notationally convenient at the end of the proof.

Consider $u \in (-\sfE_\ell, -\sfE_\infty)$, and let $\eps >0$. Write $I_\eps$ for the set $(-1,1) \setminus (-\eps, \eps)$ and apply Theorem~\ref{thm5} for $B =(-\infty, u)$ and $I_R = I_\eps$. 
\begin{align}
\limsup_{N \to \infty} \frac{1}{N} \log \E [ \crit_{N,\,\ell} ( (-\infty, u ), I_\eps ) ]_2 &\leq \sup_{r \, \in\, I_\eps} \sup_{u_1,\,u_2 \in (-\infty,\, u)} \Psi_{p,\ell} (r, u_1, u_2) \\
&\leq \sup_{r \,\in\, I_\eps} \sup_{w \,\in \,(-\infty, u) } \Psi_{p,\ell}^w(r) \,,
\label{eq:coro8_1}
\end{align}
where Lemma~\ref{coro:lem6} implying second line. From the proof of Lemma~\ref{lem:lem7_1}, we have
\begin{align}
\un{\Psi}_{p,\ell}^w (1) \leq \Psi_{p,\ell}^w(0) \quad \text{ if and only if $u \in [-\sfE_{\ell-1}, -\sfE_\infty )$ } \,, 
\end{align}
which follows from  \eqref{eq:lem7_1}, and additionally,
\begin{align}
\un{\Psi}_{p,\ell}^w(1) \leq 0 \quad \text{ for $w \in (-\infty, -\sfE_{\ell-1} )$} \,,
\end{align}
which follows from \eqref{eq:lem7_2}. Applying these observations in \eqref{eq:coro8_1}, we have
\begin{align}
\limsup_{N \to \infty} \frac{1}{N} \log \E [ \crit_{N,\ell} ( (-\infty, u ), I_\eps ) ]_2 &\leq 0 \vee \left( \sup_{r \in I_\eps} \sup_{ w \in [-\sfE_{\ell-1}, u)} \Psi_{p,\ell}^w (r) \right) \,.
\label{eq:coro8_2}
\end{align}
Note that, as $\Psi_{p,\ell}^u(0)$ is twice the complexity function, which is continuous, it follows from the intermediate value theorem that there is $z \in (-\sfE_\ell , u )$ such that $\Psi_{p,\ell}^z(0) = \Sigma_{p,\ell}(u)$, i.e. half the value of $\Psi_{p,\ell}^u(0)$. For this $z = z(u)$, let us take $\eps$ small enough so that, by continuity of the map $r \mapsto \Psi_{p,\ell}^w(r)$, we have that $\Psi_{p,\ell}^z(\eps) \geq \tfrac{1}{2} \Sigma_{p,\ell}(u) > 0$. It follows that \eqref{eq:coro8_1} admits the simpler bound 
\begin{align}
\limsup_{N \to \infty} \frac{1}{N} \log \E [ \crit_{N,\ell} ( (-\infty, u ), I_\eps ) ]_2 &\leq \sup_{r \in I_\eps} \sup_{ w \in [-\sfE_{\ell-1}, u)} \Psi_{p,\ell}^w (r) \,,\\
&\leq\sup_{0 < r \in I_\eps} \sup_{ w \in [-\sfE_{\ell-1}, u)} \Psi_{p,\ell}^w (r) \,,
\end{align}
where the parity considerations at the beginning of the proof of Lemma~\ref{lem:prelem7} allow us to restrict our attention to $0 < r \in I_\eps$ in the second line.

From \eqref{eq:r_star}, we see that $w \mapsto r_*(w)$ is decreasing, which follows from the fact that $v(x) = \tfrac{1}{2}(|x| + \sqrt{x^2-4})$ is increasing in $x$. Thus, we may take $\eps$ sufficiently small so that $\eps < r_*(u)$, and hence that $\eps < r_*(w)$ for all $w < u$. 

We now make two claims, both for fixed $w < u$:
\begin{enumerate}
\item Over the interval $[0,r_*(w))$, the function $r \mapsto \Psi_{p,\ell}^w(r)$ decreases, and 
\item that over the interval $(r_*(w),1)$, the function $r \mapsto \Psi_{p,\ell}^w(r)$ increases. 
\end{enumerate}
For $w \in (-E_\ell(p),   u)$ fixed, and given $r < r_*(w)$, we have
\begin{align}
\Psi_{p,\ell}^w(0) - \Psi_{p,\ell}^w(r) &= -\log \mathcal{G}(r) + g(r) u^2\,,
\label{eq:coro8_3}
\end{align}
where $g(r)$ was defined in \eqref{eq:little_g_def} at the start of the proof. In \cite{subag2017complexity}, it is shown that $g'(r) > 0$, see equation (6.14). Moreover, it is easily seen that 
\begin{align}
G(r)^{-2} \equiv \frac{ 1-r^{2p-2}}{1-r^2} 
\end{align}
is an increasing function of $r$, from which it follows from \eqref{eq:coro8_3} that the function $r \mapsto \Psi_{p,\ell}^w(0) - \Psi_{p,\ell}^w(r)$  is an increasing function of $r$ over the interval $[0, r_*(w))$, which settles the first claim.  

Turning to the second claim, note that in Lemma~\ref{lem:rate_fn_dom}, after equation \eqref{eq:increase_in_r} we show that for fixed $w$, the function $r \mapsto \mathscr{I}_r(w,w)$ is non-increasing. Given $r, r' \in (0,1)$, write
\begin{align}
\Psi_{p,\ell}^w(r) - \Psi_{p,\ell}^w(r') = \log \mathcal{G}(r) - \log \mathcal{G}(r') - \ell \left( \mathscr{I}_r(w,w) - \mathscr{I}_{r'}(w,w) \right) - \left( g(r') - g(r) \right) w^2 \,,
\end{align}
which is, by the above observations about $\mathscr{I}_r$, $g$ and $\mathcal{G}$ negative when $r' > r$. Thus both claims hold. Using these claims, we have shown
\begin{align}
\limsup_{N \to \infty} \frac{1}{N} \log \E [ \crit_{N,\ell} ( (-\infty, u ), I_\eps ) ]_2 &\leq \left( \sup_{ w \in [-E_{\ell-1}(p), u)} \Psi_{p,\ell}^w(\eps) \right) \vee \un{\Psi}_{p,\ell}^{u}(1)\,,
\label{eq:coro8_4}
\end{align}
where we've used \eqref{eq:lem7_2}, the formulation of $\un{\Psi}_{p,\ell}^w$ in terms of the complexity, to deduce that it increases in $w$. Finally, for $\eps < r_*(u)$ and $w < u$, we consider the difference
\begin{align}
\Psi_{p,\ell}^u (\eps) - \Psi_{p,\ell}^w(\eps) = 2(\Omega(u) - \Omega(w)) - \ell ( \mathscr{I}_\eps(u,u) - \mathscr{I}_\eps(w,w) ) + (1- g(\eps))(w^2 -u^2)\,,
\end{align}
and we observe that this difference is positive: firstly, the integral formulation in \eqref{eq:Omega_def} implies the first term is positive. Here, we are subtracting a negative number in the second term, and finally the denominator of $g$, as defined in \eqref{eq:little_g_def} is positive, and hence that the difference in the third term above is positive. Using this in \eqref{eq:coro8_4},
\begin{align}
\limsup_{N \to \infty} \frac{1}{N} \log \E [ \crit_{N,\ell} ( (-\infty, u ), I_\eps ) ]_2 &\leq  \Psi_{p,\ell}^u(\eps) \ \vee \un{\Psi}_{p,\ell}^{u}(1)\,,
\end{align}
from which we can use claim (1) above: that $r \mapsto \Psi_{p,\ell}^u(r)$ decreases on the interval $[0,\eps]$. As $\Psi$ is continuous, we can choose $\eps$ small enough so that $\Psi_{p,\ell}^u(\eps) > \un{\Psi}_{p,\ell}^u(1)$, and hence
\begin{align}
\limsup_{N \to \infty} \frac{1}{N} \log \E [ \crit_{N,\ell} ( (-\infty, u ), I_\eps ) ]_2 &\leq  \Psi_{p,\ell}^u(\eps) < \Psi_{p,\ell}^u(0) \equiv \Sigma_{p,\ell}(u) \,,
\end{align}
completing the proof. \end{proof}

\section{\large Large deviation principle for eigenvalue pairs}\label{sec:LDP_pair}

The purpose of this last section is to prove a large deviation principle for the pair of $\ell\oith$ largest eigenvalues of two correlated GOE matrices, Theorem~\ref{coro:LDP_ell}. We state the LDP below and give its proof over the next three subsections. For $x \geq 2$, define
 \begin{align}
 v(x) \tri \frac{x + \sqrt{x^2 -4}}{2} \,\,,
 \label{eq:v_def}
 \end{align}
 and define the functions
 \begin{align}
J(x,y) &= I_1(x;1) + I_1(y;1)\,, \label{eq:leading_rate_1}\\
J_r( x,y) &= \frac{1}{2} \left( I_1(x;1) + I_1(y;1) \right) + \frac{1}{2} \log r^{p-2} + \frac{1+r^{2p-4}}{8(1-r^{2p-4})} \left(x^2 +y^2\right) - \frac{|r|^{p+2} xy}{ 2(1-r^{2p-4})} \,, \label{eq:leading_rate_2}
\end{align}
where $I_1(x;1)$ from \eqref{eq:GOE_rate_general} is the rate function for the leading eigenvalue of a single GOE matrix.

\begin{thm} Consider two correlated GOE matrices $\hat{G}^1$ and $\hat{G}^2$ distributed as in \eqref{eq:GOE_joint} with $r \in [0,1)$. Let $\eta_\ell^1$ and $\eta_\ell^2$ denote the $\ell\oith$ largest eigenvalues of $\hat{G}^1$ and $\hat{G}^2$ respectively. The pair $(\eta_\ell^1,\eta_\ell^2)$ obeys a LDP with speed $N$ and rate function
\begin{align}
\rateell_r( x,y) = 
\begin{dcases} 
\ell \cdot J(x,y) &\quad \text{ if } v(x)v(y) \leq |r|^{2-p} \quad \text{ and $x,y >2$,\, } \text{ or if $x=2$ or $y =2$\,,}\\
\ell \cdot J_r(x,y) &\quad \text{ if } v(x)v(y) \geq |r|^{2-p}\quad \text{ and $x,y >2$\,,} \\
\infty &\quad \text{ otherwise,}
\end{dcases}
\label{eq:LDP_ell}
\end{align}
where $v$ is the function given in \eqref{eq:v_def}, and where $J(x,y)$ and $J_r(x,y)$ are defined in \eqref{eq:leading_rate_1} and \eqref{eq:leading_rate_2}. 
\label{coro:LDP_ell}
\end{thm} 

To prove Theorem~\ref{coro:LDP_ell}, we apply the contraction principle to a result of Donati-Martin and Ma\"ida \cite{donati2012large}. We then analyze the resulting variational problem to get the explicit rate function \eqref{eq:LDP_ell} above. 

\subsection{Contracting the Donati-Martin-Ma\"ida LDP} 

Let $(B_{ij} )_{1\leq i \leq j \leq N}$ be a collection of real i.i.d. standard Brownian motions. The \emph{symmetric Brownian motion} $(H_N(t))_{t \geq 0}$ is a process taking values in the space of $N \times N$ symmetric matrices, with entries
\begin{align}
(H_N)_{ij} =
\begin{cases}
\sqrt{\tfrac{1}{N}}  B_{ij} & \text{ if } i < j\,, \\
\sqrt{\tfrac{2}{N}} B_{jj} & \text{ if } i = j \,. 
\end{cases}
\label{eq:hermitian_BM}
\end{align}
\begin{comment}
The eigenvalues of this matrix satisfy the following system of SDEs, for $i, j \in \{1, \dots N\}$ and $t \geq 0$:
\begin{align}
d\lambda_i(t) = \frac{1}{\sqrt{N}} dB_i(t) + \frac{1}{N} \sum_{ j \neq i} \frac{1}{\lambda_i(t) - \lambda_j(t) } dt \,. 
\end{align}
This is called \emph{Dyson Brownian motion}. 
\end{comment}
Consider the process $(\lambda_\ell^N(t))_{0 \leq t \leq 1}$ tracking the $\ell\oith$ largest eigenvalue of the symmetric Brownian motion over the time interval $[0,1]$. The next theorem in the case $\ell =1$ is the central result of \cite{donati2012large}; it gives a large deviation principle for the leading eigenvalue process.

Let $C([0,1],\R)$ denote the space of continuous functions from $[0,1]$ to $\R$ with initial value zero, and equip this space with the topology of uniform convergence. Let $\mathcal{A}$ denote the set of $\vp \in C([0,1],\R)$ which are absolutely continuous, and such that $\vp(t) \geq 2\sqrt{t}$ for all $t \in [0,1]$.

\begin{thm} The law of the process $(\lambda_\ell(t))_{0 \leq t \leq 1}$ satisfies a large deviation principle on $C([0,1],\R)$ with speed N. The good rate function associated to this LDP is given by
\begin{align}
I_\ell(\vp)=\ell I(\vp) = 
\begin{cases}
\frac{\ell}{4} \int_0^1\left( \dot{\vp}(s) - \frac{1}{2s} \left(\vp(s) - \sqrt{ \vp(s)^2 - 4s } \right) \right)^2 ds\,, & \text{ if } \vp \in \mathcal{A} \,,\\
+ \infty & \text{ else}\,. 
\end{cases}
\label{eq:donati_rate}
\end{align}
\label{thm:donati_main}
\end{thm}

%\cite{donati2012large}, Theorem 1.1 and Remark 1.3

\begin{rmk}
Although Donati-Martin and Ma\"ida proved the above theorem only in the case $\ell = 1$, their proof also works for arbitrary $\ell$ provided minor modifications. We include a short discussion of these minor changes in subsection~\ref{AppendixProofLDPl} for completeness. The following intuition for obtaining the general case from the $\ell=1$ case guides the proof: in order to force the $\ell\oith$ largest eigenvalue process to be above a function $\varphi > 2\sqrt{t}$, one needs to put $\lambda_k > \varphi$ for $k=1, \ldots, \ell$. The cost of moving each one of these lines is the same and equal to $I(\varphi)$ since once they are away, they essentially exert no interaction to the rest of the Dyson Brownian motion. 
\end{rmk}

In \cite{donati2012large}, Donati-Martin and Ma\"ida applied the contraction principle to their main result at the right endpoint of the leading eigenvalue process; this gave another proof of the LDP for the leading eigenvalue of a GOE matrix. We apply the contraction principle to their result at two distinct points in time, and upon rescaling, this yields a LDP for the pair of leading eigenvalues of two \emph{correlated} GOE matrices. Remark~\ref{rmk:q_and_r} below has more details. With this approach in mind, we make the following definitions. Given $q \in [0,1)$ and $x,y \in \R$, define
\begin{align}
\mathcal{C}(0 \mapsto 0; q\mapsto y) &\tri \{ \vp \in C([0,q],\R) : \vp(0) = 0,\, \vp(q) = y \} \cap \mathcal{A} \, \label{eq:C_0to0_qtoy}\\
\mathcal{C}(q \mapsto y; 1 \mapsto x) &\tri \{ \vp \in C([q,1],\R) : \vp(q) = y,\, \vp(1) = x \} \cap \mathcal{A}\label{eq:C_qtoy_1tox} \,,
\end{align} 
and write $\mathcal{C}(x,y;q)$ for the functions on $[0,1]$ expressible as a concatenation of a function in \eqref{eq:C_0to0_qtoy} and a function in \eqref{eq:C_qtoy_1tox}:
\begin{align}
\mathcal{C}(x,y;q) \tri \{ \vp \in C([0,1],\R) :  \vp(0) = 0,\, \vp(1) = x,\, \vp(q) = y\} \cap \mathcal{A}\,.
\end{align}
For fixed $q$, the map $C([0,1], \R) \ni \vp \mapsto (\vp(1), \vp(q)) \in \R^2$ is continuous, and the contraction principle implies that the rate function for the pair $(\lambda_\ell(1), \lambda_\ell(q))$ is 
\begin{align}
\ell \bar{J}(x,y;q) \tri \ell \inf_{\vp\, \in\, \mathcal{C}(x,\,y;\,q)} I(\vp) \,,
\label{eq:our_rate_1}
\end{align}
with the convention that the infimum over the empty set is $\infty$.

\begin{rmk}Given the form of the rate function $I_\ell$, we lose no generality analyzing the leading eigenvalue case, and so we take $\ell = 1$ in the rest of this subsection and in those leading up to subsection~\ref{AppendixProofLDPl}. 
\end{rmk}

Define
\begin{align}
f(t,a,b) = \frac{1}{4} \left( b - \frac{1}{2t}  \left( a - \sqrt{ a^2 - 4t} \right) \right)^2 \,,
\end{align}
so that the rate function $I$ from Theorem~\ref{thm:donati_main} is given by
\begin{align}
I(\vp) = \int_0^1 f(t, \vp(t), \dot{\vp}(t) ) dt \,
\end{align}
when $\vp \in \mathcal{A}$. Split $I(\vp)$ as
\begin{align}
I(\vp) &= \int_0^q f(t, \vp(t), \dot{\vp}(t)) dt + \int_q^1 f(t, \vp(t), \dot{\vp}(t)) dt  \hspace{2mm} \tri  \hspace{2mm}I_0^q(\vp) + I_q^1(\vp) \,,
\end{align}
and define
\begin{align}
 \bar{J}_{\,0 \,\mapsto\, 0}^{\,q\,\mapsto\, y} \,\,&\tri  \inf_{\vp\, \in \,\mathcal{C}(0 \,\mapsto \,0; \,q\,\mapsto \,y)  } I_0^q(\vp) \,. \label{eq:A_lower}  \\
\bar{J}_{\,q \,\mapsto\, y}^{\,1\,\mapsto\, x} \,\,&\tri  \inf_{\vp \in \mathcal{C}(q \,\mapsto \,y; \,1\,\mapsto \,x) } I_q^1(\vp)\,.
\label{eq:A_upper}
\end{align}

\begin{rmk} Our goal is now to write the contracted rate function $\bar{J}(x,y;q)$ defined in \eqref{eq:our_rate_1} as an explicit function of these parameters, which will allow us to prove Theorem~\ref{coro:LDP_ell} in the case $\ell =1$ after a straightforward change of variables. We do this by exhibiting a function $\psi \in \mathcal{C}(x,y;q)$ such that $ \bar{J}_{\,0 \,\mapsto\, 0}^{\,q\,\mapsto\, y} = I_0^q(\psi)$, and such that $\bar{J}_{\,q \,\mapsto\, y}^{\,1\,\mapsto\, x} = I_q^1(\psi)$. This will show $J(x,y;q) =  \bar{J}_{\,0 \,\mapsto\, 0}^{\,q\,\mapsto\, y} + \bar{J}_{\,q \,\mapsto\, y}^{\,1\,\mapsto\, x} $, and $J(x,y;q)$ will become explicit after computing $I_0^q(\psi)$ and $I_q^1(\psi)$.

We build $\psi$ using results and methods from section 6 of  \cite{donati2012large}. Their analysis determines the values of $x,y$ and $q$ for which either:
\begin{itemize}
\item[(i)] the optimal function touches the barrier $t \mapsto 2\sqrt{t}$ in the interior of the given interval, or
\item[(ii)] the (linear) solution to the Euler-Lagrange equation corresponding to \eqref{eq:donati_rate} is optimal.
\end{itemize}
\label{rmk:obvious_inf}
\end{rmk} 
 \begin{comment} 
\begin{rmk} Returning to \eqref{eq:A_obvious}, consider the first infimum:
\begin{align}
\bar{J}_{\,0 \,\mapsto\, 0}^{\,q\,\mapsto\, y} \tri  \inf_{\vp \in \mathcal{C}(0 \mapsto 0; q\mapsto y)  } I_0^q(\vp) \,. 
\end{align}
The infimum of $I_0^q$ with the barrier constraint $2\sqrt{t}$ is not related by a change of variables to the infimum of $I_0^1$ with the same barrier constraint. Nonetheless, we can still study this problem using the strategy in \cite{donati2012large}. As $I$ is a good rate function, the infimum of $I_0^q$ over $\mathcal{C}(0 \mapsto 0; q\mapsto y) $ is attained. Let $\vp^q$ denote a  such a function.
\end{rmk} 
\end{comment}

We analyze \eqref{eq:A_lower} and \eqref{eq:A_upper} one at a time, starting with the former. The setting of this first variational problem is the interval $[0,q]$, and the infinite slope of the barrier $t \mapsto 2\sqrt{t}$ at the origin forces optimal functions to fall into case (i) above. The proof of Lemma~\ref{lem:optimal_J_bar} below is based on the proofs of Lemmas 6.2 and 6.3 in \cite{donati2012large}.

\begin{lem} For any $q \in [0,1)$ and $y \geq 2\sqrt{q}$, 
\begin{align}
\bar{J}_{\,0 \,\mapsto\, 0}^{\,q\,\mapsto\, y} = \frac{1}{2} \int_2^{\,y\, /\sqrt{q}} \sqrt{u^2 - 4} du \,.
\label{eq:q_variational_1}
\end{align}
\label{lem:optimal_J_bar}
\end{lem} 

\begin{proof} Starting from
\begin{align}
I_0^q(\vp) = \frac{1}{4} \int_{0}^q \left[ \dot{\vp}(s) - \frac{1}{2s} \left( \vp(s) - \sqrt{ \vp^2(s) - 4s} \right) \right]^2 ds \,
\end{align}
define $K(\vp)$ as the above expression with a single sign change:
\begin{align}
K(\vp) \tri \frac{1}{4} \int_{0}^q \left[ \dot{\vp}(s) - \frac{1}{2s} \left( \vp(s) + \sqrt{ \vp^2(s) - 4s} \right) \right]^2 ds\,,
\end{align}
and note that
\begin{align}
I_0^q(\vp) - K(\vp) = \frac{1}{2}  \int_{0}^q \frac{1}{s} \left( \dot{\vp}(s) - \frac{\vp(s)}{2s} \right) \sqrt{\vp^2(s) - 4s} \,\, ds \,.
\end{align}
As $K$ is nonnegative, it follows that 
\begin{align}
I_0^q(\vp) \geq \frac{1}{2} \int_{0}^q \frac{1}{s} \left( \dot{\vp}(s) - \frac{\vp(s)}{2s} \right) \sqrt{\vp^2(s) - 4s} ds \,. 
\end{align}
Define $z(s) \tri \tfrac{\vp(s)}{\sqrt{s}}$. Writing the right-hand side above in terms of $z(s)$ instead of $\vp(s)$, we find that for each $\eps > 0$,
\begin{align}
I_0^q(\vp) \geq  \frac{1}{2} \int_{\eps}^q \dot{z}(s) \sqrt{ z^2(s) -4 } ds \,,
\end{align}
and making the change of variables $u = z(s)$ yields
\begin{align}
I_0^q(\vp) \geq \limsup_{\eps \to 0} \frac{1}{2} \int_{ \vp(\eps) / \sqrt{\eps} }^{\,y\,/ \sqrt{q}}  \sqrt{u^2 - 4} du\,.
\label{eq:q_variational_2}
\end{align}
We now make the claim that, for any optimizer $\vp$ of $I_0^q$, 
\begin{align}
\liminf_{\eps \to 0} \frac{ \vp(\eps) }{ \sqrt{\eps}} = 2 \,. 
\end{align}
The lower bound on the liminf is automatic from the class of functions we consider. To see that equality holds, note that $\vp$ encounters infinitely many points along the curve $t \mapsto 2\sqrt{t}$ as $\eps \to 0$. If this were not the case, a nearly identical argument to the proof of Lemma 6.1 in \cite{donati2012large} implies that the restriction of $\vp$ to a sufficiently small interval of the form $[0,\delta]$ is a linear function. But, this is impossible, as the slope of the barrier at the origin is infinite. This proves the claim. Combined with \eqref{eq:q_variational_2}, we have shown for any optimal $\vp$ that
\begin{align}
I_0^q(\vp) \geq \frac{1}{2}  \int_{2 }^{\,y\,/ \sqrt{q}}  \sqrt{u^2 - 4} du \,. 
\label{eq:optimal_J_bar_1}
\end{align}
To complete the proof, it suffices to exhibit a function achieving this lower bound. Let $\psi$ be the following piecewise-defined function:
\begin{align}
\psi(t) = 
\begin{cases}
2\sqrt{t} & 0 \leq t \leq t^* ,\\
\frac{1}{\sqrt{t^*}}t + \sqrt{t^*}& t \geq t^* ,
\end{cases}
\end{align}
with $t^*$ defined so that 
\begin{align}
\sqrt{t^*} = \tfrac{ y - \sqrt{y^2-4q}}{2} \,.
\end{align}
The above function is chosen to be continuously differentiable, and so that it has the correct value at the endpoints of $[0,q]$. We leave it to the reader to verify that equality holds in \eqref{eq:optimal_J_bar_1} for this $\psi$: using the same change of variables as above, one only needs to show that $K(\psi) \equiv 0$.  \end{proof}

We now turn to \eqref{eq:A_upper}, and in this case, a function realizing this infimum will either touch the barrier $t \mapsto 2\sqrt{t}$ on $(q,1)$ or it will be linear depending on the paramters $x,y$ and $q$. These cases are treated separately in Lemma~\ref{lem:donati_6.3} and Lemma~\ref{lem:donati_6.4}. In the following lemma, the range of parameters is such that the function realizing the infimum \eqref{eq:A_upper} touches the barrier. 

\begin{lem}[{\cite[Lemma 6.3]{donati2012large}}] For any $q \in [0,1)$, suppose the triple $(x,y,q)$ satisfies
\begin{align}
y = 2\sqrt{q} \quad \text{\rm or } \quad \left( 2\sqrt{q} < y < 1 + q \quad\text{ \rm and }\quad x \leq \tfrac{y + \sqrt{y^2 - 4q}}{2} + \tfrac{2}{y + \sqrt{y^2 - 4q}} \right)\,,
\label{eq:donati_c_1}
\end{align}
then we have
\begin{align}
\bar{J}_{\,q \,\mapsto\, y}^{\,1\,\mapsto\, x} = \frac{1}{2} \int_2^x \sqrt{u^2 - 4} du \,.
\end{align}
and the argmin of $\bar{J}_{\,q \,\mapsto\, y}^{\,1\,\mapsto\, x}$ can be written explicitly.
\label{lem:donati_6.3}
\end{lem} 
It is shown in the proof of Lemma~6.1 of \cite{donati2012large} that the function
\begin{align}
g_{x,y,q}(t) \tri \frac{ x- y}{1-q} (t - q) + y \,
\label{eq:EL_sol}
\end{align}
solves the Euler-Lagrange equations associated to \eqref{eq:donati_rate}. In the next lemma, the parameters $x,y$ and $q$ are constrained so that \eqref{eq:EL_sol} realizes the infimum \eqref{eq:A_upper}. 

\begin{lem}[{\cite[Lemma 6.4 and Lemma 6.5]{donati2012large}}] For any $q \in [0,1)$, suppose the triple $(x,y,q)$ satisfies the constraint 
\begin{align}
y \geq 1+ q \quad \text{\rm or } \quad \left( 2\sqrt{q} < y < 1 + q \quad\text{\rm and }\quad x > \tfrac{y + \sqrt{y^2 - 4q}}{2} + \tfrac{2}{y + \sqrt{y^2 - 4q}} \right)\,.
\label{eq:donati_c_2}
\end{align}
Then, we have $\bar{J}_{\,q \,\mapsto\, y}^{\,1\,\mapsto\, x} = I_q^1(g_{x,y,q})$, where the linear function $g_{x,y,q}$ is defined in \eqref{eq:EL_sol}. 
\label{lem:donati_6.4}
\end{lem} 

We remark that while in \cite{donati2012large}, the first inequality $y \geq 1+q$ in \eqref{eq:donati_c_2} is presented as strict, one easily sees that when $y = 1+q$, the linear function only touches the barrier at the endpoint $t=1$ and hence does not feel its presence.

\subsection{Making the rate function explicit} 
 
The next corollary follows directly from Remark~\ref{rmk:obvious_inf} and Lemmas~\ref{lem:optimal_J_bar} -- \ref{lem:donati_6.4}.

\begin{coro} For $q \in [0,1)$, the pair $(\lambda_1^N(1), \lambda_1^N(q))$ obeys a large deviation principle with good rate function
\begin{align}
\bar{J}(x,y;q) = 
\begin{dcases}
 \frac{1}{2}\int_{2 }^{\,y\,/ \sqrt{q}}  \sqrt{u^2 - 4} du + \frac{1}{2}\int_2^x \sqrt{u^2 - 4} du & \text{ if } \,\, (x,y,q) \,\, \text{satisfies} \quad\eqref{eq:donati_c_1} \, \\
\frac{1}{2}  \int_{2 }^{\,y\,/ \sqrt{q}}  \sqrt{u^2 - 4} du + I_q^1(g_{x,y,q}) & \text{ if }\,\, (x,y,q) \,\, \text{satisfies}\quad\eqref{eq:donati_c_2} \,,
\end{dcases}
\label{eq:SBM_LDP}
\end{align}
where $g_{x,y,q}$ is the linear function defined in \eqref{eq:EL_sol}. 
\label{coro:SBM_LDP}
\end{coro}

\begin{comment}
\begin{proof} Applying Lemma~\ref{lem:donati_6.3} and Lemma~\ref{lem:donati_6.4}, we have an exact expression for $\bar{J}_{\,q \,\mapsto\, y}^{\,1\,\mapsto\, x}$. In each case, an explicit function achieving the infimum defining $\bar{J}_{\,q \,\mapsto\, y}^{\,1\,\mapsto\, x}$ is given in the proof of Lemma 6.3 of \cite{donati2012large}. Moreover, Lemma~\ref{lem:optimal_J_bar} furnishes the desired value of $\bar{J}_{\,0 \,\mapsto\, 0}^{\,q\,\mapsto\, y}$ as well as a minimizer. By Remark~\ref{rmk:obvious_inf}, the proof is complete.

The pair $(\lambda(1), \lambda(q))$ of leading eigenvalues arising from symmetric Brownian motion are related to $(\eta^1,\eta^2)$, by the linear transformation $\diag(1, \sqrt{1/q})$, provided $q$ is chosen well in terms of $r$. 
\end{proof}
\end{comment}

\begin{rmk} Rescale all entries of $H_N(q)$, the symmetric Brownian motion at time $q$, so that each diagonal entry has variance $2/N$, and observe that this rescaled matrix and $H_N(1)$ form a pair of correlated GOE matrices. This pair has the law described by \eqref{eq:GOE_joint} provided $q$ is chosen well in terms of $r$. Applying the contraction principle once more, one finds the rate function for the pair of leading eigenvalues associated to the GOE matrix pair is given by
\begin{align}
J_r(x,y) \tri \bar{J}(x, \sqrt{q}y;q ) \,,
\label{eq:rate_fn_relation}
\end{align}
with the relation $q = r^{2p-4}$ obtained by matching variances, and considered as part of the above definition. 
\label{rmk:q_and_r}
\end{rmk}

We  have not computed the integrals in \eqref{eq:SBM_LDP} above because it is convenient to first make the change of variables \eqref{eq:rate_fn_relation}. Before doing this computation, we show in Lemma~\ref{lem:new_cond} below that the constraints on $(x,y,q)$ given in \eqref{eq:SBM_LDP} transform nicely. 

Let us make an observation about the function $v$ which we use several times in the proof of this lemma: for $q \in [0,1)$, we have 
\begin{align}
a = (1 + q) / \sqrt{q} \quad\text{ if and only if }\quad  v(a) = 1/\sqrt{q} \,,
\label{eq:v_observation}
\end{align}
and moreover, as $v$ is strictly increasing on its domain, \eqref{eq:v_observation} holds when we replace the two equalities by two of the same inequality. 

\begin{lem} Consider $x,y > 2$ and $q \in [0,1)$, and recall the function $v$ defined in \eqref{eq:v_def}. The constraint on the parameters $(x,y,q)$ formed by making the substitution $y \mapsto \sqrt{q} y$ in \eqref{eq:donati_c_1} is equivalent to $v(x)v(y) \leq 1 /\sqrt{q}$. Likewise, the constraint formed by making this substitution in \eqref{eq:donati_c_2} is equivalent to $v(x)v(y) > 1/\sqrt{q}$.
 \label{lem:new_cond}
 \end{lem}
 
 \begin{proof} We write the first constraint $\eqref{eq:donati_c_1}$ after making the substitution $y \mapsto \sqrt{q}y$:
\begin{align}
2 < y < (1+q)/\sqrt{q} \quad\text{ and }\quad x \leq \sqrt{q} v(y) + \left(\sqrt{q} v(y)\right)^{-1}\,.
\label{eq:first_condition}
\end{align}
Consider the constraint on the right involving $x$: if we set the inequality to an equality and solve for $\sqrt{q}v(y)$, we find $v(x)$ and $1/v(x)$ as possible solutions. Relaxing this back to an inequality, we see this latter constraint is equivalent to either \hyperlink{eq:new_cond_1}{(i)} or \hyperlink{eq:new_cond_2}{(ii)} holding:
\begin{itemize}
\item[\hypertarget{eq:new_cond_1}{(i)}]  $v(y) \geq v(x) / \sqrt{q}$.
\item[\hypertarget{eq:new_cond_2}{(ii)}] $v(x) v(y) \leq 1/\sqrt{q}$.
\end{itemize} 
But if \hyperlink{eq:new_cond_1}{(i)} held, we could combine \eqref{eq:v_observation} with the inequality $v(y) \geq v(x) / \sqrt{q} \geq 1/\sqrt{q}$ to contradict the left-most constraint on $y$ in \eqref{eq:first_condition}. Thus  \eqref{eq:first_condition} implies \hyperlink{eq:new_cond_2}{(ii)}.

As \hyperlink{eq:new_cond_2}{(ii)} gives the right-most constraint in \eqref{eq:first_condition}, note also that \eqref{eq:v_observation} tells us $x,y > 2$ implies  $v(x), v(y) > 1$. Using this observation on  \hyperlink{eq:new_cond_2}{(ii)} gives $v(x), v(y) < 1/\sqrt{q}$, and another application of \eqref{eq:v_observation} implies $x, y < (1+q)/\sqrt{q}$, so that \hyperlink{eq:new_cond_2}{(ii)}  implies the left-most constraint in \eqref{eq:first_condition}, and hence that \hyperlink{eq:new_cond_2}{(ii)} implies \eqref{eq:first_condition}, which verifying the lemma for \eqref{eq:donati_c_1}.

Turning to \eqref{eq:donati_c_2}, make the same substitution:
\begin{align}
y \geq (1+ q)/ \sqrt{q} \quad \text{ or } \quad \left( 2 < y < (1 + q)/\sqrt{q} \quad\text{\rm and }\quad x > \sqrt{q}v(y) + \left( \sqrt{q} v(y) \right)^{-1} \right)\,.
\label{eq:second_condition}
\end{align}
The simpler constraint $y \geq (1+q) /\sqrt{q}$ implies $v(x)v(y) > 1/\sqrt{q}$: by \eqref{eq:v_observation}, it gives $v(y) \geq 1/\sqrt{q}$, and we use this with $v(x) > 1$. Thus we may now assume the more complicated condition in \eqref{eq:second_condition} in parentheses holds. Examining the lower bound on $x$ in \eqref{eq:second_condition}, the reasoning that gave us  \hyperlink{eq:new_cond_1}{(i)} and  \hyperlink{eq:new_cond_2}{(ii)} above implies that the bound on $x$ is equivalent to \hyperlink{eq:new_cond_3}{(iii)} and \hyperlink{eq:new_cond_4}{(iv)} holding:
\begin{itemize}
\item[\hypertarget{eq:new_cond_3}{(iii)}]  $v(y) < v(x) / \sqrt{q}$.
\item[\hypertarget{eq:new_cond_4}{(iv)}] $v(y)v(x) > 1/\sqrt{q}$. 
\end{itemize} 
We have shown \eqref{eq:second_condition} implies \hyperlink{eq:new_cond_4}{(iv)}, so to verify the lemma for \eqref{eq:donati_c_2}, it remains to show the reverse implication. In this case, we can suppose that $2 < y < (1+q) /\sqrt{q}$ holds in addition to \hyperlink{eq:new_cond_4}{(iv)}. Applying \eqref{eq:v_observation} once more, we have $v(y) < 1/\sqrt{q}$ and hence that $v(y) < v(x) / \sqrt{q}$, i.e. that  \hyperlink{eq:new_cond_3}{(iii)} holds. The inequalities  \hyperlink{eq:new_cond_4}{(iii)} and  \hyperlink{eq:new_cond_4}{(iv)} together are equivalent to the desired lower bound on $x$ appearing in \eqref{eq:second_condition}, and the proof is finished. \end{proof}

Using Lemma~\ref{lem:new_cond}, we can now prove Theorem~\ref{coro:LDP_ell} in the case $\ell =1$ as a corollary of the above work. In the proof of this corollary, we describe the computations producing the expressions \eqref{eq:leading_rate_1} and \eqref{eq:leading_rate_2}. 

\begin{coro} Consider two correlated GOE matrices $\hat{G}^1, \hat{G}^2$ distributed as in \eqref{eq:GOE_joint} with $r \in [0,1)$. The pair $(\eta^1,\eta^2)$ of leading eigenvalues associated to these matrices obeys a LDP with speed $N$ and rate function
\begin{align}
\mathscr{I}_r(x,y) = 
\begin{dcases} 
J(x,y) &\quad \text{ if } v(x)v(y) \leq |r|^{2-p} \text{ and $x,y >2$,\, } \text{ or if $x=2$ or $y =2$\,,}\\
J_r(x,y) &\quad \text{ if } v(x)v(y) \geq |r|^{2-p} \text{ and $x,y >2$\,,} \\
\infty &\quad \text{ otherwise,}
\end{dcases}
\label{eq:leading_rate}
\end{align}
where $v$ is the function defined in \eqref{eq:v_def}, and where $J$ and $J_r$ are given in \eqref{eq:leading_rate_1} and \eqref{eq:leading_rate_2}.
\label{coro:GOE_corr_LDP}
\end{coro} 

\begin{proof} Let $x,y >2$, and consider $r \in [0,1)$, making the identification $q = r^{2p-4}$ throughout this proof. We first handle the easier regime $v(x)v(y) \leq \sqrt{q}$: by Remark~\ref{rmk:q_and_r}, the rate function $\mathscr{I}_r(x,y)$ in this case is obtained by making the substitution $y \mapsto \sqrt{q} y$ in the first line of \eqref{eq:SBM_LDP}. Recalling \eqref{eq:GOE_rate_general}, this is evidently $I_1(x;1) + I_1(y;1)$ in accordance with \eqref{eq:leading_rate_1}. 

Consider now the more difficult regime $v(x) v(y) \geq \sqrt{q}$. The same reasoning tells us that the rate function in this case is obtained by making the same substitution $y \mapsto \sqrt{q} y$ in the expression
\begin{align}
I_1( y / \sqrt{q};1) + I_q^1( g_{x,y,q})\,,
\label{eq:coro_gcl1}
\end{align}
where $g(s) \equiv g_{x,y,q}(s)$ is the linear function defined in \eqref{eq:EL_sol}. The integral 
\begin{align}
I_q^1(g) = \frac{1}{4} \int_q^1 \left( g'(s) - \frac{1}{2s} \left( g(s) - \sqrt{g^2(s) - 4s } \right) \right)^2 ds \,
\label{eq:GOE_corr_LDP_1}
\end{align}
itself can be evaluated through a straightforward yet somewhat lengthy computation. A crucial starting point to this computation is the change of variables $u(s) = g(s) + \sqrt{g^2(s) -4s}$, noting that the sign here is changed compared to the similar looking term in the integrand. It is also helpful in the computation to abbreviate the terms in the linear function $g$, writing $g(s) = \alpha s + \beta$, where $\alpha \tri (x-y)/ (1-q)$ and $\beta \tri y - \alpha q$.

After some manipulation which makes use of the quadratic formula, one finds that in the above notation, 
\begin{align}
ds = \frac{ \alpha u^2 - 4u + 4\beta}{2(\alpha u -2)^2} du \,,
\end{align}
while the integrand itself becomes $(\alpha u -2)^2/u^2$ after the change of variables, leading to nice cancellation. Thus, 
\begin{align}
I_q^1(g) &=  \frac{1}{4} \int_{\,u(g(q)}^{\,u(g(1))}  \frac{ \alpha u^2 - 4u + 4\beta}{2u^2} du \label{eq:new_int}\,,
\end{align}
and the integral has become trivial. The rate function begins to materialize only after evaluation of the above integral at the endpoints 
\begin{align}
u(g(q)) &= y + \sqrt{y^2-4q} \non \\
&\equiv 2\sqrt{q} v(y) \\ 
u(g(1)) &\equiv x + \sqrt{x^2-4} \non\\
&\equiv 2v(x) \,,
\end{align}
with the function $v$ as in \eqref{eq:v_def}. Before expanding the resulting expression, it is useful to make the substitution $y \mapsto \sqrt{q} y$, so that we are computing the second term of \eqref{eq:coro_gcl1} after this change. The rate function in this second regime is then:
\begin{align}
I_1(y;1) + \left[ I_q^1(g) \right]_{ y \, \mapsto\, \sqrt{q}y} \,,
\label{eq:coro_gcl2}
\end{align}
and we leave it to the reader to verify that 
\begin{align}
 4\cdot \left[ I_q^1(g) \right]_{ y \, \mapsto\, \sqrt{q}y} & =\left( \frac{1+q}{1-q} \right) \left( \frac{x^2 + y^2}{2} \right) - \left( \frac{ 2\sqrt{q}}{1-q} \right) xy + \frac{ x\sqrt{x^2-4}}{2} - \frac{ y\sqrt{y^2-4}}{2}  \non\\
&\quad\quad\quad\quad\quad\quad\quad\quad  - 2 \left(\log v(x)  - \log v(y)   \right) + \log(q)  \non\,.
\end{align}
Note that $I_1(y;1)$ can be written as $\left(y\sqrt{y^2-4}\right)/4 - \log v(y)$, as is seen comparing \eqref{eq:IOmega} with \eqref{eq:Omega_def}. Using this with the expression directly above, the two terms in \eqref{eq:coro_gcl2} combine nicely into 
\begin{align}
I_1(y;1) + \left[ I_q^1(g) \right]_{ y \, \mapsto\, \sqrt{q}y} &= \left( \frac{1+q}{8(1-q)} \right) \left( x^2 + y^2 \right) - \left( \frac{ \sqrt{q}}{2(1-q)} \right) xy + \frac{ x\sqrt{x^2-4}}{8} + \frac{ y\sqrt{y^2-4}}{8}  \non\\
&\quad\quad\quad\quad\quad\quad\quad\quad  - \frac{1}{2} \left(\log v(x)  + \log v(y)   \right) + \frac{1}{4}  \log q \non\\
&= \frac{1}{2} \left[ I_1(x;1) + I_1(y;1) \right] +  \left( \frac{1+q}{8(1-q)} \right) \left( x^2 + y^2 \right) - \left( \frac{ \sqrt{q}}{2(1-q)} \right) xy + \frac{1}{4} \log q\,, \non
\end{align}
precisely the expression in \eqref{eq:leading_rate_2} under the relation $r^{2p-4}=q$. As $I_1(x;1) + I_1(y;1)$ agrees with the above expression when $v(x)v(y) = r^{2-p}$ (we show this in  Lemma~\ref{lem:rate_equal}), we have relaxed the strictness of the constraints on $x,y$ and $q$ in the \eqref{eq:leading_rate}. 

We now turn to the case that $r \in [0,1)$ and either $x =2$ or $y=2$. Returning to Remark~\ref{rmk:q_and_r}, the case $y=2$ falls under \eqref{eq:donati_c_1} after the change of variables $y \mapsto \sqrt{q}y$, in which $\mathscr{I}_r(x,y) = J(x,y)$ as desired. The case $x=2$ is left to the reader.
\end{proof}

\subsection{LDP for the $\ell\oith$ largest eigenvalue pair}\label{AppendixProofLDPl}

The task of this subsection is to provide the minor modifications needed in the proof of  Donati-Martin and Ma\"ida \cite[Theorem 1]{donati2012large}   to obtain Theorem~\ref{thm:donati_main}. We assume that the reader is familiar to the notation of that paper and we only provide the necessary changes. We start with the lower bound. 

Recall that $(\lambda_1(t), \ldots, \lambda_N(t))$ is the solution of the following system of stochastic differential equations

\begin{equation} \label{EDSvp}
d\lambda_i(t) = \frac{1}{\sqrt{N}}dB_i(t) + \frac{1}{N} \sum_{ j \not= i} \frac{1}{\lambda_i(t) - \lambda_j(t)} dt , \, t \geq 0, \; i,j = 1, \ldots, N
\end{equation}
where $B_i, i=1, \ldots, N$ are independent standard real  Brownian motions.

\begin{prop} \label{lowerbound0}
 For any open set $O$ in $ C([0,1]; \R)$ and any $\ell \geq 1$,
 \begin{equation} \label{eq-lowerbound0}
\liminf_{N \rightarrow \infty}\frac{1}{N} \ln \mathbb P(\lambda_\ell^{N} \in O) \geq - \inf_{\varphi \in O}  I_\ell(\varphi).
\end{equation}
\end{prop}
\begin{proof}
For $h \in L^2([0,1]),$ we define the exponential martingale $M^h$ such that for any $t \in [0,1],$
\begin{equation}\label{defimart}
  M_t^h = \exp \left[ N \left(  \sum_{k=1}^\ell \int_0^t h(s) \frac{1}{\sqrt{N}}
  dB_k(s) - \frac{\ell}{2} \int_0^t h^2(s) ds\right) \right],
\end{equation}
where $B_\ell$, $\ell = 1 , \ldots, \ell$ are the standard Brownian motion appearing in the SDE
for $\lambda_\ell^{N}$ \eqref{EDSvp}. This martingale replaces the martingale in (4.12) of \cite{donati2012large}. 

For  $\varphi$ such that $ I_\ell(\varphi) < \infty$, set
\begin{equation} \label{defkphi}
k_\varphi(s) \tri \dot{\varphi}(s) - \frac{1}{2s}(\varphi(s) - \sqrt{\varphi^2(s) -4s}). 
\end{equation}
As noted above, 

$$ I_\ell(\varphi) = \ell I_1(\theta) = \frac{\ell}{2}\int_0^1 k_\varphi^2(s) ds = \frac{\ell}{2} \|k_\varphi \|_2^2.$$

We now define a new probability measure  $\mathbb P^{\varphi}$ with $M_t^{k_\phi}$ as its Radon-Nykodim derivative with respect to $\mathbb P$.  We also denote by $\E^{\varphi}$ the expectation
under $\mathbb P^{\varphi}.$ 
Set 
$$\nu_N = \frac{1}{N-\ell} \sum_{k=\ell+1}^N \delta_{\lambda_k^N(t)}$$ to be the empirical distribution of all but the $\ell\oith$ largest eigenvalues, $\mu_N$ to be the the empirical distribution of all eigenvalues, and $\sigma_t =\frac{1}{2 \pi t} {\bf 1}_{[-2\sqrt t, 2\sqrt t]} \sqrt{4t-x^2}.$

An application of Girsanov's Theorem implies that under $\mathbb P^{k_\varphi}$ we have 
\begin{equation}\label{newSDE}
\displaystyle d\lambda_i(t) = \frac{1}{\sqrt{N}}d\beta_i(t) + c_i k_\varphi(t) dt + \frac{1}{N} \sum_{ j \not= 1} \frac{1}{\lambda_1(t) - \lambda_j(t)} dt,
\end{equation}
where $(\beta_i)_{1\le i\le N}$ are independent Brownian motions under $\mathbb P^{\rho_\varphi}$ and
\begin{equation}
c_i = \begin{cases} 1,    \text{ if } i = 1, \ldots, \ell, \\ 
0, \text{ if } i = \ell+1,\ldots, N.
\end{cases}
\end{equation}
By It\^o's formula we obtain a stochastic differential equation for $\langle \mu_N, f \rangle$ for any $f \in C^2$ with a diffusion coefficient that goes to $0$ as $N$ tends to $\infty$. Any limiting point of this equation satisfies a deterministic equation given by 
$$ \langle \mu_t , f \rangle =  \int f(x) d\mu_t(x) = \int f(x) d\mu_0(x) +\frac{1}{2}  \int_0^t \int \frac{f'(x) - f'(y)}{x-y} d\mu_s(x) d\mu_s(y) ds.$$
When the initial condition is $\delta_0$, the semicircle process is the unique solution of the above equation. This immediately implies that under $\mathbb P^{k_\varphi}$ both $\mu_N$ and $\nu_N$ converge towards the semicircle process $\sigma$  while $\lambda_k$ converges to $\varphi$ for all $k=1, \ldots, \ell$ as $N$ goes to infinity.

The rest of the proof of the lower bound is now identical to the one given in section 4.5 of \cite{donati2012large}.  For $\mu $ a probability measure on $\R$ and $x \in \R$, we define
 \begin{equation} \label{defdrift}
 b(x, \mu) = \int_{-\infty}^x  \frac{ d\mu(y)}{x-y}  
 \in \mathbb{R}_+ \cup \{ \infty\},
 \end{equation}
 and write $b_N=\frac{N-\ell}{N}b$. For $h\in \mathcal H$ and $(\varphi, \mu) \in
 C([0,1]; \R) \times C([0,1]; \mathcal P(\R)) $, we also set
  \begin{eqnarray}
  G_N(\varphi, \mu; h) &=& h(1) \varphi(1)  - h(0) \varphi(0) -
   \int_0^1 \varphi(s) \dot{h}(s) ds - \int_0^1 b_N(\varphi(s), \mu_s)
   h(s) ds,  \label{defG} \nonumber\\
   F_N(\varphi, \mu; h)  &\tri& G_N(\varphi, \mu; h) - \frac{1}{2} \int_0^1 h^2(s) ds \label{def:IloveSushi}.\
  \end{eqnarray}
We also set $F$ to be defined just as $F_{N}$ when $b_{N}$ is replaced by $b$.  
By It\^o, we obtain $$M_1^\varphi = \exp\left[N \sum_{k=1}^\ell F_N\left(\lambda_k,\nu_N;\varphi\right)\right].$$ Now, in short, we get
\begin{eqnarray*}
\mathbb P(\lambda_\ell \in B(\varphi, \delta)) & \geq & \mathbb P\left(\lambda_\ell \in
B(\varphi, \delta); \nu_N \in \mathbb B_r(\sigma, \alpha)\right) \\
&=& \E\left({\bf 1}_{\lambda_\ell \in B(\varphi, \delta);  \nu_N \in
  \mathbb B_r(\sigma, \alpha)} \frac{M^{ k_\varphi}_1}{M^{ k_\varphi}_1}\right) \\
&=& \E^{ k_\varphi}\left[{\bf 1}_{\lambda_\ell \in B(\varphi, \delta);
  \nu_N \in \mathbb B_r(\sigma, \alpha)} \exp\left(-N \sum_{k=1}^\ell F_N(\lambda_k, \nu_N; k_\varphi)\right)\right] \\
&\geq &   \exp\left(-N \ell \sup_{(\psi, \mu) \in C_{\alpha, \delta,r}}
F_N(\psi, \mu; k_\varphi)\right)\\
&& \times \,\mathbb P^{ k_\varphi}\left(\lambda_\ell \in B(\varphi,
\delta);  \nu_N \in \mathbb B_r(\sigma, \alpha)\right)
\end{eqnarray*}
where \begin{equation} \label{defC}
C_{\alpha, \delta, r} =  B(\varphi, \delta)
\times \mathbb B_r(\sigma, \alpha). \end{equation}
Here, the sets $B(\varphi, \delta)$, $\mathbb B_r(\sigma, \alpha)$ are balls in $C[0,1]$ and $C(\mathcal P[0,1];\mathbb R)$ respectively.

\noindent
Therefore as $\nu_N$ converges to the semi-circle process $\sigma$ we obtain
\begin{eqnarray}
\lefteqn{ \liminf_{N \rightarrow \infty} \frac{1}{N}  \ln \mathbb P(\lambda_\ell \in B(\varphi, \delta)) \geq - \sup_{(\psi, \,\mu) \, \in \, C_{\alpha, \delta,r}}  F(\psi, \mu; k_\varphi)  } \nonumber \\
&&  + \qquad \liminf_{N \rightarrow \infty} \frac{1}{N} \ln \mathbb P^{k_\varphi}(\lambda_\ell \in
B(\varphi, \delta);  \nu_N \in \mathbb B_r(\sigma, \alpha)). \label{lowerbound}
\end{eqnarray}
We now repeat the computation in Page 518 of \cite{donati2012large} and we end the proof of the lower bound.
\end{proof}

We now turn to the proof of the large deviation upper bound. There are three steps. First we deal with functions that enter the bulk of the semicircle process. 

\begin{prop}
\label{UB1}
Let $\varphi \in C([0,1]; \R)$ be such that there exists $t_0 \in [0,1]$ so
that
$\varphi(t_0) < 2\sqrt {t_0}.$ Then
$$ \lim_{\delta \downarrow 0 } \lim_{N \rightarrow \infty}
\frac{1}{N} \ln \mathbb P(\lambda_\ell^{N} \in B(\varphi, \delta) ) = - \infty.$$ 
\end{prop}
\begin{proof} Identical to \cite[Proposition 5.1]{donati2012large}.
\end{proof}

Second, we consider functions that stay above the boundary of the semicircle process, that is, functions $\varphi$ that satisfy $\varphi(t) > 2 \sqrt{t}$
 for all $t \in [0,1]$.

\begin{prop}
\label{Secondcase}
Let   $\varphi \in C([0,1]; \R)$ such that  for any $t \in [0,1],$ $\varphi(t) > 2\sqrt t.$
 Then
\begin{equation} \label{UPP2}
\lim_{\delta \downarrow 0 } \limsup_{N \rightarrow \infty}
\frac{1}{N} \ln \mathbb P (\lambda_\ell \in B(\varphi, \delta) ) \leq - I_{\ell}(\varphi).
\end{equation}
\end{prop}
\begin{proof}
This proposition is the equivalent of \cite[Proposition 5.2]{donati2012large}. The proof is almost the same. Set $r =
 \frac{1}{2} \inf(\varphi(t) - 2\sqrt t)$. Consider the following bound for $K > \ell$:
\begin{equation} \label{3terms}
 \mathbb P ( \lambda_\ell \in B(\varphi, \delta)) \leq 
\mathbb(A_{N, \delta, \alpha,K} ) 
+ \mathbb P(\exists t \in [0,1], \lambda_{K+1}(t) > 2 \sqrt{t} + r) + \mathbb P( \mu_N \not\in \mathbb B(\sigma, \alpha)) 
\end{equation}
with 
$$ A_{N, \delta, \alpha,K}  \tri \left\{\lambda_\ell \in B(\varphi, \delta); \forall p>K, \forall t, \lambda_p(t) \leq 2\sqrt{t} + r; \ \mu_N \in \mathbb B(\sigma, \alpha)\right\}.$$

As in \cite[Section 5.1]{donati2012large}, the probabilty of the event  $A_{N, \delta, \alpha,K}$ can be estimated by events of the form
\begin{multline*}
  B_{N, {\bf i}, \delta, \alpha} = \left\{\lambda_\ell \in B(\varphi, \delta), \forall i< i_k,  \forall t \in [t_k, t_{k+1}[,
 \lambda_{i}(t) \geq \varphi(t) - \left(i+ \frac{1}{3}\right)\delta, \right. \\
\left. \lambda_{i_k}(t) - \lambda_{i_k+1}(t)>\frac{2}{3} \delta,
\mu_N \in \mathbb B(\sigma, \alpha)  \right\},
\end{multline*}
where ${\bf i} = (i_1,\ldots,i_R)\in \{1, \ldots,K\}^R,$  $t_{i}$'s form a partition of $[0,1]$, $R$ and $K$ are sufficiently large and $\delta >0$. See \cite[Equation (5.7)]{donati2012large}.
Similar to \cite[Equation (5.8)]{donati2012large}, the proof boils down to show that  for any $K \in \N,$  any $h$ and any subdivision $(t_k)_{1 \le k \le R}$ of $[0,1],$
\begin{equation}\label{limBsN}
 \lim_{\delta \rightarrow 0}  \lim_{\alpha \rightarrow 0} \limsup \frac{1}{N} \ln \mathbb P(B_{N, {\bf i}, \delta, \alpha} ) \leq - \ell F(\varphi, \sigma; h).
\end{equation}
Proceeding as in \cite[Page 22]{donati2012large}, one obtains the bound 
$$ \lim_{\delta \rightarrow 0}  \lim_{\alpha \rightarrow 0} \limsup \frac{1}{N} \ln \mathbb P(B_{N, {\bf i}, \delta, \alpha} ) \leq -  \inf_{ \underline{\psi}  \in \Lambda_{\bf{i}, \delta } } F_{\bf i}(\underline{\varphi} , \underline{\sigma} ; h), $$
where now the set  
\begin{eqnarray*}
 \Lambda_{\bf{i}, \delta} &= &\bigg \{(\psi_1, \ldots \psi_K, \nu_1,\ldots  \nu_K): \psi_i \in B(\varphi, \delta) \; \forall i =1,\ldots, \ell,\\
&& \quad \forall k\le R,
 \forall i< i_k,  \forall t \in [t_k, t_{k+1}[,\psi_{i}(t) \geq \varphi(t) - \left(i+ \frac 1 3\right)\delta, \;  \psi_{i_k}(t) - \psi_{i_{k+1}}(t)>\frac{2}{3} \delta \bigg \}.
 \end{eqnarray*}
The fact that  we now require $\psi_i \in B(\varphi, \delta) \; \forall i =1,\ldots, \ell$ and this implies by taking $\delta$ to zero 
 \begin{equation*}
\lim_{\delta \rightarrow 0}  \lim_{\alpha \rightarrow 0} \limsup \frac{1}{N} \ln\mathbb P( B_{N, {\bf i}, \delta, \alpha}) \leq -  F_{\bf i}(\underline{\varphi} , \underline{\sigma} ; h)
\end{equation*}
where
$$F_{\bf i}(\underline{\varphi} , \underline{\sigma} ; h) = \ell \bigg(h(1)\varphi(1) - h(0) \varphi(0) - \int_0^1 \int_\R \frac{\sigma_t(dx)}{\varphi(t) - x}h(t) dt - \int_0^t \dot h(s)\varphi(s)ds - \sum_k \frac{1}{2i_k} \int_{t_k}^{t_{k+1}}h^2(s) ds\bigg)
$$
and thus
$$ - F_{\bf i}(\underline{\varphi} , \underline{\sigma} ; h)  \leq - \ell F(\varphi, \sigma; h)$$ where 
$F$ is defined by \eqref{def:IloveSushi}. The rest of the proof is identical to \cite[Proposition 5.2]{donati2012large}.

\end{proof}
\noindent
The last item to end the proof of the upper bound is given by the following proposition.
\begin{prop}
\label{lastone}
Let   $\varphi \in C([0,1]; \R)$ such that for any $t \in [0,1],$ $\varphi(t) \geq 2\sqrt t.$
 Then
$$
\lim_{\delta \to 0 } \limsup_{N \rightarrow \infty} \frac{1}{N} \ln \mathbb P(\lambda_\ell \in B(\varphi, \delta) ) \leq - I_\ell(\varphi).$$
\end{prop}
\begin{proof} Identical to \cite[Proposition 5.3]{donati2012large}, replacing the use of \cite[Proposition 5.2]{donati2012large} by Proposition \ref{Secondcase} \end{proof}
The argument in the previous subsection now implies Theorem~\ref{coro:LDP_ell}.

\appendix 

\section{Appendix}

The appendix has three purposes: the first is to analyze the rate function $\mathscr{I}_r$ in greater detail. The second purpose is to gather key inputs to the paper not introduced in section~\ref{sec:notation_inputs}, including the explicit  covariances underpinning Lemma~\ref{lem:rosetta_stone}. The last one is to explain how the methods introduced in the previous sections to prove Theorem \ref{thm:exp_match} also allow us a proof of Theorem \ref{thm:Jesus} (a detailed account of this Theorem will appear in a forthcoming paper). 

\subsection{Further analysis of the rate function $\mathscr{I}_r$} 

The rate function $\scrI_r$ is defined in \eqref{eq:leading_rate} in a piecewise fashion. The next lemma verifies that $\scrI_r$ is continuous.

\begin{lem} Consider the rate functions $J$ and $J_r$ defined in \eqref{eq:leading_rate_1} and \eqref{eq:leading_rate_2}. For all $x,y \geq 2$ and $r \in (0,1)$, we have $J_r(x,y) \equiv J(x,y)$ along the curve $v(x)v(y) = |r|^{2-p}$, where the function $v$ is defined in \eqref{eq:v_def}. 
\label{lem:rate_equal}
\end{lem}

\begin{proof} Write $s \equiv |r|^{p-2}$ for ease of notation, and write $J_r(x,y)$ equivalently as $J_s(x,y)$. It suffices to show $J_s(x,y) - J(x,y) =0$ when $v(x) v(y) = 1/s$. Let us also write $\sfx$ for $v(x)$ and $\sfy$ for $v(y)$. We have
\begin{align}
\label{eq:rate_equal_1}
\Delta &\tri J_s(x,y) - J(x,y) \non \\
&=   \left( \frac{1+s^2}{8(1-s^2)} \right) \left( x^2 + y^2 \right) - \left( \frac{ s}{2(1-s^2)} \right) xy + \frac{1}{2} \log s -  \frac{1}{2} \left[ I_1(x;1) + I_1(y;1) \right].
\end{align}
The relations $x = \sfx + \sfx^{-1}$ and $y = \sfy + \sfy^{-1}$ were leveraged throughout Section~\ref{sec:LDP_pair}. We use these relations below, as well as the following identities:
\begin{align}
I_1(x;1) &= \frac{1}{4} \left( \sfx^2 - \sfx^{-2} \right) - \log \sfx \non \\
I_1(y;1) &=  \frac{1}{4} \left( \sfy^2 - \sfy^{-2} \right) - \log \sfy.\non
\end{align}
which we plug into \eqref{eq:rate_equal_1} above:
\begin{align}
8(1-s^2)\Delta &= \left(1+s^2 \right) \left[ \left( \sfx + \sfx^{-1} \right)^2 + \left( \sfy + \sfy^{-1} \right) \right] - 4s \left( \sfx + \sfx^{-1} \right) \left( \sfy + \sfy^{-1} \right) + 4(1-s^2) \log s \non\\
&\quad\quad\quad - 4\left(1-s^2\right) \left[ \frac{1}{4} \left( \sfx^2 - \sfx^{-2} \right) + \frac{1}{4} \left( \sfy^2 - \sfy^{-2} \right) - \log \sfx \sfy \right] \non\\
&= \left(1+s^2\right) \left[ \left( \sfx + \sfx^{-1} \right)^2 + \left( \sfy + \sfy^{-1} \right) \right] - 4s \left( \sfx + \sfx^{-1} \right) \left( \sfy + \sfy^{-1} \right) \non\\
&\quad\quad\quad\quad\quad\quad\quad\quad\quad\quad - 4\left(1-s^2\right) \left[ \frac{1}{4} \left( \sfx^2 - \sfx^{-2} \right) + \frac{1}{4} \left( \sfy^2 - \sfy^{-2} \right)  \right] \,, \non
\end{align}
where we've used the relation $\sfx \sfy = 1/s$ to cancel the log-terms going from the first line to the second. Expanding what is written directly above, we have
\begin{align}
8(1-s^2)\Delta&= \left( 1 + s^2 \right) \left[ \sfx^2 + \sfx^{-2} + \sfy^2 + \sfy^{-2} + 4 \right] - 4s \left[ \sfx \sfy + \sfx \sfy^{-1} + \sfx^{-1} \sfy + ( \sfx \sfy)^{-1} \right] \non\\
&\quad\quad\quad\quad\quad\quad\quad\quad\quad\quad - \left( 1-s^2 \right) \left[ \sfx^2 - \sfx^{-2} + \sfy^2 - \sfy^{-2} \right] \non \\
&= \left( \left(1 + s^2\right) - \left( 1-s^2 \right) \right) \left[ \sfx^2 +  \sfy^2 \right] + \left( \left(1 + s^2\right) + \left( 1-s^2 \right) \right) \left[ \sfx^{-2} + \sfy^{-2} \right] \non \\
&\quad\quad\quad\quad\quad\quad\quad\quad\quad\quad  - 4s  \left[ \sfx \sfy + \sfx \sfy^{-1} + \sfx^{-1} \sfy + ( \sfx \sfy)^{-1} \right] + 4\left( 1 + s^2 \right) \non \\
&= \left( 2s^2 \right) \left[ \sfx^2 +  \sfy^2 \right] + 2  \left[ \sfx^{-2} + \sfy^{-2} \right] - 4s  \left[ \sfx \sfy + \sfx \sfy^{-1} + \sfx^{-1} \sfy + ( \sfx \sfy)^{-1} \right] + 4 \left( 1+s^2 \right)\,. \non
\end{align}
To simplify the above further, we again use that $x$ and $y$ satisfy $\sfx \sfy = 1/s$, we replace the negative powers of $\sfx$ and $\sfy$ above via the relations $\sfx^{-1} = s \sfy$ and $\sfy^{-1} = s \sfx$. Continuing from the last display, 
\begin{align}
8(1-s^2)\Delta&= \left(4s^2 \right) \left[ \sfx^2 + \sfy^2 \right] - \left(4s \right) \left[ \frac{1}{s} + s \sfx^2 + s \sfy^2 + s \right] + 4(1+s^2) \equiv 0\,,\non
\end{align}
completing the proof. \end{proof}

We now determine where $\scrI_r (x,y)$ is minimized, thinking of $r$ as fixed. When applying the LDP associated to these rate functions, we constrain $x$ and $y$ separately. For our purposes, it suffices to consider this optimization taking place over domains of the form $[ u_1, \infty) \times [u_2, \infty)$. The next lemma handles a special case of this. 

\begin{lem} For $r \in (-1,1)$, and for any $u > 2$, the function $\mathscr{I}_r(x,y)$ defined in \eqref{eq:leading_rate}, restricted to the domain $[u, \infty) \times [u, \infty)$ achieves its minimum at the pair $(u,u)$. 
\label{lem:rfm_diag} 
\end{lem}

\begin{proof} The lemma clearly holds for $J(x,y) = I_1(x;1) + I_1(y;1)$ defined in \eqref{eq:leading_rate_1}. This follows from the fact that $z \mapsto I_1(z;1)$ is strictly increasing over $(2, \infty)$. 

Examining the form of $J_r(x,y)$ in \eqref{eq:leading_rate_2}, it will suffice to show the lemma holds for the last two summands in this expression. When put in terms of the variable $s$ using the equivalence $s \equiv \sqrt{q} \equiv |r|^{p-2}$, we call the sum of these two terms $T_s(x,y)$:
\begin{align}
T_s(x,y) \tri \frac{1+s^2}{8(1-s^2)} \left(x^2 + y^2\right) - \frac{sxy}{2(1-s^2)} \,.
\label{eq:T_xy_def}
\end{align}
Considering $x$ and $z$ such that $x+z, x-z > 2$, observe that 
\begin{align}
T_s(x +z, x-z) &=  \frac{1+s^2}{8(1-s^2)}\left[ \left(x+z\right)^2 + \left(x - z \right)^2\right] -  \frac{s}{2(1-s^2)} \left[ \left( x + z \right)\left( x -z \right) \right]  \non \\
&= \frac{1 +s^2}{ 4(1-s^2) } \left(x^2 + z^2\right)- \frac{2s }{4(1-s^2) } \left[ x^2 -z^2\right] \non \\
&= \frac{(1-s)^2x^2}{4(1-s^2)} + \frac{(1+s)^2}{4(1-s^2)} z^2\,. \label{eq:T_x_z}
\end{align}
Treating $x$ as fixed, the function $z \mapsto T_s(x+z, x-z)$ is minimized at $z =0$, and hence $T_s(x,y)$ itself is minimized along the diagonal $y=x$. Along the diagonal, we have 
\begin{align}
T(x,x) = \frac{(1-s)^2}{4(1-s^2)} x^2\,,
\end{align}
a strictly increasing function of $x$ for $x \geq u > 2$, completing the proof.  \end{proof}

The next lemma complements the one above. 

\begin{lem} 
\label{lem:rectangle}
For $r \in (-1,1)$, and for any $u_2 > u_1 > 2$, the function $\mathscr{I}_r(x,y)$ defined in \eqref{eq:leading_rate}, restricted to the domain $[u_1, \infty) \times [u_2, \infty)$ achieves its minimum at the pair $(u_*\vee u_1,u_2)$, where $u_*$ satisfies
\begin{align}
v(u_*) = s v(u_2) \,, \non
\end{align}
and with the function $v$ from \eqref{eq:v_def}. A symmetric statement holds in the case $u_1 > u_2 > 2$. 
\end{lem}

\begin{proof} As above, write $s$ for $\sqrt{q} \equiv |r|^{p-2}$. It suffices to verify the lemma for $J_s(x,y)$:
\begin{align}
J_s(x,y) = \frac{1}{2} \left( I_1(x;1) + I_1(y;1) \right) + \frac{1}{2} \log s + \frac{1+s^2}{8(1-s^2)} \left(x^2 +y^2\right) - \frac{sxy}{ 2(1-s^2)} .
\label{eq:rectangle1}
\end{align}
Treating $s$ as fixed, we find critical points of this function. The proof has three short steps.
\begin{itemize}
\item[(1)] Using the notation $\sfx \equiv v(x)$ and $\sfy \equiv v(y)$ once again, we show $\pa_xJ_s(x,y) =0$ iff $\sfx =s\sfy$.
\item[(2)] We apply (1) to show $J_s(x,y)$ achieves its minimum on the ray $\{y = u_2\} \cap \{ x \geq u_1\}$.
\item[(3)] We conclude using the first two parts. 
\end{itemize}

\vspace{4mm}

\emph{Step 1:} Writing $\pa_x$ for a derivative in $x$, note that
\begin{align}
\pa_x J_s(x,y) = \frac{1}{2} \pa_x \left[ I_1(x;1) \right] + \left(\frac{ 1+s^2}{4(1-s^2)} \right) x -  \left( \frac{2s}{4(1-s^2)}\right) y.\non
\end{align}
Recalling the integral representation $I_1(x;1) = \tfrac{1}{2} \int_2^x \sqrt{z^2-4} dz$, we have
\begin{align}
\pa_xJ_s(x,y) &= \left( \frac{1}{4(1-s^2)} \right) \left[ \left(1-s^2 \right) \sqrt{x^2-4} + \left(1+s^2\right)x - \left(2s\right) y \right]\,,
\label{eq:pa_x_J}
\end{align}
from which one sees $\pa_x J_s(x,y) = 0$ if and only if 
\begin{align}
y &= \left( \frac{ 1-s^2}{2s} \right) \sqrt{ x^2-4} + \left( \frac{ 1+ s^2}{2s} \right) x \non\\
&= \left( \frac{ 1-s^2}{2s} \right) \left[ x + \sqrt{x^2-4} \right] + sx \non\\
&= \left( \frac{1-s^2}{s} \right) \sfx + s \left( \sfx + \sfx^{-1} \right) \non\\
&= \frac{\sfx}{s} + \frac{s}{\sfx} \,,\non
\end{align}
and the only values of $\sfx$ which can satisfy the above equality are (via the quadratic formula) either $s \sfy$ or $s\sfy^{-1}$. The latter is impossible: were we to have $\sfx = s\sfy^{-1}$, it would follow that $\sfx \sfy = s$, where $s \in [0,1)$ and $\sfx, \sfy > 1$ (the lower bounds on $\sfx$ and $\sfy$ follow directly from the constraints $u_1,u_2 >2$). Thus $\pa_xJ_s(x,y) =0$ iff $\sfx =s\sfy$, completing the first step. 

\emph{Step 2:} A symmetric argument implies the function $J_s(x,y)$ has a critical point on the interior of $[u_1, \infty) \times [u_2, \infty)$ when $\sfx = s\sfy$ and $\sfy = s\sfx$. Were both these constraints to hold, we would have $ \sfy =s^2 \sfy$,  impossible for $s \in [0,1)$. 

Moreover, we showed in the proof of Lemma~\ref{lem:rfm_diag} that $J_s(x,y)$ is minimized along the diagonal $y = x$, and that the function is strictly increasing along the diagonal. Using this fact with what we have just shown, it follows that $J_s(x,y)$ achieves its minimum over $[u_1, \infty) \times [u_2, \infty)$ on the boundary of this region: either the ray $\{x = u_1\} \cap \{y \geq u_2\}$ or on the ray $\{y = u_2\} \cap \{ x \geq u_1\}$. 

To complete the second step, we rule out the former ray. Use the above computation: we have shown that $\pa_y J_s(u_1,y) =0$ at $y_*$, the $y$-value satisfying $v(y_*) =s v(u_1)$. As $z \mapsto v(z)$ is strictly increasing on $(2,\infty)$, that $v(y_*) < v(u_1)$ implies $y_* < u_1$, and hence $y_* < u_2$. Using the symmetry of $J_s(x,y)$ in $x$ and $y$, and examining \eqref{eq:pa_x_J}, we have
\begin{align}
\pa_{yy} J_s(u_1,y) = \frac{y}{4 \sqrt{y^2-4}} + (1+s^2) > 0\,,
\label{eq:toaster}
\end{align}
which implies that the function $y \mapsto J_s(u_1,y)$ restricted to the interval $(y_*, \infty)$ is increasing.

\emph{Step 3:} By step 2 it suffices to restrict $J_s(x,y)$ to $\{y = u_2\} \cap \{ x \geq u_1\}$. Using step 1 and the symmetry of $J_s(x,y)$ with \eqref{eq:toaster}, we find this restricted function has a minimum when $v(x) = sv(u_2)$, completing the proof. \end{proof}

Having studied $\scrI_r$ at fixed $r$, we next fix the energies in the argument of $\scrI_r$ and establish a relationship between its two piecewise components as $r$ varies. 

\begin{lem} 
\label{lem:rate_fn_dom}
For a fixed $u \geq 2$, consider the functions $J$ and $J_r$, defined in \eqref{eq:leading_rate_1} and \eqref{eq:leading_rate_2}. For all $r$ with $|r|^{p-2} \in ((v(u))^{-1}, 1)$, we have $J_r(u,u) \leq J(u,u)$.
\end{lem}

\begin{proof}
Continuing to write $|r|^{p-2}$ as $s$, we recall the expression $J_s(x,y)$ for $J_r(x,y)$ under this change of variables is given in \eqref{eq:rectangle1}, and we remark that $J(x,y)$ has no dependence on $s$. Expressions for $J_s(x,y)$ and $J(x,y)$ simplify after setting $x = u$ and $y = u$:
\begin{align}
J_u(s) \equiv J_r(u,u) &= I_1(u;1)  + \frac{1}{2} \log s + \left( \frac{1-s}{4(1+s)} \right)u^2 \,,\\
J_u \equiv J(u,u) & = 2 I_1(u;1) \,.
\end{align}
We have changed our notation slightly in the above display to emphasize that $u$ is fixed. Consider the difference $D_u(s)  \tri J_u(s) - J_u$, 
\begin{align}
D_u(s) = \frac{1}{2} \log s + \left( \frac{1-s}{4(1+s)} \right) u^2 - I_1(u;1) .
\label{eq:increase_in_r}
\end{align}
To prove the lemma, it suffices to show $\pa_s D$ is non-positive on the interval $s \in [1/v(u),1]$.

A short computation, using the identity $u = v(u) + v(u)^{-1}$ (once again), shows $\pa_sD$ is zero when
\begin{align}
s = v(u)^2 \text{ or } s = 1/v(u)^2 \,.
\end{align}
The former is not relevant as $s \in (0,1)$ and $v(u) \geq 1$. The latter root implies $\pa_s D = 0$ at the point $s = 1/v(u)^2$, exactly where $J_r = J$ by Lemma~\ref{lem:rate_equal}. Another short computation shows $\pa_s^2 D < 0$ at the point $s = 1 /v(u)^2$, which completes the proof. \end{proof}

\subsection{Covariance structures}{\label{subsec:cov}} For $1 \leq i \leq 4$ and any $r \in (-1,1)$, define
\begin{align}
a_1(r) &\tri \left[ p(1-r^{2p-2}) \right]^{-1} \label{eq:a1_def}\\
a_2(r) &\tri \left[ p[1-(r^p - (p-1)r^{p-2} (1-r^2))^2] \right]^{-1} \label{eq:a2_def}\\
a_3(r) &\tri  -r^{p-1} a_1(r) \label{eq:a3_def}\\
a_4(r) &\tri (-r^p + (p-1) r^{p-2} (1-r^2) ) a_2(r)\,, \label{eq:a4_def} 
\end{align}
and also define
\begin{align}
b_1(r) &\tri -p +a_2(r) p^3r^{2p-2}(1-r^2) \label{eq:b1_def}\\
b_2(r) &\tri -pr^p -a_4(r) p^3 r^{2p-2} (1-r^2) \label{eq:b2_def}\\
b_3(r) &\tri a_2(r) p^2(p-1) r^{2p-4} (1-r^2) [ -(p-2) + pr^2 ] \label{eq:b3_def}\\
b_4(r) &\tri p(p-1) r^{p-2} (1-r^2) - a_4(r) p^2(p-1)r^{2p-4}(1-r^2) [ -(p-2) + pr^2 ] \,. \label{eq:b4_def} 
\end{align}
 
The $a_i(r)$ and $b_i(r)$ are used to describe the covariances (and constants) present in Lemma~\ref{lem:rosetta_stone}. We start with the constants $m_i$: these are the additional perturbations made to the last entry of each Hessian, and are defined as follows.
\begin{align}
m_1(r,u_1,u_2) & \tri \lbob \begin{matrix} b_3(r) & b_4(r)  \end{matrix} \rbob  \bm{\Sigma}_U(r)^{-1} \lbob \begin{matrix} u_1 \\ u_2 \end{matrix} \rbob  \\
m_2(r,u_1,u_2) &\tri m_1(r,u_2,u_1)  \,,
\label{eq:little_m}
\end{align}
with 
\begin{align}
\bm{\Sigma}_U(r) \tri -\frac{1}{p} 
\lbob\begin{matrix} 
b_1(r) & b_2(r) \\ 
b_2(r) & b_1(r) 
\end{matrix} \rbob .
\label{eq:Sigma_U}
\end{align}

The covariance matrix $\bm{\Sigma}_Z(r)$ describes the joint law of the last column of each Hessian in Lemma~\ref{lem:rosetta_stone}. It is a $2 \times 2$ matrix describing the random vectors in the last column of each Hessian, and its four entries are given as follows. 
\begin{align}
\bm{\Sigma}_{Z,11}(r) \equiv \bm{\Sigma}_{Z,22}(r) &\tri p(p-1) -a_1(r) p^2 (p-1)^2 r^{2p-4} (1-r^2) \non \\
\bm{\Sigma}_{Z,12}(r) \equiv \bm{\Sigma}_{Z,21}(r) &\tri p(p-1)^2 r^{p-1} - p(p-1) (p-2) r^{p-3} + a_3(r) p^2 (p-1)^2 r^{2p-4} (1-r^2) 
\label{eq:Sigma_Z} 
\end{align}

The covariance matrix $\bm{\Sigma}_Q(r)$ describes the random variables in the last entry of each Hessian. Its diagonal entries are given by
\begin{align}
\bm{\Sigma}_{Q,11}(r) = \bm{\Sigma}_{Q,22}(r) & \tri 2p(p-1)  - a_2(r) (1-r^2) [ p(p-1)r^{p-3}(pr^2 - (p-2))]^2  \non\\
&\quad - \lbob \begin{matrix} b_3(r) & b_4(r)  \end{matrix} \rbob  \bm{\Sigma}_U(r)^{-1} \lbob \begin{matrix} b_3(r) \\ b_4(r) \end{matrix} \rbob ,
\label{eq:Sigma_Q_11}
\end{align}
\begin{align}
\bm{\Sigma}_{Q,12}(r) \equiv \bm{\Sigma}_{Q,21}(r) &\tri p^4 r^p - 2p(p-1) (p^2 -2p +2)r^{p-2} + p(p-1)(p-2)(p-3)r^{p-4} \non \\
&\quad + a_4(r) p^2 r^{2p-6} (1-r^2) (p^2 r^2 - (p-1) (p-2) )^2 \non \\
&-  \lbob \begin{matrix} b_1(r) + b_3(r) & b_2(r) + b_4(r)  \end{matrix} \rbob  \bm{\Sigma}_U(r)^{-1} \lbob \begin{matrix} b_2(r) + b_4(r) \\ b_1(r) + b_3(r)  \end{matrix} \rbob
\label{eq:Sigma_Q_12}.
\end{align}

It will be convenient to have the eigenvalues of $\bm{\Sigma}_U(r)$ on hand.

\begin{lem} For any $r \in (-1,1)$, the covariance matrix $\bm{\Sigma}_U(r)$ x has eigenvectors $\lbob 1, 1 \rbob^T$ and $\lbob 1,-1 \rbob^T$, with respective eigenvalues
\begin{align}
\sigma_1(r) &\tri \frac{1 + (p-1)r^{p-2}(1-r^2)-r^{2p-2}}{1 + (p-1)r^{p-2}(1-r^2) -r^p} \label{eq:U_sig_1}, \\
\sigma_2(r) &\tri \frac{ 1- (p-1)r^{p-2} (1-r^2) - r^{2p-2}}{1-(p-1)r^{p-2}(1-r^2) + r^p} . \label{eq:U_sig_2}
\end{align} 
\label{lem:U_eval}
\end{lem}

\subsection{Proof of Theorem \ref{thm:Jesus}}

%%%%%%%%%%%%%%%%%%%%%%%%%%%%%%%%%%%%%%%%%%%%%%%%%%%
%
We set up notation to state Lemma~\ref{lem:sub18}, Lemma~\ref{lem:sub19}. They are used to sharpen Theorem~\ref{thm:exp_match} and provide a proof of Theorem \ref{thm:Jesus}. We return to the convention in 
\eqref{eq:u_bar_def}: for $z \in \R$, we write
$ \gamma _ p \, z / \sqrt { N - 1 } $ as
$ \bar { z } $. As usual, 
$ \gamma_p 
	\equiv 
		\sqrt{ p / ( p - 1 ) }$. 
		
		For 
$ u 
	< 
		-\sfE_\infty $, define
\begin{align}
\mathfrak { S } _{(\ell)} ( u ) 
	&\tri (\ell +1)
		\int 
			\frac { 1 }
				{ 
					\sqrt{ 
						\tfrac { p - 1 }
							{ p }
					}
					\lambda - u 
				} 
		\semi ( d \lambda ) 											\non \\ %
	&\equiv 
		\gamma _ p (\ell + 1)
			\int 
				\frac { 1 }
					{ \lambda - \gamma_p u } 
			\semi ( d \lambda ) \, ,
	\label{eq:frakS_def}%					eq:frakS_def
\end{align}								%%%
so that for such 
$ u $, 
\begin{align*}								%%%
\frac { d }
	{ du } 
	\Sigma _ {p,  \ell} ( u ) 
		= 
			- \leftp
				\mathfrak { S } _{(\ell)} ( u ) + \frac{2(p-1) + \ell p}{2(p-1)}u 
			\rightp 
				> 
					0 \,. 
\end{align*}								%%%

Finally, for 
$ \matX 
	\sim 
		\goe _ { N - 1 } $ given and 
$ u 
	< 
		- \sfE _ \infty $, define 
\begin{align}								%%%
\label{eq:frak_C_N_def}%						eq:frak_C_N_def
\mathfrak { C } _ N ( u ) 
	\tri 
		\omega_N 
		\sfC _N ^ { 1 / 2 } 
		\sqrt{ 
			\frac { N } 
				{ 2 \pi } 
		} 
		e ^ { 
			- N ( u ^ 2 / 2 )
		} 
			\E 
				\det 
					\leftp
						\matX - \sqrt{N} \bar{u} \matI 
					\rightp \,,
\end{align}								%%%
recalling $\omega_N$ and $\sfC_N$ from \eqref{eq:omega_N_def} and \eqref{eq:C_N_Kac_Rice}.
\begin{rmk}								%%%%%%%%%
\label{rmk:sec6_setting}%						rmk:sec6_setting
In the three results below, we consider a distinguished 
$ u _ * 
	< 
		- \sfE _ \infty $ and a sequence of intervals 
$ J _ N = (a_N, b_N) $, whose endpoints satisfy
$ a _ N , b _ N \to u _ * $ as 
$ N \to \infty $. The next two results concern either critical points of a given index with energies in these shrinking intervals, or pairs of such points having overlap in $ ( - \rho _ N , \rho _ N ) $, for some sequence 
$ \rho _ N \downarrow 0 $. 
\end{rmk}						

Set 
\[
c_{\ell,p} = \frac{2(p-1) + \ell p}{2(p-1)}.
\]
			%%%%%%%%%

Lemma~\ref{lem:sub18} plays the role of Theorem
~\ref{thm:ABC} in this setting.
\begin{lem}								%%%%%%%%%
\label{lem:sub18}%							lem:sub18
Let 
$ u _ * 
	< 
		- \sfE _ \infty $ and consider a sequence of intervals 
$ J _ N 
	= 
		( a _ N , b _ N ) $ with 
$ a _ N $ and 
$ b _ N $ tending to 
$ u _ * $ as 
$ N \to \infty $. With $ \mathfrak { S } _{(\ell)}( u ) $ and $ \mathfrak { C } _ N ( u ) $  as in 
\eqref{eq:frakS_def} and \eqref{eq:frak_C_N_def} 
, as $N \to \infty$ 
\begin{align*}								%%%
\E 
	\leftp 
		\crit_{N,\ell} ( J _ N )  
	\rightp 
		= 
			( 1 + o ( 1 ) ) 
			\mathfrak { C } _ N ( b _ N ) 
			\int _ { J _ N } 
				\exp 
					\Big( 
						- N ( \,
							c_{\ell,p} u_* + \mathfrak { S }_{(\ell)} ( u _ * ) 
						\, ) \cdot (
							v - b _ N
						) 
					\Big) 
			\, dv.
\end{align*}								%%%
\end{lem}									%%%%%%%%%
%
						%%%
%
Lemma~\ref{lem:sub19} extends Lemma
 ~\ref{lem:sub18} to pairs of nearly orthogonal critical points.

\begin{lem}								%%%%%%%%%
\label{lem:sub19} 
Let 
$ u _ * 
	< 
		- \sfE _ \infty $ and 
$ J _ N 
	= 
		( a _ N , b _ N ) $ be as in Remark
~\ref{rmk:sec6_setting} for this energy, and let 
$ \rho _ N \downarrow 0 $.
With 
$ \mathfrak { S } _{(\ell)} ( u ) $ 
and $ \mathfrak { C } _ N ( u ) $ 
as in 
\eqref{eq:frakS_def}
and
\eqref{eq:frak_C_N_def}, and using the notation 
\eqref{eq:abbrev_ov_int}, as $N \to \infty$,
\begin{align*}								%%%
\E 
	\leftcb
		\critNell ( J _ N ) 
	\rightcb _ { \bm { 2 } }^{\rho_N} 
		\leq 
			( 1 + o ( 1 ) ) 
			\leftp 
				\mathfrak { C } _ N ( b _ N ) 
					\int _ { J _ N } 
						\exp 
							\Big( 
								- N ( \, 
									c_{\ell,p}u_* + \mathfrak{S}_{(\ell)}(u _ *) 
								\, ) \cdot (
									v - b _ N 
								) 
							\Big)
					\, dv 
			\rightp^2 .
\end{align*}								%%%
\end{lem}									%%%%%%%%%
We now prove Theorem \ref{thm:Jesus} assuming Lemma~\ref{lem:sub18}, and Lemma~\ref{lem:sub19}.
\begin{proof}
[ 
	Proof of Theorem
~\ref{thm:Jesus} 
]
It will suffice to show that, for $p \geq 3$ and $u_* \in (-\sfE_\ell, - \sfE_\infty)$,
\begin{align}
\label{eq:thm_jesus}
\lim_{N \to \infty} 
	\frac{
	\E \left[ \CritNell ( \, ( -\infty, u_*) \, ) \right]^2
	}
	{
	\left[
	\E \CritNell ( \, (- \infty, u_*) \, ) 
	\right]^2
	}
		=
			1.
\end{align} 
To this end, we first use Theorem~\ref{thm:ABC} and that $u \mapsto \Sigma_{p,\,\ell}(u)$ is strictly increasing over the interval $(-\infty, -E_\infty)$: there is a positive sequence $\eps_N \downarrow 0$ such that:
\eq{
 \lim_{N\to \infty} 
		\frac{\E \critNell ( \, (u_* -\eps_N, u_*) \, ) }{\E \critNell( \, (-\infty, u_*) \, ) } = 1.
}
Using the above display with the trivial inequality $\E \CritNell ( \, (-\infty, u_*) \,) \geq \E \CritNell ( \, (u_* -\eps_N , u_* ) \, )$, \eqref{eq:thm_jesus} will follow if
\eq{
\lim_{N \to \infty} 
	\frac{
	\E \left[ \critNell ( \, (-\infty, u_*) \, ) \right]^2}
	{
	\left[ \E \critNell ( \, (u_* - \eps, u_*) \, ) \right]^2	
	}
	\leq 1
}
By Lemma~\ref{lem:sub20}, without loss of generality, we may use the same sequence $\eps_N$ and find another positive sequence $\rho_N \downarrow 0$ so that
\eq{
\lim_{N \to \infty} 
	\frac{
	\E 		\leftcb
			\crit _ N ( \, (- u_* - \eps_N, u_*) \,  ) 
		\rightcb_ { \bm { 2 } }^{\rho_N}  }
		{
			\E \left[ \critNell(\, (- \infty, - u_*) \, ) \right]^2 
		}
		&=1,
}
and thus \eqref{eq:thm_jesus} is implied by
\eq{
\lim_{N \to \infty}
	\frac{
	\E \{ \CritNell( \, (u_* - \eps _ N, u_* ) \, ) \}_{\bm{2}}^{\rho_N}
	}
	{
	\left[ \E \critNell ( \, (u_* -\eps _ N, u_*) \, ) \right]^2	
	}
		\leq 1,
}
which follows directly from  Lemma~\ref{lem:sub18} and \ref{lem:sub19}, completing the proof. \end{proof}

%The proofs of Lemma~\ref{lem:sub18} and \ref{lem:sub19} require two steps: the first is similar to the index transfer performed in the course of proving Theorem~\ref{thm5}. The second step furnishes an estimate which effectively decouples nearly independent Hessians (corresponding to very small $|r|$). The strategy is similar to the one implemented by Subag in \cite{subag2017complexity}. Their proof will follow from the  Lemmas \ref{lem:sub25} and \ref{lem:sub24} below. We omit the computational details as these steps are long and tedious. 

Let 
$ \rho _ N 
	\downarrow 
		0 $, and let 
$ \matX _ { \iid, \, 0 } $, 
$ \matX _ { \iid, \, 1 } $ and 
$ \matX _ { \iid, \, 2 } $ be three i.i.d. 
$ \goe _ { N - 1 } $ random matrices. For 
$ i 
	= 
		1 , 2 $ and 
$ r 
	\equiv 
		r _ N \in ( - \rho _ N , \rho _ N ) $, set
\begin{align}									%%%
\label{eq:l24prepre} %							eq:l24prepre
\matX _ i ( r )  
	\tri 
		\sqrt{ 
			1 - | r | ^ { p - 2 } 
		} 
		\matX _ { \iid , \, i} 
		+ (\sgn(r))^{ip} 
		\sqrt{ 
			| r | ^ { p - 2 } 
		} 
		\matX _ { \iid , \, 0 } \,.
\end{align}									%%%
From Lemma
~\ref{lem:rosetta_stone}, if one restricts the pair 
$ ( \matX _ 1 ( r ) , \matX _ 2 ( r ) ) $ to the pair of $(N-2) \times (N-2)$ principal minors, the resulting random matrix pair has the correlation structure 
\eqref{eq:GOE_joint}, suggesting a natural and useful coupling.

Let 
$ \Xminor $ denote the 
$ ( N - 2 ) \times (N-2)$ principal minor of 
$ \matX _ i ( r ) $, and let 
$ \Xzout $ denote the 
$(N - 1) 
	\times 
		( N - 1 )$ matrix obtained from 
$ \matX _ i ( r ) $ by setting each matrix element of 
$ \matX _ {\minor, \, i} (r) $ within
$ \matX _ i ( r ) $ to zero. The only non-zero entries of 
$ \Xzout $ are in the last row or the last column. Couple the matrices 
$ \matX _ i ( r ) $ and 
$ \un { \matM } _ { \, i } ( r ) $ together so that (i) -- (iii) below hold almost surely. 
\begin{itemize}								%%%%%%%%%
\item[(i)]%									item (i)
Recalling that 
$ \un { \matM } _{ \, i } ( r ) $ is a shift of 
$ \matM _ i ( r ) $, whose block structure has the form 
\eqref{eq:hess_block_structure}, we couple the largest block of 
$ \matM _ i ( r ) $ to 
$ \Xminor $ so that 
$ \Xminor 
	= 
		\matG _ i ( r ) $ almost surely. 
\item[(ii)]%								item (ii)
Having coupled most of 
$ \matX _ i ( r ) $ to most of 
$ \matM _ i ( r ) $, we couple the column
$ Z _ i ( r ) $ in 
\eqref{eq:hess_block_structure} to the last column of 
$ \matX _ i ( r ) $ so that, for
$ j 
	= 
		1 , \dots , N - 2 $, 
\begin{align*}								%%%
\lbob 
	Z _ i ( r ) 
\rbob _ j 
	&= 
		\sqrt{
			\frac{ 
				\bm { \Sigma } _ { Z , \, 11 } (r) - | \bm { \Sigma } _ { Z , \, 12 } (r) | 
			}
			{
				p ( p - 1 )
			} 
		}
		\lbob
			\matX _ {\iid , \, i } 
		\rbob _ { j , \, N - 1 } 
		+ \leftp 
			\sgn \leftp 
				\bm { \Sigma } _ { Z , \,12 } ( r ) 
			\rightp 
		\rightp ^ i  
		\sqrt{ 
			\frac{ 
				| \bm { \Sigma } _ { Z , \, 12 } ( r ) | 
			}
			{ 
				p ( p -1 ) 
			} 
		} 
		\lbob 
			\matX _ { \iid, \, 0 } 
		\rbob _ { j , \, N- 1 } 
\end{align*}								%%%
\item[(iii)]%								item (iii)
We finally couple the matrix element 
$ Q _i ( r ) $ in 
\eqref{eq:hess_block_structure} to the last element of each 
$ \matX _ i ( r ) $ so that
\begin{align*}								%%%
Q _ i ( r ) 
	&= 
		\sqrt{
			\frac { 
				\bm { \Sigma } _ { Q , \, 11 } (r) - | \bm { \Sigma } _ { Q , \, 12 } ( r ) | 
			}
			{
				p ( p - 1 )
			} 
		} 
		\lbob
			\matX _ {\iid , \, i }
		\rbob_{N-1,\,N-1} 
		+ \leftp 
			\sgn 
				\leftp \bm { \Sigma } _ { Q , \, 12 } ( r ) 
			\rightp 
		\rightp ^ i  
		\sqrt{ 
			\frac{ | \Sigma_{Q,\,12} (r) | }{ p (p-1) } 
		} 
		\lbob
			\matX _ { \iid , \, 0 } 
		\rbob _ { N - 1 , \, N - 1 } 
\end{align*}								%%%
\end{itemize}								%%%%%%%%%

Define the matrix
\begin{align}								%%%
\label{eq:T_i_def}%							eq:T_i_def
\matT _ i ( r ) 
	\tri 
		\leftp 
			\begin{matrix}
				\bm { 0 } & Z _ i ( r ) \\
				Z_i ( r ) ^ T &  Q _ i ( r ) 
			\end{matrix}
		\rightp
- \Xzout, 
\end{align}								%%%
so that:
\begin{align}								%%%
\label{eq:M_X_perturb}
\matM _ i ( r ) 
	= 
		\matX _ i ( r ) 
		+ \matT _ i ( r ) \,. 
\end{align}								%%%

The next lemma is analogous to \cite[Lemma 25]{ 
	subag2017complexity
 }. 
\begin{lem}									%%%%%%%%%
\label{lem:sub25}%								lem:sub25
Given $r \in [-1,1]$, define $\rho(r) \tri \sgn(r) \sqrt{ |r|^{p-2} }$. For $i= 1,2$, let $\matW_i \equiv \matW_i(r)$ be 
$ (N-1) 
	\times 
		(N-1) $ jointly Gaussian matrices such that 
\begin{align}									%%%
\label{eq:sub25_cov}
\matW_i \equiv \matW_i(r) \sim \sqrt{N-1} \matM_i(r),
\end{align}									%%%
with the $\matM_i$ as in Lemma~\ref{lem:rosetta_stone}. Let 
$ g : 
	\R ^ { (N-1) \times (N-1) } 
		\to 
			\R $ be the function mapping matrices $\R^{ (N - 1) \times (N-1 ) } \ni \bm{A}$ to their determinants $\det ( \bm{A} )$, and define
$\hat{g}(r) 
	\tri 
		\E 
			( \,
				g \leftp 
					\matW_1
				\rightp  
				\cdot 
				g \leftp
					\matW _ 2 
				\rightp  
			\,).$								%%%
Letting $\til{g}$ denote the equivalent function of $\rho$ under the above change of variables, one has that $\til{g} : [-1,1] \to \R$ is a polynomial in $\rho$ satisfying
\begin{align}									%%%
\label{eq:sub25_state_1}%						eq:sub25_state_1
		\til{ g } ( \rho ) 
		- \til{g}(0) 
			\leq \rho ( \til{g}(1) - \til{g}(0) )  
\end{align}									%%%
for any 
$ \rho \in [ 0 , 1 ] $.
\end{lem} 		

\begin{proof} The method of proof is the same as that of \cite[Lemma 25]{subag2017complexity}, whose statement considers more general functions $g$. In that case, the random matrices considered have a much simpler covariance structure. The lemma is shown, in either case, by differentiating $\til{g}$ in the parameter $\rho$. In \cite{subag2017complexity}, the simple covariance structure of the matrices considered leads many terms to vanish after differentiation. The covariances in our setting also simplify under the change of variables $r \mapsto \rho$, but not to the same extent, and we counteract this difficulty by specializing $g$ to the determinant function. Using the chain rule to take a derivative in $\rho$, the determinants inside the expectation are differentiated with respect to the matrix entries of the $\matW_i$. These derivatives are computed easily, expanding by minors, and we leave the details of the proof to the reader. \end{proof}			

We now summarize the rest of the argument, leaving the details to a forthcoming paper. A similar decoupling lemma is in some sense the last step of Subag's argument, specifically in the proof of \cite[Lemma 19]{subag2017complexity}. In our case, it is necessary to apply Lemma~\ref{lem:sub25} sooner. As a consequence, we show that for $r = r_N \in (-\rho_N, \rho_N)$, 
\eq{
\E \left( \prod_{i=1,\,2} \det \left( \un{\matM}_{\,i}(r,u_1,u_2) \right) \cdot \1\{ \Eind_\ell(r) \}  \right) \leq (1 + o(1) )  \E \left( \prod_{i=1,\,2} \det \left( \un{\matM}_{\,i} (0, u_1, u_2) \right) \cdot \1 \{ \Eind_\ell(r) \} \right),
}
where $\Eind_\ell(r)$ is the event that both of the $\un{\matM}_{\,i}(r,u_1,u_2)$ have index $\ell$. For $i = 1,2$, the coupling introduced above implies $\un{\matM}_{\,i}(0,u_1, u_2) = \un{\matX}_{\iid,\,i}$, where $\un{\matX}_{\iid,\,i}$ denotes a shift of the matrix $\matX_{\iid,\,i}$. Thus the random matrices in the above display on the right are independent of one another, and it is only the index constraint on the $\un{\matM}_{\,i}(r,u_1, u_2)$ which prevents the expectation from factoring. 

An index transfer lemma based on Proposition~\ref{prop:index_transfer} is necessary to correct the indicator function on the right, getting the bound
\eq{
\E \left( \prod_{i=1,\,2}  \det \left( \un{\matM}_{\,i}(r,u_1,u_2)\right) \cdot \1\{ \Eind_\ell(r) \}  \right) &\leq (1 + o(1) ) \E \left( \prod_{i=1,\,2} \det  \left( \un{\matM}_{\,i} (0, u_1, u_2) \right) \cdot \1 \{ \Eind_\ell(0) \} \right) \\
&= (1 + o(1))  \prod_{i=1,\,2} \E \left( \det ( \un{\matX}_{\iid,\,i}  ) \cdot \1\{ \ind ( \un{\matX}_{\iid,\,i} ) = \ell \} \right),
}
and what remains is to analyze the factors $\E ( \det ( \un{\matX}_{\iid,\,i} ) \cdot \1 \{ \ind ( \un{\matX}_{\iid,\,i} ) = \ell \} )$ at the $O(1)$ scale rather than the exponential scale. This requires us to use concentration results, as in \cite{subag2017complexity}, in place of large deviations, but this step becomes even more delicate with an index constraint: on the event that the $\un{\matX}_{\iid,\,i}$ have negative eigenvalues, there is no way to use Subag's concentration results directly -- for instance, he works within the event that the spectrum of the shifted GOE matrix $\un{\matX}_{\iid,\,i}$ is bounded away from zero by some $\eps > 0$, which specializes to the case $\ell = 0$. 

This is the reason we perform the decoupling first: with only one matrix determinant in each expectation, we have access to the law governing the eigenvalues as an explicit density. Along the lines of \cite{ABC}, the index constraint can be combined with the explicit density, leading to a reformulation of $\E ( \det ( \un{\matX}_{\iid,\,i} ) \cdot \1 \{ \ind ( \un{\matX}_{\iid,\,i} ) = \ell \} )$ in terms of the expectation of a slightly smaller GOE matrix, all of whose eigenvalues are positive. The concentration results used at the end of \cite{subag2017complexity} apply to this smaller matrix and allow us to carry the analysis through to a proof of Theorem~\ref{thm:Jesus}.

\subsection{Additional inputs} We record the main theorem of \cite{benarous1997large}, a large deviation principle for the empirical spectral measure of GOE matrices. Let $\lambda_1 < \dots < \lambda_N$ denote the eigenvalues of $G \sim \goe_N$, and denote empirical spectral measure of $G$ by 
\begin{align}
L_N = \frac{1}{N} \sum_{i=1}^N \delta_{\lambda_i} \,. 
\end{align}

 Let $M_1(\R)$ be the space of Borel probability measures on $\R$ endowed with the weak topology. Let $\mathcal{L}$ denote the collection of Lipschitz functions $f: \R \to \R$ which are uniformly bounded by one, and which have Lipschitz constant at most one. Equip $M_1(\R)$ with the metric
\begin{align}
d_L(\mu,\mu') = \sup_{f \in \mathcal{L}} \left| \int_\R fd\mu - \int_R f d\mu' \right| \,,
\label{eq:lip_metric}
\end{align}
which metrizes the weak topology. We state the LDP in the form given in the appendix of \cite{subag2017complexity}.

\begin{thm}[\cite{benarous1997large}, Theorem 2.1.1] There is a good rate function $J(\mu)$, for which $J(\mu) = 0$ if and only if $\mu = \semi$, and such that the empirical measure $L_N$ satisfies the LDP on $M_1(\R)$ with speed $N^2$ and rate function $J(\mu)$.
\label{thm:spectral_LDP}
\end{thm}

\bibliographystyle{abbrv}
\bibliography{complexity}
%\nocite{*}

%%% FIN
\end{document}